\documentclass[11pt]{article}

\usepackage[utf8]{inputenc}
\usepackage{amsthm,amsmath,amsfonts,latexsym,url}
\usepackage{amssymb}
\usepackage{comment}
\usepackage{mathrsfs}
\usepackage[colorlinks=true,linkcolor=blue,citecolor=blue,urlcolor=magenta,pdfpagelabels=false]{hyperref}
\usepackage{xcolor}
\usepackage{enumitem}
\usepackage{tikz}
\usepackage[margin=3cm]{geometry}
\usepackage{charter}
\usepackage{dsfont}
\usepackage[all]{xy}
\usepackage{microtype}
\usepackage{mathtools}

\usepackage[capitalize]{cleveref}
\usepackage{todonotes}
\usepackage[english]{babel}
\usepackage{bigints}
\newcommand{\ud}{\mathrm{d}}

\usepackage{scalerel,stackengine}
\stackMath
\newcommand\reallywidehat[1]{%
\savestack{\tmpbox}{\stretchto{%
  \scaleto{%
    \scalerel*[\widthof{\ensuremath{#1}}]{\kern-.6pt\bigwedge\kern-.6pt}%
    {\rule[-\textheight/2]{1ex}{\textheight}}
  }{\textheight}%
}{0.5ex}}%
\stackon[1pt]{#1}{\tmpbox}%
}

\definecolor{vert}{HTML}{008000}

\newtheorem{theo}{Theorem}[section]

\newtheorem{prop}[theo]{Proposition}

\newtheorem{lem}[theo]{Lemma}
\newtheorem{coro}[theo]{Corollary}
\newtheorem{conj}[theo]{Conjecture}
\newtheorem{hyp}{Hypothesis}
\theoremstyle{definition}
\newtheorem{defi}[theo]{Definition}
\newtheorem{example}[theo]{Example}
\newtheorem{ex}[theo]{Example}

\newtheorem{rem}[theo]{Remark}
\newtheorem{remark}[theo]{Remark}
\newtheorem{question}[theo]{Question}

\newcommand{\N}{\mathbb N}
\newcommand{\Z}{\mathbb Z}
\newcommand{\Q}{\mathbb{Q}}
\newcommand{\Qbar}{\overline{\mathbb{Q}}}
\newcommand{\R}{\mathbb R}
\newcommand{\C}{\mathbb C}
\newcommand{\F}{\mathbb F}
\newcommand{\Fp}{\F_p}

\renewcommand{\P}{\mathbb P}

\newcommand{\calA}{\mathcal A}

\newcommand{\calM}{\mathcal M}

\newcommand{\id}{\text{\rm id}}

\newcommand{\op}[1]{\mathscr{#1}}

\renewcommand{\mod}{\,\text{\rm mod}\,}
\renewcommand{\hyp}[1]{\mathscr{H}\!\left(#1\right)}

\newcommand{\DP}{\text{\rm dp}}
\newcommand{\softO}{O^{\sim}}

\renewcommand{\hyp}[1]{\mathscr{Hyp}(#1)}

\newcommand{\Diag}{\operatorname{Diag}}

\newcommand{\ord}{\text{\rm ord}}

\newcommand{\univ}[1]{\mathscr{U}}

\def\fS{\mathscr{S}}

\def\pFqnoargs#1#2{{}_#1F_#2}
\def\pFq#1#2#3#4#5{
  \mathchoice
      {\pFqnoargs{#1}{#2}\biggl(\begin{matrix}{\def,{\kern.707em}#3}\\{\def,{\kern.707em}#4}\end{matrix}\,\bigg|\,#5\biggr)} 
      {\pFqnoargs{#1}{#2}(#3;#4;#5)} 
      {\pFqnoargs{#1}{#2}(#3;#4;#5)} 
      {\pFqnoargs{#1}{#2}(#3;#4;#5)} 
}
\def\twoFone#1#2#3#4{\pFq21{#1,#2}{#3}{#4}}

\def\twoFone#1#2#3#4{{_2F_1}\biggl(\begin{matrix}
  {#1}\kern.707em {#2}\\{#3}
\end{matrix}\,\bigg|\,#4\biggr)}

\newcommand{\footremember}[2]{%
    \footnote{#2}
    \newcounter{#1}
    \setcounter{#1}{\value{footnote}}%
}

\title{Algebraic solutions of linear differential equations:\\an arithmetic approach}

\author{%
  \href{https://mathexp.eu/bostan/}{Alin Bostan}\footremember{1}{Inria, Université Paris-Saclay, 1 rue Honoré d'Estienne d'Orves, 91120 Palaiseau, France.}%
  \and \href{https://xavier.caruso.ovh}{Xavier Caruso}\footremember{trailer}{CNRS; Université de Bordeaux, IMB; Inria Bordeaux Sud-Ouest, CANARI, 351 cours de la Libération, 33405 Talence, France.}%
  \and \href{http://math.univ-lyon1.fr/~roques/}{Julien Roques}\footremember{alley}{Université de Lyon,
Universite Claude Bernard Lyon 1, CNRS UMR 5208,
Institut Camille Jordan,
43 Bd du 11 Novembre 1918,
69622 Villeurbanne, France.}%
  }

\date\today

\usepackage{graphicx}

\begin{document}

\maketitle

\begin{abstract}
Given a linear differential equation with coefficients in $\Q(x)$, 
an important question is to know whether its full space of solutions consists 
of algebraic functions, or at least if one of its specific solutions is 
algebraic.
After presenting motivating examples coming from various branches of mathematics,
we advertise in an elementary way a beautiful local-global arithmetic approach to these questions, initiated by Grothendieck in the late sixties.
This approach has deep ramifications and leads to the still unsolved Grothendieck-Katz $p$-curvature conjecture.
\end{abstract}

\begingroup
\renewcommand{\thefootnote}{}
\footnotetext{%
}%
\endgroup

\setcounter{tocdepth}{2}
\tableofcontents

\section{Context, motivation and basic examples} \label{sec:intro}
In this text we consider linear differential equations (LDEs) of order $r$
\begin{equation}
\label{eq:lineardiff}
a_r(x) y^{(r)}(x) + a_{r-1}(x) y^{(r-1)}(x) + \cdots + a_1(x) y'(x) + a_0(x) y(x) = 0,
\end{equation}
where the $a_i$'s are known rational functions in $\Q(x)$, with $a_r$ not
identically zero and $y(x)$ is an unknown ``function''. 
In many applications, the desired solution $y(x)$ is a formal power series with coefficients in $\Q$. Therefore, in what follows, when we write ``function'' we actually mean an element of $\Q[[x]]$ unless otherwise specified.
We will say that a function $y\in\Q[[x]]$ is \emph{differentially finite} (in short, \emph{D-finite}) if it satisfies a linear differential equation like~\eqref{eq:lineardiff}.

A function $y\in\Q[[x]]$ is called \emph{algebraic} if it is algebraic over $\Q(x)$, that is, if $y(x)$ satisfies a polynomial equation of the form $P(x,y(x))=0$, for some $P\in\Q[x,y]\setminus \{ 0 \}.$ 
Otherwise, $y(x)$ is called \emph{transcendental}.
The simplest algebraic functions are polynomials in~$\Q[x]$, closely followed by rational power series: these are rational functions in $\Q(x)$ that have no pole at $x=0$ and therefore admit a Taylor expansion around the origin. 
A little more general are $N$-th roots of rational power series, such as $y(x)=1/\sqrt[N]{1-x}$.
In all these three cases, $y(x)$ is clearly D-finite and satisfies a linear differential equation of order $r=1$.

Many other examples of interesting functions that might or might not 
be solutions of linear differential equations arise from combinatorics.
A basic example is given by the Catalan numbers.

\begin{example}[Catalan numbers]\label{ex:Catalan}
By definition, a \emph{Dyck path} is a path drawn in the quarter
plane $\N^2$ that starts at $(0,0)$, consists of steps $\nearrow$ 
(directed by the vector $(1,1)$) or $\searrow$ (directed by the
vector $(1, -1)$) and finally ends on the $x$-axis
(see Figure~\ref{fig:Dyck}).

Let $C_n$ be the number of Dyck paths ending at $(2n,0)$; we say
that such paths have length~$n$.
For instance $C_1=1$ since there is a single Dyck path ending at $(2,0)$, namely
$\nearrow$--$\searrow$, while $C_2=2$ since there are two Dyck paths ending at $(4,0)$, namely
$\nearrow$--$\nearrow$--$\searrow$--$\searrow$ and $\nearrow$--$\searrow$--$\nearrow$--$\searrow$.
We use the convention that $C_0=1$.
We notice that any Dyck path of length $n+1$ can be written uniquely 
as the concatenation of (1)~a step $\nearrow$, (2)~a Dyck path of 
length $k{-}1$ (translated by $(1, 1)$), 
(3)~a step $\searrow$ and (4)~a Dyck path of length
$n{-}k$. It follows that the sequence $(C_n)_{n \geq 0}$ satisfies 
the following nonlinear recurrence relation:
\[C_n = \sum_{k=1}^{n} C_{k-1} C_{n-k},
\qquad \text{for all } n \geq 1.\]
If $y(x)$ denotes the generating function of the $C_n$'s, \emph{i.e.}
$y(x) = \sum_{n=0}^\infty C_n x^n$, the previous relation translates
to the algebraic identity 
\begin{equation}
\label{eq:algCatalan}
y(x) = 1 + x{\cdot}y(x)^2
\end{equation}
(the summand $1$ comes from the fact that $C_0 = 1$).
Therefore $y(x)$ is algebraic and one can even solve equation~\eqref{eq:algCatalan} and
get the closed formula $y(x) = \frac{1-\sqrt{1 - 4x}}{2x}$.
It is worth noting that, starting from the algebraic relation 
\eqref{eq:algCatalan}, one can also derive a linear differential
equation satisfied by $y(x)$.
Indeed differentiating~\eqref{eq:algCatalan}, one gets
$y'(x) = y(x)^2 + 2 x\: y(x)\:y'(x)$. Therefore:
\[y'(x) = \frac{y(x)^2}{1 - 2x\:y(x)}.\]
The right hand side in the latter expression can be further simplified
using again equation~\eqref{eq:algCatalan}. Indeed notice that
$$\big(1 - 2x\:y(x)\big)^2 = 1 - 4x \:y(x) + 4x^2 \: y(x)^2 = 1 - 4x$$
and consequently
\[\frac{y(x)^2}{1 - 2x\:y(x)} =
 \frac{y(x)^2 \cdot \big(1 - 2x\: y(x)\big)}{1 - 4x} =
 \frac{\big(y(x) - 1\big) \big(1 - 2x\: y(x)\big)}{x(1 - 4x)} =
 \frac{2x\:y(x) - y(x) + 1}{x(1 - 4x)}\]
after replacing two times $y(x)^2$ by $\frac{y(x) - 1}x$.

Finally, one obtains the inhomogeneous differential equation
\[(4x^2-x)y'(x) + (2x-1)y(x) + 1 = 0.\]
From this, we can derive new interesting information about the 
sequence $(C_n)_{n \geq 0}$. For instance, it easily implies the
simpler recurrence relation $C_n = \frac{4n-2}{n+1} \cdot C_{n-1}$ 
for all $n \geq 1$, from which we further derive the closed formula 
$C_n = \frac 1{n+1} \binom{2n}n$. Using Stirling's formula, we also 
deduce the asymptotic estimate $C_n \sim {4^n}/{\sqrt{\pi n^3}}$.
\end{example}

\begin{figure}
\hfill%
\begin{tikzpicture}[scale=0.8]
\begin{scope}[thin,black!20]
\draw (0,1)--(10.5,1);
\draw (0,2)--(10.5,2);
\draw (0,3)--(10.5,3);
\draw (1,0)--(1,3.5);
\draw (2,0)--(2,3.5);
\draw (3,0)--(3,3.5);
\draw (4,0)--(4,3.5);
\draw (5,0)--(5,3.5);
\draw (6,0)--(6,3.5);
\draw (7,0)--(7,3.5);
\draw (8,0)--(8,3.5);
\draw (9,0)--(9,3.5);
\draw (10,0)--(10,3.5);
\end{scope}
\draw[-latex] (0,0)--(10.5,0);
\draw[-latex] (0,0)--(0,3.5);
\draw[very thick] (0,0)--(2,2)--(3,1)--(5,3)--(8,0)--(9,1)--(10,0);
\fill (0,0) circle (0.7mm);
\fill (10,0) circle (0.7mm);
\node[scale=0.8, below left] at (0,0) { $0$ };
\node[scale=0.8, below] at (8,0) { $2k$ };
\node[scale=0.8, below] at (10,0) { \vphantom{$k$}$2n$ };
\end{tikzpicture}%
\hfill\null
\caption{A Dyck path}
\label{fig:Dyck}
\end{figure}

The previous example shows that being able to write down an equation
(either algebraic or differential) for a generating series can help a
lot in studying its coefficients.
Of course, obtaining explicit closed formulas (as we did for
the Catalan numbers) will not be possible in general; however, meaningful
information (such as the asymptotic growth of the coefficients) can be 
often extracted from the equation.
Besides, in many cases it turns out that the algebraicity of a 
generating series is the mirror of a (sometimes hidden) ``algebraic''
structure on the combinatorial side which often takes the form of
a recursive tree structure: in Example~\ref{ex:Catalan}, for 
instance, a Dyck path can be decomposed as a concatenation of smaller 
Dyck paths which themselves can be decomposed similarly, \emph{etc.}
We refer to~\cite{Bousquet06} for a much more detailed discussion
on this topic (including much more examples).

In Example~\ref{ex:Catalan}, we transformed an algebraic equation
into a differential equation. It is actually a general fact and an
old result, already known by Abel, that \emph{any algebraic function
is D-finite}. Precisely, if $y(x)$ satisfies an algebraic equation $P(x,y(x))=0$ with $P$ of
degree~$n$ in~$y$, then $y(x)$ also satisfies a differential equation like~\eqref{eq:lineardiff}
of order $r$ bounded from above by~$n$. This follows easily from the following reasoning. 
By differentiating $P(x,y(x))=0$ with respect to~$x$ and by using the chain rule, we obtain the equality
\[P_x(x,y(x)) + y'(x) P_y(x,y(x))=0.\]
Here and in what follows we denote by $P_x$ the derivative
$\partial P / \partial x$ of $P$ with respect to~$x$. Therefore, if $P$ is assumed to be a
polynomial of minimal degree in~$y$ satisfied by $y(x)$, then $P_y(x,y(x))$ is a nonzero
function in $\Q[[x]]$, and hence $y'(x) = -P_x(x,y(x))/P_y(x,y(x))$ is a rational function in
$y(x)$. By using again the equation $P(x,y(x))=0$, it is easy to see that any
rational function in~$y(x)$ can be re-written as a polynomial of degree at most~$n-1$ in~$y(x)$.
In other terms, the derivative $y'(x)$ lives in the $\Q(x)$-vector space generated by $1,y(x),
\ldots, y(x)^{n-1}$. The same similarly holds for all derivatives $y(x), y'(x), y''(x), \ldots,
y^{(n)}(x)$, and hence these elements must satisfy a nontrivial linear relation over~$\Q(x)$;
any such relation yields a linear differential equation~\eqref{eq:lineardiff} of order at most~$n$.
Observe that the same reasoning also proves the existence of an inhomogeneous linear
differential equation of order at most~$n-1$ for $y(x)$.

A naive though very natural question is whether the converse of Abel's result holds: \emph{is every D-finite function algebraic?}
The answer is negative, already for differential equations of order $r=1$, as the following example shows.

\begin{example}\label{ex:exp}
The function $\exp(x) \coloneqq  \sum_{n\geq 0} x^n/n!$, solution of $y'=y$, is transcendental. Here is
a purely algebraic proof.
Let us assume by contradiction that $\exp(x)$ satisfies a polynomial equation, of minimal 
degree~$d\geq 1$,
of the form $e^{dx} + \sum_{k=0}^{d-1} r_k(x) e^{kx} =0$ for some rational functions
$r_k(x) \in\Q(x)$.
By differentiating this equality with respect to~$x$ and using
$\exp'(x)=\exp(x)$, we get a new degree-$d$ equation $d e^{dx} + \sum_{k=0}^{d-1} (r_k'(x) + k
r_k(x))e^{kx}=0$ which by minimality is equal the former up to a factor~$d$. In other words,
$r_k'(x) + k r_k(x) = d r_k(x)$ for all~$k<d$. In particular, $r_0'(x)=d r_0(x)$, which implies that 
$r_0=0$. 
(Indeed, if $r_0(x)=A/B$ for two coprime polynomials $A,B\in\Q[x]$ with
$A'B-AB'=dAB$, then $B$ divides $AB'$, hence $B$ divides $B'$ and $B' = 0$.
Thus $A'=dA$, which implies $A=0$ and $r_0(x)=0$.)
The nullity of $r_0$ now implies that $\exp(x)$ satisfies a polynomial equation of degree $d-1$, which contradicts the minimality of~$d$.
\end{example}	

The reader could object that in \cref{ex:exp} we were probably lucky, because the differential equation of $\exp(x)$ is so simple, being of order 1 with constant coefficients.
Indeed, in the particular case of the exponential function, there are many other \emph{ad-hoc}
transcendence proofs, based on various branches of mathematics. For instance, a direct analytic
argument is that, viewed as a complex analytic function, any non-polynomial algebraic function needs to have a
finite (and positive) radius of convergence, while $\exp(x)$ is entire (that is, analytic in the whole complex plane).
Another proof is that $\exp(x)$ cannot satisfy a nontrivial algebraic equation, since otherwise
by specializing that equation at $x=1$ we would obtain that the number $e=\exp(1)$ is an
algebraic number, a statement known be to false since Hermite (1873). One could qualify this
last proof as ``cheating'', since it is intuitively clear that proving transcendence of functions
should be easier than proving transcendence of numbers.

\smallskip 

A systematic and very useful analytic way to establish functional 
transcendence is \emph{Flajolet's criterion}~\cite[Criterion~D]{Flajolet87} 
(see also~\cite[Thm.~VII.8]{FlSe09} and \cite[\S2.3]{Melczer21}). It is 
a consequence of the classical Newton-Puiseux theorem on fractional 
(Puiseux) series expansion of algebraic functions and on Darboux's 
transfer results from the local behavior of $f(x)=\sum_{n \geq 0} a_n 
x^n$ around its singularities to the asymptotic behavior of its 
coefficient sequence $(a_n)_{n \geq 0}$.
Before stating it, we recall the definition of the gamma function
\[\Gamma(s) \coloneqq \int_0^{+\infty} t^{s-1} e^t dt,\]
where the variable $s$ is a complex number with positive real part.
The gamma function interpolates the factorial in the sense that
$\Gamma(n) = (n{-}1)!$ for any positive integer $n$.

\begin{prop}[{Theorems A and D} in \cite{Flajolet87}] \label{TheoremD}
Let $f(x)=\sum_{n \geq 0} a_n x^n \in \Q[[x]] \setminus \Q[x]$ be an algebraic non-polynomial function.
Then $f(x)$ has a finite number of singularities, a finite nonzero radius of convergence,
and
its coefficient sequence $(a_n)_{n\geq 0}$
is such that
\begin{equation} \label{eq:asy}
a_n = \frac{\beta^n n^r}{\Gamma(r+1)} \sum_{i=0}^m C_i \omega_i^n + O(\beta^n n^q),
\end{equation}
where $m\in \Z_{\geq 0}$, $r\in\mathbb{Q}\setminus\Z_{<0}$, $q<r$, 
$\beta \in \overline{\Q}_{>0}$, and $C_i, \omega_i \in \overline{\Q} 
\setminus \{ 0 \}$ with $|\omega_i|=1$.
\end{prop}

The most useful form of the criterion is the following:
\begin{coro}[``Flajolet's criterion''] \label{crit:Flajolet}
If $f(x)=\sum_{n \geq 0} a_n x^n \in \Q[[x]]$ 
and
$a_n \sim \gamma \,  \beta^n \, n^r$ with
either 
$r\not\in\mathbb{Q}\setminus\Z_{<0}$,
or
$\beta \notin \overline{\Q}_{>0}$,
or
$\gamma \cdot \Gamma(r+1) \notin \overline{\Q}$,
then $f(x)$ is transcendental.
\end{coro}

As an example of application, \cref{TheoremD} immediately implies that 
$\exp(x)=\sum_{n\geq 0} x^n/n!$ is transcendental, since the sequence 
$1/n!$ is not of the form~\eqref{eq:asy} (or, since the radius of 
convergence of $\exp(x)$ is infinite).

\medskip At this point, we can ask ourselves: \emph{is there a 
purely arithmetic proof of the transcendence of $\exp(x)$?}
This question can be seen as the starting point of the present article, whose
main aim is precisely to advertise a very beautiful number-theoretic approach to algebraicity of solutions of linear differential equations.
More generally, we can raise the following question.

\begin{question}
Is there any number-theoretic way to recognize whether the differential equation~\eqref{eq:lineardiff} admits only algebraic solutions in its solution space?
\end{question}

Nicely enough, the answer to this question is positive, for two distinct but related reasons.
Let us first explain them a bit in the case of the exponential function $\exp(x) = \sum_{k \geq
0} x^k/k!$. 
The first arithmetic proof of the transcendence of $\exp(x)$ is based on 
the following result, which in rough terms asserts that the coefficients of 
algebraic functions are ``almost integral''.

\begin{prop}[``Eisenstein's criterion'' (1852)]\label{prop:eisenstein}
If the function $y(x) = \sum_{k \geq 0} a_k x^k \in \Q[[x]]$ is algebraic, 
then there exists $N\in\N\setminus \{ 0 \}$ such that $y(Nx) - y(0) \in \Z[[x]]$.
In particular, only a finite number of prime numbers can divide the denominators of the coefficients $a_k$.
\end{prop}

Since in the factorial sequence $(k!)_{k \geq 0}$ obviously all prime numbers
appear as divisors, \cref{prop:eisenstein} immediately implies that
$\exp(x)$ is transcendental.

To formulate the second arithmetic proof of the transcendence of $\exp(x)$, we will need a little bit of additional vocabulary. 
The differential equation~\eqref{eq:lineardiff} can be rewritten
in the compact form $\op L(y) = 0$, where $\op L$ is the linear differential operator
\begin{equation}
\label{eq:opL0}
\op L = 
a_r(x) {\cdot} \partial_x ^r+ a_{r-1}(x) {\cdot}
\partial_x^{r-1} + \cdots + a_1(x) {\cdot}\partial_x + a_0(x).
\end{equation}
We denote  by 
$\Q(x)\langle\partial_x\rangle$ the set of such linear differential operators. For convenience, we also allow the trivial operator, in which all coefficients $a_i(x)$ are zero.
The elements of $\Q(x)\langle\partial_x\rangle$ act on functions in $x$ by letting the 
variable $\partial_x$ act through the differentiation $\frac d {dx}$. 
The set $\Q(x)\langle\partial_x\rangle$ is then endowed with a 
structure  of \emph{noncommutative} ring where the addition is the usual
one but the multiplication is twisted according to the following
rule, reminiscent from Leibniz's differentiation rule:
\[ \forall r \in \Q(x), \quad \partial_x r(x) =  r(x) \partial_x + r'(x).\]
Although the ring $\Q(x)\langle\partial_x\rangle$ is noncommutative, it shares many
properties with the classical commutative ring of polynomials $\mathbb{Q}(x)[y]$. First,
one has a well-defined notion of degree: 
the \emph{degree} of the nonzero operator $\op L$ in~\eqref{eq:opL0}
is the order $r$ of the corresponding differential equation~\eqref{eq:lineardiff},
that is the largest integer~$r$ such that $a_r(x) \neq 0$. We will denote it by
$\ord(\op L)$ in what follows.
Second, the ring $\Q(x)\langle \partial_x \rangle$ admits an Euclidean division.
\begin{prop}
The ring $\Q(x)\langle \partial_x \rangle$ is right Euclidean,
\emph{i.e.}, for all $A, B \in \Q(x)\langle \partial_x \rangle$ with
$B \neq 0$, there exist $Q$ and $R$ in $\Q(x)\langle \partial_x \rangle$
such that $A = QB + R$ 
and $\ord R < \ord B$.
Moreover, the pair $(Q,R)$ is unique with these properties.
\end{prop}

Using these notions, we can now formulate a very basic but important arithmetic result.

\begin{prop}
	\label{prop:cartier}
If all solutions of~\eqref{eq:lineardiff} are algebraic functions, then for all but a finite number of prime numbers $p$, the remainder of the right Euclidean division of $\partial_x^p$ by $\op L$ has all its coefficients divisible by~$p$.
\end{prop}

\begin{proof}
By Eisenstein's criterion (\cref{prop:eisenstein}), there exists a basis of algebraic solutions $y_1,\ldots, y_r$ of $\op L$ and an integer $N\in\N\setminus \{ 0 \}$ such that $y_i(x)$ is in $\Q + x\Z[[\frac{x}{N}]]$ for all $i$. This implies that for all but a finite number of prime numbers $p$ (namely, the ones dividing $N$ and the denominators of $y_i(0)$) the power series $y_i^{(p)}(x)$ are in $p \Z[[x]]$, i.e. zero modulo~$p$. Thus, writing $\partial_x^p = Q \op L + R$ with $R$ of order at most $r-1$, we have $R(y_i) = 0 \bmod p$ for all $i$. Writing $R = c_0(x) + \cdots + c_{r-1}(x) \partial_x^{r-1}$, we deduce that the vector $(c_0, \ldots, c_{r-1})$ times the Wronskian matrix of the $y_i$'s is zero modulo~$p$. Since the Wronskian matrix is invertible (because the $y_i$'s form a basis of solutions of $\op L$), it is also invertible modulo~$p$ for all but a finite number of prime numbers~$p$, and thus
the $c_i$'s are all~$0$ modulo~$p$ for any such prime~$p$.
\end{proof}

\begin{example}
The generating function of the Catalan numbers, $y(x) = \sum_{k \geq 0} C_k x^k$,
satisfies the differential equation
$(4x^2-x) y''(x)+(10x-2) y'(x) + 2y(x)=0$, which is easily deduced, either from the inhomogeneous differential equation of order 1 in~\cref{ex:Catalan}, or directly from the recurrence  $(k+2) C_{k+1} - (4k+2)C_k= 0$.
The associated differential operator is $\op L = \left(4 x^{2}-x \right) \partial_x^{2}+\left(10 x -2\right) \partial_x +2$, and the remainders of the right Euclidean divisions of $\partial_x^p$ by $\op L$ for $p\in \{ 2, 3, 5 \}$ are
\begin{align*}
\partial_x^2 \bmod \op L 
& =  -\frac{2 \left(5 x -1\right)}{x \left(4 x -1\right)} \partial_x-\frac{2}{x \left(4 x -1\right)}, \\
\partial_x^3 \bmod \op L 
& = \frac{6 \left(22 x^{2}-9 x +1\right)}{x^{2} \left(4 x -1\right)^{2}} \partial_x
  +\frac{6(6 x -1)}{x^{2} \left(4 x -1\right)^{2}}, \\
\partial_x^5 \bmod \op L
& = 
  \frac{120 \left(386 x^{4}-325 x^{3}+110 x^{2}-17 x +1\right)}{x^{4} \left(4 x -1\right)^{4}}\partial_x
  +\frac{120(130 x^{3}-69 x^{2}+14 x -1)}{x^{4} \left(4 x -1\right)^{4}}.
\end{align*}
Note that indeed, we have $\partial_x^p \bmod \op L = 0$ modulo $p$, in the three cases.
\end{example}

\cref{prop:cartier} will be discussed in more detail in~\S\ref{ssec:groth-general}. 
For now, let us simply observe how it implies the transcendence of $\exp(x)$.
In this case $\op L = \partial_x - 1$. Hence $\partial_x^p \bmod \op L$ is equal to 1 for all~$p$. Indeed, in this case, $\op L$ and $\partial_x$ commute, hence the computation of the remainder is the same as the computation in $\Q[x]$ of the remainder of $x^p$ by $x-1$, that is the evaluation of $x^p$ at $x=1$. 

\begin{example}\label{ex:log}
Let us consider the logarithmic function $y(x) = \log(1-x)$. Of course, since $\log(\exp(x))=x$ and $\exp(\log(1-x))=1-x$, the transcendence of the logarithm function follows from that of the exponential function. However, the two arithmetic criteria can be used directly.
First, Eisenstein's criterion (\cref{prop:eisenstein}) can be applied since $\log(1-x)=-\sum_{k
\geq 1} x^k/k$. Second, we have that a full basis of solutions of $\op L = (1-x)\partial_x^2 - \partial_x$ is $\{ 1, \log(1-x) \}$, and it holds by induction that
\[
\partial_x^n \bmod \op L = \frac{(n-1)!}{(1-x)^{n-1}} \partial_x \qquad \text{for all} \; n \geq 1.
\]
Therefore, Wilson's theorem implies that modulo any prime number~$p$, the remainder $\partial_x^p \bmod \op L$ is equal to $-\frac{1}{(1-x)^{p-1}} \partial_x$, hence it is never 0. Then, \cref{prop:cartier} implies that $\log(1-x)$ is transcendental.
\end{example}	

\medskip

A natural question is whether the converses of \cref{prop:eisenstein} and \cref{prop:cartier} have
any chance to hold true. Concerning \cref{prop:eisenstein}, it is not difficult to exhibit a transcendental D-finite power series with integer coefficients, showing that its converse is false.
A natural example that comes to mind is
 $y(x) = \sum_{k\geq 0} k!x^k$, which 
satisfies $x^2 y''(x)+(3x-1)y'(x)+y(x)=0$. 
One can prove that $y(x)$ is not algebraic in various ways. One of
them is again analytic, by observing that $y(x)$ has radius of convergence~$0$ and by applying \cref{TheoremD}. 
A purely algebraic
proof also exists, but it is less immediate: it relies on the combination of the following facts: (i)
$y(x)$ does not satisfy any first-order differential equation, hence the second-order differential equation above is the minimal-order LDE satisfied by $y(x)$; (ii) this second-order
differential equation admits in its solution space the transcendental solution $e^{-1/x}/x$;
(iii) if $y(x)$ were algebraic, then its minimal-order LDE would have only algebraic solutions (Proposition 2.5 in \cite{Singer79}).
 
However, the reader may object that this counterexample is ``degenerate'' since the coefficient sequence $(k!)_{k\geq 0}$ grows too fast, which is not compatible with the growth of the coefficient sequence of an algebraic function. A better converse of \cref{prop:eisenstein} would be: \emph{is there any example of a D-finite but transcendental function $y\in\Z[[x]]$, whose coefficient sequence grows at most geometrically?} The answer is again positive.

\begin{example}
	Let again $(C_k)$ be the sequence of Catalan numbers, and consider the function 
	$y(x) = \sum_{k\geq 0} C_k \binom{2k}{k}x^k$. 
	One easily checks that it is D-finite and it satisfies the second-order equation
$x(16x-1)y''(x) + 2(16x-1)y'(x) + 4y(x)=0$. From there, it follows that
$y(x)$ is the Gauss hypergeometric function
${}_2 F_1([1/2, 1/2], [2]; 16x)$
and classical results (that we shall recall in \S \ref{sec:hypergeom})
imply that $y(x)$ is transcendental.
\end{example}

Thus, even the stronger converse of \cref{prop:eisenstein} appears to be false. 
One may wonder if there is any way to reinforce even further the conclusion of \cref{prop:eisenstein}, such that its converse becomes true. As of today, this is still an open problem, although there exist conjectural statements in this spirit. One of them is the following:

\begin{conj}[Christol-André conjecture]\label{conj:ChristolAndre}
Assume that  $y(x) = \sum_{k \geq 0} a_k x^k \in \Q[[x]]$ is D-finite.
Let $\op L$ be the minimal-order monic linear differential equation satisfied by~$y(x)$. We assume that:
\begin{itemize} 
\item[$(1)$] the sequence $(a_k)_{k \geq 0}$ has at most geometric growth;
\item[$(2)$] there exists $N\in\N$ such that $y(Nx) - y(0) \in \Z[[x]]$;
\item[$(3)$] the point $x=0$ is not a pole of any of the coefficients of $\op L$.
\end{itemize}
Then, $y(x)$ is algebraic.
\end{conj}

\noindent
If $r$ denotes the order of~$\op L$, it is also conjectured that condition (3) can be replaced by
\begin{itemize} 
\item[$(3b)$] $\op L$ has $r$ linearly independent solutions in $\C((x^{1/D}))$ for some positive integer $D$.
\end{itemize}

Now, what about the converse of \cref{prop:cartier}? It turns out that this converse is one of the simplest formulations of what is usually called the \emph{Grothendieck conjecture}. This conjecture has been formulated in the late 1960s and it has a rich history. It was proved for some important classes of differential equations~\eqref{eq:lineardiff}, which will be discussed in \cref{sec:Grothendieck}.

\begin{conj}[Grothendieck's conjecture, version 1]\label{conj:Grothendieck1} 
Let $\op L \in \Q(x)\langle\partial_x\rangle$ be the differential operator attached to~\eqref{eq:lineardiff}.
If for all but a finite number of prime numbers $p$, the remainder of the right Euclidean division of $\partial_x^p$ by $\op L$ has all its coefficients divisible by~$p$, then 
all solutions of~\eqref{eq:lineardiff} are algebraic functions over $\Q(x)$.
\end{conj}

\begin{example}\label{ex:L2r}
Consider the operator
\[\op L = 2 x \left(x-1 \right) \partial_x^{2}+\left(4 x -1\right) \partial_x + 1.
\]
Then, for any prime number $p>2$, the reduction modulo $p$ of $\partial_x^p \bmod \op L$ is equal to 0 if $p\equiv 1 \bmod 4$; else, it is equal to
\[
-\frac{2}{(x(x-1))^{\frac{p-1}{2}}} \partial_x - \frac{1}{(x-1)^{\frac{p+1}{2}} x^{\frac{p-1}{2}}} .
\]
Therefore, for half of the primes $p$, the remainder is nonzero, hence \cref{prop:cartier} implies that $\op L$ does not admit only algebraic solutions.
\end{example}

\begin{example}
Consider now the operator
\[\op L_r = 2 x \left(x-1 \right) \partial_x^{2}+\left(4 x -1\right) \partial_x +2 r(1-r), \quad \text{with} \; r\in\Q\setminus \{ 1/2 \},
\]
which is a tiny modification of the operator in \cref{ex:L2r}: only the constant term has changed.
Then, 
for any prime $p \neq 2$ not dividing the denominator of $r$,
the reduction modulo~$p$ of $\partial_x^p \bmod \op L_r$ is equal to
\[
\frac{r(r+1) \cdots \reallywidehat{(r + \frac{p-1}{2})} \cdots (r+p-1)}{(x(x-1))^{\frac{p-1}{2}}} \partial_x + \frac{r(r+1) \cdots \reallywidehat{(r + \frac{p-1}{2})} \cdots (r+p-1)}{2(x-1)^{\frac{p+1}{2}} x^{\frac{p-1}{2}}} ,
\]
where the notation $\reallywidehat{(r + \frac{p-1}{2})}$ indicates that $r + \frac{p-1}{2}$ is missing in the above products.
For any prime number $p>2$ not dividing the denominator of $r$ nor 
the numerator of $2r-1$, the previous remainder is zero modulo~$p$.
Therefore, for all but finitely many primes $p$, the remainder is zero modulo~$p$.

Can we conclude that the operator $\op L_r$ admits only algebraic solutions?
\cref{conj:Grothendieck1} predicts that the answer is positive and,
indeed, one easily checks that $\op L_r$ admits
the algebraic solutions
$f_1 = {\left(\sqrt{x}+\sqrt{x -1}\right)^{2 r-1}}/{\sqrt{x -1}}$
and
$f_2 = {\left(\sqrt{x}+\sqrt{x -1}\right)^{1-2 r}}/{\sqrt{x -1}}$,
which are linearly independent for $r\neq \frac12$ since
$f_1 f_2' - f_2 f_1' = (1-2 r)/(\sqrt{x}\, \left(x -1\right)^{{3}/{2}})$.
\end{example}

\section{Several natural differential equations have algebraic solutions}
\label{sec:examples}

As we have just seen in \cref{sec:intro}, 
although in general the solutions $y(x)$ of \eqref{eq:lineardiff} are 
transcendental functions (e.g.,  $y(x) = \exp(x)$ and $y(x) = \log(1-x)$), it may happen 
sometimes that they are algebraic.
Most of the examples given in \cref{sec:intro} were ``academic examples'', in the sense that they were simple and constructed to illustrate the exposition. 

The aim of this section is to give evidences showing that linear 
differential equations naturally appear in many areas of mathematics, 
including algebra, combinatorics and number theory. In all these settings, we 
shall see examples of differential equations that admit algebraic 
solutions and others that do not, demonstrating that both behaviors 
are common and deserve our attention. 
Hence, besides its beauty and intrinsic fundamental nature, being 
able to recognize whether a given D-finite function is algebraic or 
not appears as an important question, likely to have profound
applications. 

For instance, in number theory one often wants to understand the 
(algebraic or transcendental) nature of a complex number given as the 
value $f(\alpha)$ of a function $f\in \Q[[x]]$ at an algebraic point 
$\alpha\in\overline{\Q}$; understanding the algebraic nature of the 
corresponding function $f(x)$ is a first important step in this process. 
In the theory of the \emph{E-functions} (an 
important class of D-finite functions generalizing the exponential 
function), the methods of Siegel and Shidlovskii, revisited recently 
in~\cite{AdRi18} and~\cite{BoRiSa22}, are used to show that a number 
(such as $e=\exp(1)$) is transcendental by checking, among other things, 
that it is the value of a transcendental $E$-function. Another 
motivation comes from combinatorics, where predicting the nature of the 
generating function of some class of combinatorial objects is a central 
task, as it may reveal strong underlying structures. Indeed, knowing 
that a class of objects is counted by an algebraic function suggests 
that it should be possible to construct these objects recursively by 
concatenation of objects of the same type~\cite{Bousquet06}.

\subsection{Examples from Special functions: Hypergeometric functions}

\subsubsection{Elliptic integrals: Euler's differential equation} \label{sec:perimeter}

Perhaps one of the oldest special functions distinct from the classical algebraic, exponential, logarithmic, or trigonometric functions, is the one arising from the question: \emph{what is the perimeter $p(x)$ of an ellipse with semi-major axis $1$, 
as a function of its eccentricity~$x$?} 
(Recall that the eccentricity is the quotient between the focal distance and the semi-major axis.)

This question, which was solved by Euler~\cite[\S7]{Euler1733} in 1733, is more challenging than the analogue one with ``perimeter'' replaced by ``area'', since the area is expressible algebraically as $\pi \sqrt{1-x^2}$.
 First, we express the arc length using a real integral and the parametrization of the ellipse:
\begin{equation} \label{eq:ellipse}
	 p(x) = 4 {\bigintsss_{0}^1\sqrt{\frac{1-x^2 u^2}{1-u^2}} \, \ud u} = 
2 \pi  -   \frac{\pi}{2} x^2  -  \frac{3\pi}{32} x^4  
- \frac{5\pi}{128} x^6
-\frac{175\pi}{8192} x^8 - \cdots.
\end{equation}
Up to the factor of 4, the function $p(x)$ is called the \emph{complete elliptic integral of the second kind}.
The second equality above is obtained by expanding the integrand in power series with respect to $x$, and integrating between 0 and 1. 
It is a fact familiar to algebraic geometers that this function $p(x)$ satisfies a linear differential equation: $p(x)$ is what we call a \emph{period function} and, as such, it satisfies a linear differential equation called Picard-Fuchs differential equation (or Gauss-Manin connection).
An alternative way to establish this fact and to find a differential equation satisfied by $p(x)$ is to use the ``method of creative telescoping''~\cite{AlZe90}.
The ``magic'' of creative telescoping is that it constructs the equality below, which expresses a linear combination of the integrand and of its first and second derivative with respect to~$x$ as a pure derivative with respect to $u$ of another algebraic function (a rational multiple of the integrand):
  \begin{equation} \label{eq:telescopic_ellipse}
    \left( {(x-x^3)\partial_x^2+(1-x^2)\partial_x + x}\right)
\left(\sqrt{\frac{1-x^2 u^2}{1-u^2}}\right) = \partial_u\left( \frac{x u \sqrt{1-u^{2}}}{\sqrt{1-x^{2} u^{2}}} \right) .
  \end{equation}
Now, integrating both sides of \cref{eq:telescopic_ellipse} with respect to $u$, it follows that $p(x)$ is a D-finite function with respect to $x$, and that it satisfies the linear differential equation
\begin{equation} \label{deq:Euler}
	(x - x^3)p''(x) + (1-x^2) p'(x) + x p(x) = 0.
\end{equation}
Writing $p(x) = \sum_{k \geq 0} a_k x^k$, we deduce from \cref{deq:Euler} the recursion
$(k - 1) (k + 1) a_k = (k + 2)^2  a_{k+ 2}$ for all $k \geq 0$.
From $a_0=2\pi$ and $a_1=0$ it follows that
\[
a_{2k} = \frac{2\pi \binom{2k}{k}^2}{(1-2k)16^k} \quad \text{and} \quad
a_{2k+1}=0 \quad \text{for all} \; k \geq 0.
\]
Stirling's formula then implies that $a_{2k} \sim -1/k^2$, and hence transcendence of $p(x)$.
Indeed, the presence of the factor $k^{-2}$ is incompatible with algebraicity by 
Flajolet's criterion (\cref{crit:Flajolet}).

\subsubsection{Elliptic integrals: Legendre's differential equation} \label{sec:Legendre}

A special function similar to $p(x)$ is obtained by a different construction.
Consider the family of elliptic curves (with $x\in\mathbb{C}$) given by the 
\emph{Legendre equation}
\[E_x : \qquad y^2 = u(u-1)(u-x).\]
On $E_x$, there exists a unique (up to a constant multiple) holomorphic 1-form, given by
\[
\omega_x = \frac{\ud u}{y} = \frac{\ud u}{\sqrt{u(u-1)(u-x)}} .
\]
This form is necessarily closed since it is a holomorphic 1-form on a variety of complex dimension 1.
The method of creative telescoping finds an exact form that is a linear combination of $\omega_x$ and of its first and second derivatives with respect to $x$, namely
\begin{equation}\label{eq:exactform}
    \left( {(4x^2-4x)\partial_x^2+(8x-4)\partial_x + 1}\right) \omega_x = 
- \ud \left( \frac{2 \sqrt{u(u - 1)}}{(u -x)^{3/2}}
 \right) .
\end{equation}
From \cref{eq:exactform} it follows that the integral $y(x)=\int_\gamma \omega_x$ 
over any closed curve $\gamma$ on $E_x$
satisfies the
\emph{Legendre differential equation}
\begin{equation} \label{deq:Legendre}
 (4x^2-4x) y''(x) + (8x-4) y'(x) + y(x) = 0.
\end{equation}
This is the most basic case of the \emph{Picard–Fuchs differential equation} of a period function.
For instance, by taking $C$ to be the curve on $E_x$ given by the double cover $y=\pm \sqrt{u(u-1)(u-x)}$ of $[1,\infty)$, the corresponding period is the 
\emph{complete elliptic integral of the first kind}, 
\[
\bigintssss_C \omega_x = 2 \bigintssss_1^\infty \frac{\ud u}{\sqrt{u(u-1)(u-x)}} = \sum_{k \geq 0}
b_k x^k,
\]
where $b_0 = 2 \bigintsss_1^\infty \frac{\ud u}{u\sqrt{u-1}} = 2\pi$ and 
$\left(2 k +1\right)^{2} b_k = 4 \left(k +1\right)^{2} b_{k +1}$ for all $k\geq 0$, this recurrence relation being a consequence of the fact that $y(x) = \int_C \omega_x$ satisfies \cref{deq:Legendre}. Thus,
\begin{equation} \label{eq:Legendre}
\bigintssss_C \omega_x = 2\pi \sum_{k \geq 0} \binom{2k}{k}^2 \left( \frac{x}{16} \right) ^k.
\end{equation}
Once again, Stirling's formula gives $b_{k} \sim 2/k$, which excludes algebraicity of $y(x) = \int_C \omega_x$ 
by Flajolet's criterion (\cref{crit:Flajolet}).

\subsubsection{Gauss' hypergeometric functions} 
\label{sec:hypergeom}

The D-finite functions considered in \cref{sec:perimeter,sec:Legendre} are special cases of the \emph{Gauss hypergeometric function} with parameters $a,b,c\in\mathbb{Q}$, $c\notin \Z_{\leq 0}$, defined by
\begin{equation} \label{def:2F1}
	{}_2 F_1 ([a,b], [c]; x) \coloneqq  \sum_{k \geq 0} \frac{(a)_k (b)_k}{(c)_k k!} x^k,
\end{equation}
where $(a)_k=a(a+1)\cdots(a+k-1)$ denotes the rising factorial.
Indeed, $p(x)$ in \cref{eq:ellipse} is equal to $2 \pi \cdot {}_2 F_1 ([-1/2,1/2], [1]; x^2)$, while $\int_C \omega_x$
in \cref{eq:Legendre} is equal to $2 \pi \cdot {}_2 F_1 ([1/2, 1/2], [1]; x)$. We have seen that in both cases these functions are transcendental.

In general, $y(x) = {}_2 F_1 ([a, b], [c]; x)$ satisfies the second-order differential equation
\begin{equation}\label{deq:2F1}
x(x - 1)  y''(x)  + ((a  + b + 1)x - c) y'(x) + a b y(x)=0
\end{equation}
and the name \emph{hypergeometric} comes from the fact that the coefficient sequence $(u_k)_{k \geq 0}$ of ${}_2 F_1 ([a, b], [c]; x) = \sum_{k \geq 0} u_k x^k$ 
satisfies a linear recurrence of order 1, namely
\[(a + k) (b + k) u_k = (k + 1) (k + c) u_{k + 1}, \qquad (k \geq 0).\]
For other choices of parameters we recover the functions
\begin{align*}
(1-x)^{\alpha} & =  {}_2 F_1 ([-\alpha, 1], [1]; x), \quad \text{for all} \; \alpha\in\Q,\\
\sum_{k \geq 0} C_k x^k & =  {}_2 F_1 ([1, 1/2], [2]; 4x), \\
\log(1-x)  & =  -x \cdot {}_2 F_1 ([1, 1], [2]; x), \\
\arcsin(x) & =  x \cdot {}_2 F_1 ([1/2, 1/2], [3/2]; x^2),
\end{align*}
the first two of which are algebraic, the last two of which are transcendental. 
In some cases, the Gauss' hypergeometric function even becomes a polynomial: this is so for
\[
P_n(x)  = {}_2 F_1 ([-n, n+1], [1]; (1-x)/2),
\]
the \emph{Legendre polynomial} given by $P_n(x) \coloneqq  \frac{1}{2^n \cdot n!} \cdot \frac{\partial^n}{{\partial x}^n} (x^2 -1)^n$, as well as for 
\[
T_n(x)  = (-1)^n \cdot {}_2 F_1 ([-n, n], [1/2]; (x+1)/2),
\]
the \emph{Chebyshev polynomial of the first kind} given by $T_n(\cos x) = \cos (nx)$.

Deciding the algebraicity of $_2F_1$ hypergeometric functions is an old problem, solved by Schwarz~\cite{Schwarz1873} using geometric tools (Riemann mappings, Schwarzian derivatives and sphere tilings by spherical triangles) and by Landau~\cite{Landau1904,Landau1911} and Errera~\cite{Errera1913} using arithmetic tools (Eisenstein's criterion for algebraic power series, and Dirichlet's theorem on prime numbers in arithmetic progressions).  
Both approaches are algorithmic:  Schwarz's criterion reduces the problem to a table look-up after some preprocessing on the parameters $a,b,c$; the Landau-Errera criterion amounts to checking a finite number of inequalities. 

More precisely, 
let us assume that none of $a$, $b$, $c-a$ and $c-b$ is an integer (equivalently, the operator 
$H(a,b;c) \coloneqq x(1-x)\partial_x^2 +
(c-(a  +  b  +  1)x)\partial_x- ab$ is irreducible)
and let $D$ be the common denominator of $a, b$ and~$c$.
Then, the Landau-Errera criterion says that the following assertions are equivalent:
\begin{enumerate}
\item the hypergeometric function ${}_2 F_1 ([a, b], [c]; x)$ is algebraic;	
\item the operator $H(a,b;c)$ admits only algebraic solutions;
\item for every~$\ell < D$ coprime with~$D$,
either $\{ \ell a \} < \{ \ell c \} < \{ \ell b \}$ or $\{ \ell b \} < \{ \ell c \} < \{ \ell a \}.$
(Here $\{ x \}$ denotes the fractional part $x - \lfloor x \rfloor$ of~$x$.)
\end{enumerate}
The last condition is equivalent to the fact that, 
for every~$\ell< D$ coprime with~$D$,
the two sets 
$\{ e^{2\pi i \ell a}, e^{2\pi i \ell b} \}$
and
$\{ e^{2\pi i \ell c}, 1 \}$
are interlaced on the unit circle. 
This ``interlacing condition'' was first proved by Landau~\cite{Landau1904,Landau1911} to be necessary for the algebraicity of ${}_2 F_1 ([a, b], [c]; x)$ and then proved to also be sufficient by Errera~\cite{Errera1913}; see also Stridsberg's intermediate contribution~\cite{Stridsberg1910}, which relates the 
conditions in Eisenstein's criterion to the ones in the Landau-Errera condition. 
In \S\ref{ssec:progressGK} we will see that \cref{theo:beukers and heckmann} provides an extension of the Landau-Errera ``interlacing criterion'' to the generalized hypergeometric function $_{s+1}F_s$ defined by
\begin{equation} \label{deq:sFs}
	{}_{s+1} F_s ([a_1, \ldots, a_{s+1}], [b_1, \ldots, b_s]; x) = \sum_{k \geq 0} \frac{(a_1)_k \cdots (a_{s+1})_k}{(b_1)_k \cdots (b_{s})_k k!} x^k.
\end{equation}

\subsection{Examples from Algebra: Diagonals}

As proved in~\cref{sec:intro},
algebraic functions are D-finite. A larger, yet very important, class of D-finite functions is formed by the \emph{diagonals of rational functions}. 
By definition, the \emph{diagonal} of a multivariate power series 
	\[ F = \sum_{{(i_1,\dots,i_n)\in\N^n}}a_{i_1,\dotsc,i_n} x_1^{i_1}\dotsm x_n^{i_n} \; \in\Q[[x_1, \ldots, x_n]]\]
is the univariate power series
 \[ \Diag(F) = \sum_{i \in\N} a_{i,\dotsc,i}t^i \; \in \Q[[t]] .\]

\begin{example} [Dyck bridges]\label{ex:Polya-Dyck}
Let $B_{n}$ be the number of $\{\uparrow, \rightarrow\}$-walks in $\mathbb{Z}^2$ 
from $(0,0)$ to $(n,n)$ (i.e., there are exactly $B_{n}$ ways of going from the origin to $(n,n)$ using only North and East steps, see \S\ref{ssec:walks} for a more general context)
and let $B(t)$ be its generating function $\sum_{n\geq 0} B_n t^n$.
Then,
\[ B(t) = \Diag\left( \frac{1}{1-x-y}\right) = \sum_{n \geq 0}\binom{2n}{n}t^n. \]
\end{example}	

This is perhaps the simplest example of a diagonal. By the binomial theorem, it comes that 
$B(t) = {1}/{\sqrt{1-4t}}$, hence this diagonal is even an algebraic function. This is not an accident. Indeed, a century ago Pólya~\cite{Polya22}
proved that diagonals of bivariate rational functions are algebraic. 
(Later, Furstenberg~\cite{Furstenberg67} showed that the converse also holds true.)
Pólya's result can be proved as follows. First, using the simple observation\footnote{Here, and in all the text, $[x^n]$ denotes coefficient extraction of~$x^n$.}
$\textsf{Diag} (F)(t) = [x^0] \, F(x,t/x)$, the diagonal of the rational
function $F(x,y)\in \Q(x,y)$ is encoded as a complex integral using Cauchy's
integral theorem (for some $\epsilon>0$)
\[ \Diag\left( F \right)(t) = [x^{-1}] \frac{1}{x} F \left(
x, \frac{t}{x} \right) = \frac{1}{2 \pi i} \oint_{|x|=\epsilon} F \left( x,
\frac{t}{x} \right) \, \frac{dx}{x},\] 
which in a second step can be evaluated using the residues theorem as a sum of
residues (precisely: the residues of $F(x,t/x)/x$ at its ``small poles'',
having limit~0 at $t=0$). Each of these residues are algebraic functions, and
so is their sum $\textsf{Diag} (F)$.

\begin{example} [Dyck bridges, continued]\label{ex:Polya-Dyck-2}
The proof sketched above directly concludes that
\[  
\Diag\left( \frac{1}{1-x-y}\right)
= \frac{1}{2 \pi i} \oint_{|x|=\epsilon}  \frac{dx}{x-x^2-t} = 
{\left . {\frac{1}{1-2x}} \right |_{x = \frac{1-\sqrt{1-4t}}{2}} =} 
\frac{1}{\sqrt{1-4t}}.
\]
\end{example}	

\begin{example}\label{ex:diag3}
Interestingly, Pólya's result
becomes false for more than two variables. A simple example is provided by the rational function $1/(1-x-y-z) = \sum_{i,j,k} \frac{(i+j+k)!}{i!j!k!} x^iy^jz^k$, whose diagonal is
\[
\Diag \left( \frac{1}{1-x-y-z}\right) 
= \sum_{n \geq 0} \frac{(3n)!}{n!^3} t^n .
\] 
The transcendence of this function can
be proved in various ways, for instance by using asymptotics: 
Stirling's formula implies that $\frac{(3n)!}{n!^3} \sim \frac{\sqrt{3}}{2\pi} \, \frac{27^n}{n}$
and the presence of the factor $n^{-1}$ is incompatible with algebraicity
by \cref{crit:Flajolet}.
Another proof is based on rewriting
\begin{equation}\label{eq:hyperdiag}
\Diag \left( \frac{1}{1-x-y-z}\right) 
= {}_2 F_1\left(\left[\frac 1 3, \frac 2 3\right], \left[ 1 \right]; 27t\right),
\end{equation}
and by using the Schwarz or the Landau-Errera criteria mentioned above.
\end{example}	

The diagonal in~\cref{eq:hyperdiag} is hypergeometric, hence D-finite. In general, there is no
reason that the diagonal of a multivariate rational function be hypergeometric. However, that
all such diagonals are D-finite functions is a general fact. In fact, much more holds: a
theorem by Lipshitz~\cite{Lipshitz88} states that if $F(x_1,\ldots,x_n)$ is a multivariate
D-finite function\footnote{This means that $F$ satisfies a system of $n$ linear partial differential equations, the $i$-th one being an ordinary linear differential equation with respect to $\frac{\partial}{\partial x_i}$ and with polynomial coefficients in $x_1, \ldots, x_n$.}, then $\Diag(F)$ is D-finite.

The particular case where $F$ is rational is already interesting and nontrivial to prove.
In this case, the argument is the following.
First, as in the bivariate
case, if $F = P/Q \in \Q(x_1,\ldots,x_n)\cap \Q[[x_1,\ldots,x_n]]$, then the residue
theorem allows to write (for some $\epsilon>0$)
\begin{equation*}\label{eq:Deligne}
	\Diag(F)(t) = \frac{1}{(2\pi i)^{n-1}} \bigointsss_{|x_1| = \cdots = |x_{n-1}| = \epsilon} 
F \left( x_1, \ldots, x_{n-1}, \frac{t}{x_1\cdots x_{n-1}}\right)\frac{dx_1 \cdots dx_{n-1}}{x_1\cdots x_{n-1}},
\end{equation*}
so that $\Diag(F)(t)$ is the period function of a (family of) rational
functions. Its D-finiteness is then a consequence of the
finite-dimensionality over $\mathbb{C}(t)$ of the de
Rham cohomology for the complement of the variety in
$\mathbb{A}^n_{\mathbb{C}(t)}$ defined by the equations
$Q(x_1,\ldots,x_n)=0$ and $x_1\cdots x_n=t$.
(This
finiteness proof usually relies on a geometric argument in the smooth case, and 
on Hironaka's resolution of singularities in the general case.)
In more down-to-earth terms this proof guarantees, in a non-effective way,
that repeated differentiation under the integral sign eventually produces a
finite sequence of rational integrands that admit a linear combination with
coefficients in $\Q(t)$ that becomes an exact differential.

If $f(t)$ is the diagonal of a rational function, then not only is $f(t)$ D-finite, but in
addition $f(t)$ is \emph{globally bounded}, that is, $f(t)$ has a nonzero radius of
convergence in $\C$ and $\beta \cdot f (\alpha \cdot t) \in \Z [[t]]$ for some $\alpha,
\beta\in\Z\setminus \{ 0 \}$. (Note that this second property is equivalent to the existence of an $N\in\N\setminus \{ 0 \}$ such that $f(Nt) - f(0) \in \Z[[t]]$, as in
\cref{prop:eisenstein}.)
The following beautiful conjecture predicts that the converse is also true; it was formulated by Christol in the late 1980s, see e.g. \cite{Christol86} and~\cite[Conjecture~4]{Christol90}:

\begin{conj}[Christol's conjecture] \label{conj:Christol}
For $f(t) \in \Q[[t]]$, the following properties are equivalent:
\begin{enumerate}[label=(\arabic{enumi}),topsep=\parsep,itemsep=\parsep,parsep=0pt]
 \item $f(t)=\Diag(F)$ for some $F \in \Q(x_1,\ldots,x_n)\cap \Q[[x_1,\ldots,x_n]]$;
 \item $f\in \mathbb{Q}[[t]]$ is D-finite and globally bounded.
\end{enumerate}
\end{conj}

Christol's conjecture is far from being proved;
the following explicit problem, also due to 
Christol~\cite[p.~51]{Christol90}, is a very particular case of 
\cref{conj:Christol} and it is still open as of today.

\begin{question}  \label{question:Christol}
Is 
\begin{align*}
f(t) & = {}_{3}F_{2} \left( \left[ \frac19, \frac49, \frac59 \right], \left[1, \frac13 \right]; 3^6\, t \right) \\
& = 1+60 \, t +20475  \, t^{2}+9373650  \, t^{3}+ 4881796920  \, t^{4}+ \cdots 
\end{align*}
the diagonal of a rational power series?
\end{question}  

Therefore, even understanding diagonals which are hypergeometric functions is a very difficult problem; for some recent progress see~\cite{BoYu22}. 
 
Another natural and difficult question is whether a given diagonal of a rational function is
algebraic or transcendental. This question is directly connected to the main aim of this
article.
One may wonder whether it is possible to detect transcendence of a given diagonal $f(t) = \Diag(F)$  by reducing it modulo a prime $p$, and proving the transcendence of $f(t) \bmod p$ over $\F_p(t)$.
Unfortunately, this strategy is systematically doomed to failure: indeed, even if the diagonal $f(t)$ is transcendental, its reduction modulo $p$ is necessarily algebraic! 
This was proved by Furstenberg in~\cite[Theorem~1]{Furstenberg67}.
For instance, the transcendental diagonal in \eqref{eq:hyperdiag} is equal to 
\begin{align*}
(1+t)^{-1/4} & \; \bmod 5, \\
(1+6t+6t^2)^{-1/6} & \; \bmod 7, \\
(1+6t+2t^2+8t^3)^{-1/10} & \; \bmod 11.
\end{align*}

On the other hand, the Christol-André conjecture (\cref{conj:ChristolAndre}) implies that if
$f(t)$ is the diagonal of a rational function, then $f(t)$ is algebraic if and only if the condition~(3b) is fulfilled. 
The direct implication is true without the Christol-André conjecture; it relies on three nontrivial facts:
(i) the minimal-order LDE for a diagonal 
has only regular singularities with rational exponents~\cite{Christol88};
(ii) the classification of local solutions of Fuchsian linear differential equations~\cite[\S19]{Poole};
(iii) if $f(t)$ is algebraic, then its minimal-order LDE has only algebraic solutions~\cite[\S2]{CSTU02}.
If the condition~(3b) fails, then by (i) and (ii) at least one local solution $g(t)$ around $t=0$ of the minimal-order LDE for $f(t)$ involves logarithms, hence $g(t)$ is transcendental; then (iii)
implies that $f(t)$ must be transcendental as well.

For instance, the diagonal in \eqref{eq:hyperdiag} admits \( t \left(27 t -1\right) \partial_t^{2}+\left(54 t -1\right) \partial_t +6\) as minimal-order differential equation, with local basis
\[1+6 t +90 t^{2}+\cdots 
\quad \text{and} \quad
\log \! \left(t \right)+\left(6 \log \! \left(t \right)+15\right) t 
+\cdots .\]
Hence it is transcendental.

As a final remark, note that Theorem 1.1 in~\cite{Vargas21} implies that for
any prime $p\neq 3$, the reduction modulo~$p$ of the $_3F_2$ (transcendental)
function from~\cref{question:Christol} is algebraic (of degree at most
$p^{54}$). Hence, even if this $_3F_2$ function is not known to be a diagonal
of a rational function, its reductions modulo $p$ are known to behave as
reductions modulo $p$ of diagonals of rational functions.

\subsection{Examples from Combinatorics: Walks in the Quarter Plane}\label{ssec:walks}

In combinatorics, studying the nature of generating functions is of primary importance; for
instance, algebraicity may reveal essential (but potentially hidden) recursive structures of the
combinatorial classes under consideration.
In particular, many examples have been studied over the past decades in \emph{lattice path 
combinatorics}, a subfield of enumerative combinatorics. 
A plethora of interesting mathematical phenomena occur 
even when restricting to the
\emph{walks with small steps in the quarter plane} $\N^2$.
These are walks in the lattice $\Z^2$,
confined to the cone~${\mathbb{R}}_{+}^2$
that start at the origin $(0,0)$ and use steps in a
model (or stepset)~$\fS$ which is a fixed subset of the set of nearest-neighbor steps $ \{{\swarrow,}\ {\leftarrow,}\
{\nwarrow,}\ {\uparrow,}\ {\nearrow,}\ {\rightarrow,}\ {\searrow,}\
{\downarrow}\}.$
The systematic study of small step walks in $\N^2$ has been initiated by Bousquet-Mélou and Mishna in their germinal article~\cite{BMM10}.
An earlier reference on the topic, in a probabilistic context, is the book~\cite{FIM99}.
The study of generating functions of walks with small steps in the quarter plane now spans several decades and dozens of articles.
For instance, in~\cite{DHRS18}
the differential-algebraic nature of these generating functions
is completely elucidated 
using tools from differential Galois theory.
A survey with many references can be found in the recent article~\cite{Bostan21}.

Given a model $\fS\subseteq
\{{\swarrow,}\ {\leftarrow,}\
{\nwarrow,}\ {\uparrow,}\ {\nearrow,}\ {\rightarrow,}\ {\searrow,}\
{\downarrow}\}$, we denote by $\color{black}q_{i,j,n}$ the number of {\color{black}$\fS$-walks of
length~$n$ ending at $(i,j)$}. The \emph{full counting sequence} $(q_{i,j,n})_{(i,j,n)\in\N^3}$  admits several interesting specializations, for instance
$e_n\coloneqq q_{0,0,n}$, the number of $\fS$-walks of length~$n$
  returning to $(0,0)$ (``excursions'')
and $q_n \coloneqq  \sum_{i,j\geq0}q_{i,j,n}$, the number of 
  $\fS$-walks with prescribed length~$n$.
 
To these sequences one attaches
various functions, namely the \emph{full generating function}
\vspace{-0.2cm}
\begin{equation*}
Q_\fS(x,y,t)=\sum_{n=0}^\infty
\biggl(\sum_{i,j=0}^\infty q_{i,j,n}x^i y^j\biggr) t^n  \; \in {\mathbb Q}[x,y][[t]],
\vspace{-0.09cm}
\end{equation*}
 and its corresponding univariate specializations
 $Q_\fS(0,0,t) = \sum_{n\geq0} e_n t^n$ (``excursions generating function''),
 $Q_\fS(1,1,t) = \sum_{n\geq0} q_n t^n$ (``length generating function''),
 $Q_\fS(1,0,t)$ and $\color{black}Q_\fS(0,1,t)$ (``boundary returns'')
 and $[x^0] \, Q_\fS(x,1/x,t)$ (``diagonal returns'').
 
The general question in this setting is: given a model $\fS$, what can be said about the
generating function $Q_\fS(x,y,t)$, and its specializations? In particular, is~$Q_\fS(x,y,t)$
algebraic? Is it at least D-finite? Does $Q_\fS(x,y,t)$ (or at least some of its specializations)
admit closed-form expressions?

The model $\fS = \{\uparrow, \rightarrow\}$ in \cref{ex:Polya-Dyck} (Dyck bridges) is one of the simplest possible models. In that case, the generating function $Q_\fS(x,y,t)$ is algebraic. This is actually a particular case of a classical result stating that whenever $\fS$ is included in 
$ \{
\ {\uparrow,}\ {\nearrow,}\ {\rightarrow,}\ {\searrow,}\
{\downarrow}\}$
or in  
$\{
{\leftarrow,}\
{\nwarrow,}\ {\uparrow,}\ {\nearrow,}\ {\rightarrow}
\}$, 
that is, if the walks are essentially 1-dimensional,
then $Q_\fS(x,y,t)$ is algebraic.

For this reason, most studies in the area focus on the truly 2-dimensional cases, that is on
models~$\fS$ that contain both a step oriented towards the horizontal and the vertical axes
(see Figure 1 in~\cite{Bostan21}). A full classification is now available, according to the
algebro-differential properties of $Q_\fS(x,y,t)$, but here we restrict to two examples.

\subsubsection{Trident walks}
Up to some canonical reductions, there are exactly $23$ models of walks with small steps in the quarter plane whose full generating function $Q(x,y,t)$ is D-finite (with respect to $x, y$ and $t$).
They are displayed in Figure 4 in~\cite{Bostan21}.

One of these models is that of the 
``trident walks'', with $\fS = \{ {\nwarrow,}\
{\uparrow,}\ {\nearrow,}\ {\downarrow}\}$. This is entry~$7$ in~\cite[Fig.~4]{Bostan21} and
in~\cite[Table~6]{BoChHoKaPe17}.
There exist 
$1$ trident walk of length $0$ (the empty walk), 
$2$ trident walks of length $1$ ($\{ \uparrow \}, \{ \nearrow \}$)
and
$7$ trident walks of length $2$ 
($
\{ \uparrow - \uparrow\}, 
\{ \uparrow - \downarrow\}, 
\{ \uparrow - \nearrow \},
\{ \nearrow - \nwarrow\},
\{ \nearrow - \uparrow\},
\{ \nearrow - \nearrow\},
\{ \nearrow - \downarrow\}
$).
It was proved in~\cite{BoChHoKaPe17} that the length generating function of trident walks
	\[ Q(t) = \sum_{n \geq 0} q_n t^n = 1+2 \, t+7 \, t^2+23 \, t^3+84 \, t^4+301 \, t^5+1127 \, t^6+\cdots\]      
is D-finite and satisfies $L_5(Q(t)) = 0$, where $L_5$ is the following differential equation of order $5$ with polynomial coefficients of degree at most $15$:

\begin{small}
\begin{flalign*}
t^{2} \left(t -1\right) \left(4 t -1\right) \left(12 t^{2}-1\right) \left(5184 t^{7}-4128 t^{6}+4416 t^{5}+400 t^{4}+252 t^{3}-90 t^{2}-42 t +3\right) \left(4 t^{2}+1\right) \partial_t^{5}+  & \\
t \left(21399552 t^{13}-38486016 t^{12}+43416576 t^{11}-25803264 t^{10}+7762176 t^{9}-3848960 t^{8}+ \right. & \\ \left. 337088 t^{7}+143168 t^{6}-7128 t^{5}+   45328 t^{4}-11304 t^{3}-1854 t^{2}+540 t -27\right) \partial_t^{4}+ & \\
\left(143327232 t^{13}-222621696 t^{12}+257753088 t^{11}-122575104 t^{10}+36213888 t^{9}-19897728 t^{8}+ \right. & \\ \left. 1942656 t^{7}+70768 t^{6}-100456 t^{5}+ 254712 t^{4}-35124 t^{3}-7404 t^{2}+1116 t -48\right) \partial_t^{3}+ & \\
\left(346374144 t^{12}-454643712 t^{11}+545398272 t^{10}-166067712 t^{9}+59053824 t^{8}-32668800 t^{7}+ \right. & \\ \left.  2167392 t^{6}+ 54912 t^{5}-687744 t^{4}+500616 t^{3}-31176 t^{2}-5004 t +288\right) \partial_t^{2}+ & \\
\left(262766592 t^{11}-284000256 t^{10}+358041600 t^{9}-21846528 t^{8}+ 33115392 t^{7}- 13748736 t^{6}- \right. & \\ \left.   1184640 t^{5}+ 651744 t^{4}-894672 t^{3}+278496 t^{2}-7272 t +2880\right) \partial_t + & \\35831808 t^{10}-31186944 t^{9}+42163200 t^{8}+11639808 t^{7}+4981248 t^{6}-  & \\ 981504 t^{5}-809280 t^{4}+72576 t^{3}-177408 t^{2}+8064 t -3168.
\end{flalign*}
\end{small}

The length generating function $Q(t)$ is completely and uniquely defined by $L_5$ and by the first coefficients $[t^k]Q(t)$ for $k=0,\ldots, 14$.
The question is: how to determine the algebraic or transcendental nature of $Q(t)$?

Note that the asymptotic estimate 
$q_n \sim \gamma \beta^n n^r$ 
in~\cite[Fig.~4]{Bostan21},
with
$\gamma = 4/(3\sqrt\pi)$, $\beta=4$, $r=-1/2$,
is compatible with algebraicity,
since 
$r \in \Q\setminus\Z_{<0}$, $\beta \in \Qbar$ and
$\gamma \Gamma(r+1) = 4/3\in \Qbar$. 
Hence we cannot conclude the transcendence using asymptotic arguments as in \cref{ex:diag3}.

An interesting feature is that $L_5$ admits a factorization $L_2 \cdot L_{1,a} \cdot L_{1,b} \cdot L_{1,c}$, where the three operators $L_{1,\star}$ have order 1, and $L_2$ has order~$2$. This type of factorization $2/1/\cdots/1$ actually holds in all the cases 1--19 in~\cite[Fig.~4]{Bostan21}.
On the other hand, the algorithmic method (a variant of the \emph{creative telescoping} mentioned in~\cref{sec:perimeter}) that produces the differential equation $L_5$ does not guarantee that it is the least-order one satisfied by $Q(t)$. 
In~\cite{BoChHoKaPe17}, the above factorization was algorithmically computed and exploited
(together with some other computer algebra algorithms) in order to produce the following
expression for $Q(t)$ (and similar expressions for all cases 1--19 in~\cite[Fig.~4]{Bostan21}):
\begin{multline*}
Q(t) =
\frac1{t(t-1)} \int_0^t
  \frac{u}{(1-4u)^{3/2}} \left\{4+\int_0^u
    \frac{(1-4v)^{1/2}(\frac12+v)}{v^2} \left[1+\frac1{2v(1+2v)(1+4v^2)^{1/2}} \times {} \right. \right. \\
\hspace{-0.3cm}    \left. \left. \left((1-v) \, 
	_2 F_1 \left(\left[\frac{1}{2},\frac{3}{2}\right] \!\!, [1]; \frac{16 t^{2}}{4 t^{2}+1}\right)
     -(1+v)(1-4v+8v^2) \, 
	_2 F_1 \left(\left[\frac{1}{2},\frac{1}{2} \right]\!\!, [1]; \frac{16 t^{2}}{4 t^{2}+1}\right)
	 \right)
  \right]
  \, dv \right\}
\, du .
\end{multline*}
From such an expression, it is not difficult to prove that $Q(t)$ is transcendental, using the fact that the ${}_2 F_1$ occurring in the complete elliptic integral of the first kind,
${}_2 F_1 ([1/2, 1/2], [1]; t)$, is transcendental.
In~\cite[\S4.2]{BoChHoKaPe17} it was shown that $Q(t)$ is transcendental by exploiting the
explicit factorization of~$L_5$, and by applying to $L_2$ a specific algorithm that decides algebraicity of solutions for operators of order 2, namely Kovacic's algorithm~\cite{Kovacic86}.
The same approach uniformly works on the models 1--19 and allows to prove the transcendence of the corresponding $Q(t)$ except in case~17 for the model $\fS = \{ \uparrow, \leftarrow, \searrow \}$ and 
in case~18 for the model  $\fS = \{ \uparrow, \leftarrow, \searrow, \downarrow, \rightarrow, \nwarrow \}$
(in these two cases the algorithm proves algebraicity of all solutions of the corresponding $L_2$).
Very interestingly, the full generating function $Q(x,y,t)$ is proved to be transcendental in all cases 1--19, and cases 17 and 18 are the only ones with algebraic specialization $Q(1,1,t)$.

An alternative proof of the transcendence of the length generating function $Q(t)$ for trident
walks is based on the fact that $L_5$ is indeed the minimal-order differential equation for $Q(t)$; this implies
that if $Q(t)$ were algebraic, then $L_5$ would only have algebraic solutions, a
situation discarded by exhibiting logarithms in the local solutions of $L_5$ at $t=0$. This
approach to the proof of transcendence of solutions of linear differential equations is used in~\cite{BoSaSi}; its heart is the ``minimization'' algorithm from~\cite{BoRiSa22}.

\subsubsection{Gessel walks}

The most difficult model of small-step walks in the quarter plane is Gessel's model,
for which $\fS = {\{\nearrow,\swarrow,
\leftarrow,\rightarrow \}}$.
For notational simplicity we will write 
$G(x,y,t)$ for 
the full generating function in this case, 
and  $G(t)$ for the length generating function for Gessel walks $G(1,1,t)$.

Around 2000, Ira Gessel formulated
two conjectures equivalent to the following statements:
\begin{conj}\label{conj:Gessel1}
The generating function $G(0,0,t) = 1+2t^2+11t^4+85t^6 + 782t^7+\cdots$ of Gessel excursions is equal to
\(
{_3F_2} \left([5/6, 1/2, 1], [5/3, 2]; 16t^2 \right)
.
\)
\end{conj}

\begin{conj}\label{conj:Gessel2}
The  generating function $G(x,y,t)$ is not D-finite.
\end{conj}

Gessel's first conjecture was first solved in 2009 by Kauers, Koutschan and Zeilberger
in~\cite{KaKoZe09} using a computerized guess-and-prove approach.
Unfortunately, solving \cref{conj:Gessel1} had no implication concerning \cref{conj:Gessel2},
and in particular on the D-finiteness of the length generating function $G(t)$. It came as a total surprise when Bostan and Kauers~\cite{BoKa10} proved that Gessel's second conjecture was false.

\begin{theo}[\cite{BoKa10}]\label{theo:GesselAlg}
The generating function $G(x,y,t)$ for Gessel walks is algebraic.

Moreover, the coefficients $(g_n)$ of $G(t) = 1+2 t +7 t^{2}+21 t^{3}+78 t^{4}+260 t^{5}+ \cdots$ satisfy
$(3n+1) \, g_{2n} = (12n+2) \, g_{2n-1}$ 
and 
$(n+1) \, g_{2n+1} =  (4n+2) \, g_{2n}
$
for all $n\geq 0$.

In addition, $G(t) = (H(t)-1)/(2t)$ where
$H(t) = 1+2 t +4 t^{2}+14 t^{3}+42 t^{4}+ \cdots$
is a root of 
$27 \left(4 t -1\right)^{2} H(t)^{8}-18 \left(4 t -1\right)^{2} H(t)^{4}-8 \left(16 t^{2}+24 t +1\right) H(t)^{2}-\left(4 t -1\right)^{2}=0$, and
\[
H(t) = {{}_{2}^{}{{}{{F_{1}^{}}}}\! \left(\left[-\frac{1}{12}, \frac{1}{4} \right], \left[\frac{2}{3}\right];-\frac{64 t \left(4 t +1\right)^{2}}{\left(4 t -1\right)^{4}}\right)}.
\]
\end{theo}

The original discovery and proof of Theorem~\ref{theo:GesselAlg} was
computer-driven, and used a guess-and-prove approach, based on
\emph{Hermite-Pad\'e approximants}. 
As a byproduct of this proof, 
the minimal polynomial of $G(x,y,t)$ 
has been estimated to have more than $10^{11}$ terms when
written in expanded form, for a total size of $\approx 30$\,Gb (!).
The guess-and-prove method is a 3-step process: (i) compute~$G(x,y,t)$ to precision $t^{1200}$ ($\approx 1.5$ billion coefficients!); (ii) conjecture polynomial equations for $G(x,0,t)$ and $G(0,y,t)$; (iii) conclude the proof by computing multivariate resultants of (very big) polynomials (30 pages each).

As a matter of fact, the discovery of algebraicity was initially performed in a different way: 
the expansion of the generating function $G(x,y,t)$ modulo $t^{1000}$ was 
sufficient to guess (by using \emph{differential} {Hermite-Pad\'e
approximants}) two differential operators 
\begin{itemize}
\item  $\op{L}_{x,0} \in \Q(x,t)\langle \partial_t \rangle$, of order 11 in $\partial_t$, bidegree $(96,78)$ in $(t,x)$, 
and integer coefficients of at most 61 digits
\item  $\op{L}_{0,y} \in \Q(y,t)\langle \partial_t \rangle$, of order 11 in $\partial_t$, bidegree $(68,28)$ in $(t,y)$,
and integer coefficients of at most 51 digits
\end{itemize}
such that $\; \op{L}_{x,0}({G(x,0,t)}) = 0 \bmod t^{1000}$
and $\op{L}_{0,y}({G(0,y,t)}) = 0 \bmod t^{1000}$.    

After this guessing step of plausible
differential equations for $Q(x,0,t)$ and for $Q(0,y,t)$, an important step in the discovery process 
was to apply \cref{conj:Grothendieck1} with several primes~$p$.
More precisely, for randomly chosen
prime numbers~$p$, and $a,b\in\F_p$, both $\op {L}_{a,0}$ and
$\op {L}_{0,b}$ right-divide the pure power $\partial_t^{p}$ in
$\F_p(x)\langle \partial_t\rangle$; in other terms, 
they have zero $p$-curvature for all the tested primes~$p$ (see 
\cref{def:pcurv-general}). This was the key 
observation in the discovery~\cite{BoKa10} that the trivariate generating 
function for Gessel walks is algebraic.

Several \emph{human proofs} of~\cref{conj:Gessel1} and \cref{theo:GesselAlg} have been
discovered since the publication of~\cite{BoKa10}: the first one used complex
analysis~\cite{BoKuRa17}, the second one is purely algebraic~\cite{Bousquet16}, 
the third one is both combinatorial and analytic~\cite{Budd20},
while the more recent one is probably the most elementary~\cite{BeBoRa21}.

\subsection{Examples from Number Theory}

In his ICM 2018 paper~\cite[p.~768]{Zagier18}, Zagier introduced a recurrent sequence of rational numbers $(c_n)_{n\geq 0}$ defined by 
initial terms
$c_0 = 1, c_1 = -161/(2^{10}\cdot 3^{5})$ and $c_2 = 26605753/(2^{23} \cdot 3^{12} \cdot 5^2)$,
and by the following linear recursion with polynomial coefficients:
\begin{align*}                          
c_{n-3}     
+ 20\left(4500 n^{2}-18900 n+19739\right) c_{n-2} 
+
80352000 n(5 n-1)(5 n-2)(5 n-4) c_{n}  \\ 
\hspace{-1cm}
+ 25\left(2592000 n^{4}-16588800 n^{3}+39118320 n^{2}-39189168 n+14092603\right) c_{n-1}  
 = 0.
\end{align*}
This mysteriously looking sequences arises from a topological differential equation in the
work of Bertola, Dubrovin and Yang~\cite{BeDuYa18}, each coefficient
$c_n$ being defined by an integral over the Deligne-Mumford moduli 
spaces of stable curves of genus $g$ with $n$ marked points.
It is not difficult to prove that $c_n$ behaves asymptotically like
$1/n!^2$. Inspired by an analogy with the behavior of the 
\emph{quantum periods} in mirror symmetry, Zagier asked whether it is
possible that the $c_n$'s become integers after multiplication by
$n!^2$, or more generally by the product of two rising factorials. 
Zagier mentions the following highly nontrivial two results:
\begin{itemize}
    \item {[Yang \& Zagier]}: $a_n \coloneqq (2^{10} \cdot 3^5 \cdot  5^4)^n \cdot        (\textcolor{orange}{3/5})_n \cdot  (\textcolor{orange}{4/5})_n \cdot c_n \in \Z$ for all integers $n \geq 0$;
    \item {[Dubrovin \& Yang]}: $b_n \coloneqq (2^{12}\cdot 3^{5} \cdot  5^4)^n  \cdot       (\textcolor{magenta}{2/5})_n \cdot (\textcolor{magenta}{9/10})_n \cdot c_n \in \Z$ for all integers $n \geq 0$.
\end{itemize}

Both sequences are integer sequences of exponential growth, and hence
can be expected to have a generating series that is (the Taylor expansion at $0$ of) a period function in the sense of algebraic geometry.
For the sequence $(b_n)$, Zagier mentions that its generating function
is even algebraic, but that the sequence $(a_n)$ seems to be more
challenging. Yurkevich~\cite[Theorem~2]{Yurkevich22} proved that the
generating function of the sequence $(a_n)$ is also algebraic. 
(See~\cite[Thm.~2]{DuYaZa23} for a very closely related result.)
Inspired
by these results, it is natural to look at the following question.

\begin{question}
	Find $(u,v)\in \Q$ such that 
	$\widetilde{c}_n \coloneqq w^n \cdot (u)_n \cdot  (v)_n \cdot  c_n\in\Z$ for all $n\geq 0$, for some $w\in\Z$.
\end{question}

A natural related question is the following.

\begin{question}
Is it true that for these $(u,v)\in \Q$, the generating function $\sum_{n \geq 0} \widetilde{c}_n x^n$ is algebraic?
\end{question}

In a work (in progress) by Bostan, Weil and Yurkevich, the following result is conjectured:

\begin{conj}\label{conj:BoWeYu}
	The only pairs $(u,v) \in \mathbb{Q}^2 \cap (0,1)^2$ such that 
there exists $w\in\Z$	such that $\widetilde{c}_n \coloneqq w^n \cdot (u)_n \cdot  (v)_n \cdot  c_n\in\Z$ for all $n\geq 0$ are the ones listed in Fig.~\ref{fig:pairs}. Moreover, for each of these pairs, the corresponding generating function $\sum_{n\geq 0} \widetilde{c}_n x^n$
is algebraic, of algebraicity degree as in Fig.~\ref{fig:pairs}.
\end{conj}	

\smallskip
\begin{figure}[t]
\centerline{%
\begin{tabular}{|@{\,}c@{\,}|@{\,}c@{\ \ }c@{\,}c@{\,}c@{\,}||@{\,}c@{\,}|@{\,}c@{\ \ }c@{\,}c@{\,}c@{\,}|}
\hline
 $\; \# \; $    & $\; u \;$    & $\; v \;$      & order & {\quad alg. degree}
& $\; \# \; $   & $\; u \;$    & $\; v \;$      & order & {\quad alg. degree}
\\
\hline
   1 & 1/5 & 4/5 & 2  & 120
& 6 & 19/60 & 49/60 & 4   & 155520
\\
   2 & \textcolor{orange}{3/5} & \textcolor{orange}{4/5} & 2  & 120
& 7 & 19/60 & 59/60  & 4   & 46080   
\\
   3 & \textcolor{magenta}{2/5} & \textcolor{magenta}{9/10} & 4   & 120
& 8 & 29/60 & 49/60 & 4   & 46080 
\\
   4 &7/30 & 9/10 & 4   & 155520
& 9 & 29/60 & 59/60 & 4   & 155520 
\\
   5 & 9/10 & 17/30 & 4   & 155520
&  &  &  &    &
\\
\hline
\end{tabular}
}      
\caption{Pairs $(u,v)$ for which the sequence 
$(u)_n \cdot  (v)_n \cdot  c_n$ is (conjecturally) almost integral.
In all cases, the generating functions are (conjecturally) algebraic.
Algebraicity degrees are ``guessed'' by numerical monodromy computations.}
\label{fig:pairs}
\end{figure}

For instance, in case 4 of Figure~\ref{fig:pairs}, \cref{conj:BoWeYu} states that the generating function of the sequence 
$\widetilde{c}_n \coloneqq (2^{14} \cdot 3^7 \cdot 5^4)^n \cdot (7/30)_n \cdot  (9/10)_n \cdot c_n$,
\[\sum_{n\geq 0} \widetilde{c}_n x^n = 
1-3042900 x +58917109730850 x^{2}-1389307608898903890000 x^{3}
+\cdots
\]
is an algebraic solution of degree 155520 of the following order-4 (irreducible) operator
\begin{flalign*}
\op L_4 = 	125 x^{3} \left(88335360 x +1\right) \left(7739670528000 x^{2}+31104000 x +1\right) \partial_x^{4}+\\
	25 x^{2} \left(35095911228443197440000 x^{3}+95685546737664000 x^{2}+2823828480 x +23\right) \partial_x^{3}+\\
	60 x \left(36896918938488668160000 x^{3}+56436938459136000 x^{2}+1177963920 x +7\right) \partial_x^{2}+\\\left(1254982687120120872960000 x^{3}+654118326337536000 x^{2}+16648081920 x +12\right) \partial_x +\\42055270898174263296000 x^{2}-134823448166400 x +36514800.
\end{flalign*}
In particular, this operator admits (conjecturally) only 
algebraic solutions. To our knowledge, this is an open problem. A 
heuristic check is based on \cref{conj:Grothendieck1}: it is not 
difficult to check (using a computer-algebra system) that the remainder 
of the right Euclidean division of $\partial_x^p$ by $\op L_4$ is zero 
for all primes $5 < p < 100$ except for $p \in \{ 7, 31 \}$. 
\cref{conj:Grothendieck1} then suggests that indeed $\op L_4$ possesses algebraic solutions only.

\section{Grothendieck's conjecture}
\label{sec:Grothendieck}

The examples presented in \S\ref{sec:examples} motivate the following question:
given the linear differential equation~\eqref{eq:lineardiff} with coefficients
in $\Q(x)$, can we give necessary and/or sufficient conditions for admitting
algebraic solutions (either one, or all of them)?

As already mentioned earlier, Grothendieck's conjecture
(Conjecture~\ref{conj:Grothendieck1}) provides such a criterion: in its
enhanced version (\cref{groth conj}) it relates the existence of a ``full basis
of algebraic solutions'' of the differential equation~\eqref{eq:lineardiff} to
the existence of a ``full basis of rational solutions'' of its reductions modulo
almost all prime numbers~$p$.

The aim of this section is, first of all, to introduce the notion of 
$p$-curvature and to state a precise version of Grothendieck's 
conjecture. We first examine in detail the case of first order 
equations in \S\ref{ssec:order1} and come to the general case in 
\S\ref{sec:stat Groth conj}. The next sections are more focussed on 
the $p$-curvature itself: in \S\ref{ssec:algorithms}, we describe 
efficient methods for computing it in practice, while in 
\S\ref{ssec:integrality}, we examine how it controls the growth of 
denominators and, in particular, the integrality of the solutions of the 
starting differential equation.

\subsection{The case of equations of order $1$}\label{sec: groth order 1}
\label{ssec:order1}

Consider a linear differential operator of order $1$
\begin{equation}\label{op order 1 char 0}
	\op L = \partial_x + a(x)  
\end{equation} 
with $a(x) \in \Q(x)$. It makes sense to consider the reduction $a(x) \mod p
\in \mathbb{F}_{p}(x)$ of (the coefficients of) $a(x)$ modulo $p$ for almost
all prime numbers~$p$. Thus, one can consider the reduction
\begin{equation}\label{op order 1 reduced}
	\op L_{p}= \partial_x + a(x) \mod p 
\end{equation} 
of $\op L$ modulo $p$ for almost all primes~$p$. This is a linear differential
operator of order $1$ with coefficients in $\mathbb{F}_{p}(x)$. Our aim is to
relate the existence of a nonzero algebraic solution of \eqref{op order 1 char
0} to the existence of nonzero rational solutions of the reduced equations
\eqref{op order 1 reduced}.

In what follows, we say that a function $f$ is a solution of $\op L$ 
(resp. of $\op L_p$) when it is a solution of the corresponding 
differential equation, \emph{i.e.} when $\op L(f) = 0$ (resp. $\op 
L_p(f) = 0$).

\subsubsection{Rational and algebraic solutions in characteristic $0$: a criterion} 

What makes the case of first order equations tractable is the fact that there
is an explicit criterion for the existence of a nonzero algebraic (or rational)
solution.

\begin{prop}\label{prop:rationalorderone}
The monic first order differential operator \eqref{op order 1 char 0} has a nonzero rational (resp. algebraic) solution if and only if its constant coefficient $a(x)$ has at most a simple pole with integral (resp. rational) residue at each point of~$\Qbar$ and vanishes at $\infty$.
\end{prop}

\begin{proof}
We first consider the ``rational case''. 
Let us first assume that $a(x)$ has at most a simple pole with integral residue at each point of $\Qbar$ and vanishes at $\infty$. We thus have  
\begin{equation}\label{decomp elts simples}
a(x) = \sum_{i=1}^m \frac{n_i}{x - a_i}
\end{equation}
for some $a_{i} \in \Qbar$ and some $n_{i} \in \Z$. A straightforward calculation shows that 
\begin{equation}\label{sol cas rationnel}
f(x)=\prod_{i=1}^m (x - a_i)^{n_i}
\end{equation}
is a nonzero rational solution of \eqref{op order 1 char 0}.    

Conversely, assume that \eqref{op order 1 char 0} has a nonzero rational solution $f(x)$.  This $f(x)$ can be factored as a product of linear factors 
$f(x) = c\prod_{i=1}^m (x - a_i)^{n_i}$ with $c \in \Qbar^{\times}$, $a_i \in 
\Qbar$ and $n_i \in \Z$. A straightforward calculation yields 
\[a(x)=\frac{f'(x)}{f(x)} = \sum_{i=1}^m \frac{n_i}{x - a_i}.\]
This shows that $a(x)$ has at most a simple pole with integral residue at each point of $\Qbar$ and vanishes at $\infty$, as expected.

We now consider the ``algebraic case''.
Let us first assume that $a(x)$ has at most a simple pole with rational residue at each point of $\Qbar$ and vanishes at $\infty$. 
Then, we can argue as we did above in the rational case, the only differences being that the $n_{i}$ involved in \eqref{decomp elts simples} are no longer in $\Z$ but in $\Q$ and that \eqref{sol cas rationnel} is no longer rational but algebraic.

Conversely, assume that \eqref{op order 1 char 0} has a nonzero algebraic solution $f(x)$.  
Let $M(Y) = Y^N + \sum_{i=0}^{N-1} m_i(x) Y^i \in \Q(x)[Y]$ be the minimal polynomial of~$f(x)$ over~$\Q(x)$. 
By differentiating  the equality~$M(f)=0$ with respect to~$x$ and by using     
$f'(x)=a(x)f(x)$, we get 
\begin{multline*}
 0=M(f)'= Nf^{N-1} f' + \sum_{i=0}^{N-1} m_i'(x) f^i + \sum_{i=0}^{N-1} m_i(x) i f^{i-1} f'\\ 
 =Nf^{N-1} a(x)f+ \sum_{i=0}^{N-1} m_i'(x) f^i + \sum_{i=0}^{N-1} m_i(x) i f^{i-1} a(x)f\\
 = Na(x) f^{N} + \sum_{i=0}^{N-1} (m_i'(x)+m_i(x) ia(x)) f^i. 
 \end{multline*}
Hence, the polynomial $P(Y)=Na(x) Y^{N} + \sum_{i=0}^{N-1} (m_i'(x)+m_i(x) i a(x)) Y^i$ satisfies $P(f)=0$. By minimality of $M(Y)$, we get $P(Y)=Na(x)M(Y)$. Equating the constant terms in this equality, we get that $m_0(x)$ is a nonzero solution in~$\Q(x)$ of $y'(x) = N a(x) y(x)$. Using the ``rational case'' treated above, we get that $Na(x)$ has at most a simple pole with integral residue at each point of $\Qbar$ and vanishes at $\infty$. Hence, $a(x)$ has at most a simple pole with rational residue at each point of $\Qbar$ and vanishes at $\infty$, as expected. 
\end{proof}

Note that the proof above is very similar (in fact, generalizes) the one used in \cref{ex:exp} to prove that the exponential function is transcendental.

\subsubsection{Algebraic solutions: from characteristic $0$ to characteristic $p$}

We shall now consider the following question: assuming that \eqref{op order 1
char 0} has a nonzero algebraic solution, what can be said about the reduced
equation \eqref{op order 1 reduced}? It is natural to expect that the latter
has a nonzero algebraic solution as well for almost all primes~$p$. Actually,
something even better happens. 

\begin{ex}
\label{ex:eqord1}
Consider the differential equation 
\begin{equation}\label{eq order 1 ex}
 y'=\frac{1}{2(x-1)}y.
\end{equation}
It has a nonzero algebraic solution, namely $f(x)=(1-x)^{1/2}$. For any prime
$p \neq 2$, one can consider the reduction of \eqref{eq order 1 ex} modulo $p$.
Any such reduced equation has a nonzero algebraic solution, namely
$f_{p}(x)=(1-x)^{1/2}$. Let us clarify our notations: $f_{p}(x)$ is a root of
the polynomial $Y^{2}-(1-x) \in \mathbb{F}_{p}(x)[Y]$, whereas $f(x)$ is a root
of the polynomial $Y^{2}-(1-x) \in \Q(x)[Y]$. However, the reduction of
\eqref{eq order 1 ex} modulo $p$ can also be written as $y'=\frac{n_{p}}{x-1}y$
where $n_{p} \in \Z$ is such that $n_{p} \equiv \frac{1}{2} \mod p$ and, hence
$\widetilde{f}_{p}(x)=(1-x)^{n_{p}}$ is a nonzero solution of the reduction
modulo $p$ of \eqref{eq order 1 ex}. The interesting point is that
$\widetilde{f}_{p}(x)$ is not only algebraic but rational, contrary to
$f_{p}(x)$! 
\end{ex}

The conclusion of Example~\ref{ex:eqord1} is a general fact as shown by Theorem \ref{theo:
from char 0 to p} below. Let us first give an analogue in positive
characteristic of the ``rational case'' of Proposition
\ref{prop:rationalorderone}.

\begin{prop}\label{prop:rationalorderone char p}
Consider $b(x) \in \mathbb{F}_{p}(x)$. The differential equation
	\[y' + b(x) y = 0\]
has a nonzero rational solution if and only if $b(x)$ has at most a simple pole
with residue in $\mathbb{F}_{p}$ at each point of $\overline{\mathbb{F}_{p}}$
and vanishes at $\infty$. 
\end{prop}

\begin{proof} The proof is entirely similar to the proof of the rational case
of Proposition \ref{prop:rationalorderone}, it is sufficient to replace
everywhere $\Qbar$ by $\overline{\mathbb{F}_{p}}$ and $\Z$ by $\mathbb{F}_{p}$.
\end{proof}

\begin{ex}
\label{ex exp-arctan}
Consider the differential equation 
\begin{equation}\label{eq exp-arctan}
 y'=\frac{1}{x^2 + 1}y
\end{equation}
whose general solution in characteristic zero is $c \cdot \exp(\arctan(x))$
where $c$ is a constant.
We are interested in determining whether or not this equation has a
rational solution in characteristic~$p > 0$.
Modulo $p = 2$, the rational function $b(x)=1/(x^2 + 1)$ writes $1/(x+1)^2$; thus it has a pole of order~$2$ and hence the differential
equation~\eqref{eq exp-arctan} has no nonzero rational solutions by
Proposition~\ref{prop:rationalorderone char p}.
For $p \neq 2$, the partial fraction decomposition of $b(x)$ reads
\[b(x) = \frac i 2 \cdot \left(\frac 1{x+i} - \frac 1{x-i}\right)\]
where $i$ denotes a square root of $-1$ in $\overline{\mathbb{F}_{p}}$. 
We now need to distinguish
between two cases depending on the congruence class of $p$ modulo~$4$. 
Indeed, when $p \equiv 1 \pmod 4$, we have $i \in \Fp$ and so the 
residues belong to~$\Fp$ as well. In this case, the equation~\eqref{eq exp-arctan} has
then a nonzero rational solution, namely
\[y(x) = \left(\frac{x+i}{x-i}\right)^{i/2},\]
where the exponent $i/2$ is a lift in $\Z$ of $i/2 \in \Fp$.
On the contrary, when $p \equiv 3 \pmod 4$, we know that $-1$ is not a
square in $\Fp$, showing that the residues are not in $\Fp$ either.
Therefore, in this case, the equation~\eqref{eq exp-arctan} has
no nonzero rational solution.
\end{ex}

\begin{theo}\label{theo: from char 0 to p}
 If \eqref{op order 1 char 0} has a nonzero algebraic solution, then, for almost all primes~$p$, \eqref{op order 1 reduced} has a nonzero rational solution. 
\end{theo}

\begin{proof}
 Proposition \ref{prop:rationalorderone} (and its proof) ensures that 
\begin{equation}\label{decomp simp el alg}
 a(x) = \sum_{i=1}^m \frac{e_i}{x - a_i}
\end{equation}
for some $a_{i}\in \Qbar$ and some $e_{i} \in \Q$. 

Let us first assume that the $a_i$'s belong to $\Q$.   For any prime $p$, we let $\Z_{(p)}$ be the ring of rational numbers with denominator relatively prime to $p$. We denote by $\pi_{p} : \Z_{(p)} \rightarrow \mathbb{F}_{p}$ the ``reduction modulo $p$'' map.  For almost all primes $p$, the $a_{i}$'s and the $e_{i}$'s belong to $\Z_{(p)}$. For any such $p$, we have:  
\[a(x) \mod p= \sum_{i=1}^m \frac{\pi_{p}(e_i)}{x - \pi_{p}(a_i)}\]
and the result follows from Proposition \ref{prop:rationalorderone char p}. 

The proof in the general case is similar but requires basic notions from algebraic number theory. 
Let $K$ be a number field containing the $a_{i}$ and the $e_{i}$. Let $\mathcal{O}_{K}$ be the ring of integers of $K$. For any prime $\mathfrak{P}$ of $K$ (which is by definition a prime ideal of $\mathcal{O}_{K}$), we let ${\mathcal O}_{\mathfrak{P}}$ be the valuation ring of $K$ at $\mathfrak{P}$. We denote by $\kappa_{\mathfrak{P}}$ the corresponding residue field and by $\pi_{\mathfrak{P}} : {\mathcal O}_{\mathfrak{P}} \rightarrow \kappa_{\mathfrak{P}}$ the quotient map. The residue field $\kappa_{\mathfrak{P}}$ is a finite field of characteristic $p$ such that $\mathfrak{P} \cap \Z= (p)$. We say that $\mathfrak{P}$ is above $p$. For almost all primes $p$, for all primes $\mathfrak{P}$ of $K$ above $p$, the $a_{i}$'s and the $e_{i}$'s belong to ${\mathcal O}_{\mathfrak{P}}$. For such $p$ and  $\mathfrak{P}$, we have:  
\[a(x) \mod p=a(x) \mod \mathfrak{P}= \sum_{i=1}^m \frac{\pi_{\mathfrak{P}}(e_i)}{x - \pi_{\mathfrak{P}}(a_i)}.\]
Since $e_{i}$ is rational, $\pi_{\mathfrak{P}}(e_i)$ belongs to the prime subfield $\mathbb{F}_{p}$ of $\kappa_{\mathfrak{P}}$. The result follows from Proposition \ref{prop:rationalorderone char p}. 
\end{proof}

\subsubsection{From characteristic $p$ to characteristic $0$}

It is now tempting to ask: if \eqref{op order 1 reduced} has a nonzero rational
solution for almost all primes $p$, does \eqref{op order 1 char 0} have a
nonzero algebraic solution? The (positive) answer is given by the following
result.

\begin{theo}[Honda~\cite{Honda79}] \label{theo:Honda}
	\label{theo: from char p to 0}
The converse of Theorem \ref{theo: from char 0 to p} holds true, {\it i.e.},
if, for almost all primes~$p$, \eqref{op order 1 reduced} has a nonzero
rational solution, then \eqref{op order 1 char 0} has a nonzero algebraic
solution. 
\end{theo}

\begin{proof}
Consider the partial fraction decomposition of $a(x)$:
\[
a(x) =  P(x)+\sum_{i=1}^m \sum_{j=1}^{r_{i}} \frac{\alpha_{i,j}}{(x - a_i)^{j}}
\]
with $P(x) \in \Qbar[x]$, $a_{i} \in \Qbar$, $\alpha_{i,j} \in \Qbar$ and $r_{j} \in \Z_{\geq 1}$. According to Proposition \ref{prop:rationalorderone}, we have to prove that $P(x)$ and the $\alpha_{i,j}$'s for $j \geq 2$ are $0$ and that the $\alpha_{i,1}$'s belong to $\Q$.

Let $K$ be a number field containing the $a_{i}$, the $\alpha_{i,j}$'s and the coefficients of $P(x)$. We will use the notation and terminology (prime $\mathfrak{P}$ of $K$, valuation ring ${\mathcal O}_{\mathfrak{P}}$, quotient map $\pi_{\mathfrak{P}}$, {\it etc}.) introduced in the proof of Theorem~\ref{theo: from char 0 to p}. For almost all primes $p$, for all primes $\mathfrak{P}$ of $K$ above $p$, the $a_{i}$'s, the $\alpha_{i,j}$'s and the coefficients of $P(x)$ belong to ${\mathcal O}_{\mathfrak{P}}$. For such $p$ and  $\mathfrak{P}$, we have: 
\[
a(x) \mod p = a(x) \mod \mathfrak{P} =  P^{\pi_{\mathfrak{P}}}(x)+\sum_{i=1}^m \sum_{j=1}^{r_{i}} \frac{\pi_{\mathfrak{P}}(\alpha_{i,j})}{(x - \pi_{\mathfrak{P}}(a_i))^{j}},
\]
where $P^{\pi_{\mathfrak{P}}}(x)$ denotes the polynomial obtained from $P(x)$ by applying $\pi_{\mathfrak{P}}$ coefficientwise. 

Proposition \ref{prop:rationalorderone char p} ensures that, for almost all primes $p$, $a(x) \mod p$ has at most simple poles, so, for almost all primes $p$, for all primes $\mathfrak{P}$ of $K$ above $p$, for all $j \in \{2,\ldots,r_{i}\}$, we have $\pi_{\mathfrak{P}}(\alpha_{i,j})=0$, {\it i.e.}, $\alpha_{i,j} \in \mathfrak{P}$. This implies that, for all $j \in \{2,\ldots,r_{i}\}$, we have $\alpha_{i,j}=0$.
Similarly, Proposition \ref{prop:rationalorderone char p} also ensures that, for almost all primes $p$, $a(x) \mod p$ vanishes at $\infty$,  so, for almost all primes $p$, for all primes $\mathfrak{P}$ of $K$ above $p$, $P^{\pi_{\mathfrak{P}}}(x)=0$. This implies that $P(x)=0$. 
Last, Proposition \ref{prop:rationalorderone char p} ensures that, for almost all primes $p$, for all primes $\mathfrak{P}$ above $p$, we have $\pi_{\mathfrak{P}}(\alpha_{i,1}) \in \mathbb{F}_{p}$. Using Kronecker's  Theorem recalled below, we get that $\alpha_{i,1}$ belongs to $\Q$ and Proposition \ref{prop:rationalorderone} yields the desired result: \eqref{op order 1 char 0} has a nonzero algebraic solution.
\end{proof}

The Kronecker Theorem mentioned above (which is usually seen as a consequence of Chebotarev's density Theorem) reads as follows

\begin{theo}[Kronecker]
\label{theo:kronecker}
An irreducible element $P(x)$ of $\Q[x]$ such that, for almost all primes~$p$, $P(x) \mod p$ has a zero in $\mathbb{F}_{p}$ is linear.
\end{theo}

\subsubsection{Rational solutions in characteristic $p$ and $p$-curvature}\label{sec: p-curv order 1}
Consider a differential equation
\begin{equation}\label{eq order 1 char p for p curv}
 y' + b(x) y = 0
\end{equation}
with $b(x) \in \mathbb{F}_{p}(x)$. We will give an alternative criterion (an alternative to Proposition \ref{prop:rationalorderone char p}) for determining whether \eqref{eq order 1 char p for p curv} has a nonzero rational solution based on the notion of $p$-curvature that we shall now introduce. 
 
Consider the $\mathbb{F}_{p}(x^p)$-linear map 
\begin{eqnarray*}
 \Delta : \mathbb{F}_{p}(x) &\rightarrow& \mathbb{F}_{p}(x) \\ 
 f &\mapsto& f'+b(x) f. 
\end{eqnarray*}
The additivity is clear, the homogeneity follows from the fact that the elements of $\mathbb{F}_{p}(x^p)$ are constants of (the differential field) $\mathbb{F}_{p}(x)$ in the sense that their derivative is $0$ (more precisely, 
$\mathbb{F}_{p}(x^p)=\{f(x) \in \mathbb{F}_{p}(x) \ \vert \ f'(x)=0 \}$) implying that $(\alpha f)'=\alpha f'$ for all $\alpha \in \mathbb{F}_{p}(x^p)$ and $f \in \mathbb{F}_{p}(x)$. 

\begin{defi}
The map 
\[
\Delta^{p}: \mathbb{F}_{p}(x) \rightarrow \mathbb{F}_{p}(x)
\] 
is called the $p$-curvature of \eqref{eq order 1 char p for p curv}.  
\end{defi}

A remarkable and fundamental fact is that the $p$-curvature is not only $\mathbb{F}_{p}(x^p)$-linear but it is also $\mathbb{F}_{p}(x)$-linear.  Indeed, a simple induction along with the Leibniz rule show that, for all $k \geq 0$, for all $\alpha, f \in \mathbb{F}_{p}(x)$, we have 
\[
\Delta^{k} (\alpha f) = \sum_{i=0}^{k} \binom{k}{i} \alpha^{(i)} \Delta^{k-i}(f).
\]
Taking $k=p$ and using the fact that 
$\binom{p}{i} \equiv 0 \bmod p$ for all $1<i<p$,
we get 
\[
\Delta^{p} (\alpha f)= \alpha^{(p)} + \alpha \Delta^{p}(f),
\]
and therefore the $\mathbb{F}_{p}(x)$-homogeneity follows from the fact that $\alpha^{(p)}=0$. 

\begin{prop}\label{prop:cartier order 1}
 The differential equation \eqref{eq order 1 char p for p curv} has a nonzero rational solution if and only if $\Delta^{p}=0$.
\end{prop}

\begin{proof}
 If \eqref{eq order 1 char p for p curv} has a nonzero rational solution $f$, then $\Delta(f)=0$ and, hence, $\Delta^{p}(f)=0$. As $\Delta^{p}$ is $\mathbb{F}_{p}(x)$-linear, we get $\Delta^{p}=0$.  
Conversely, if $\Delta^{p}=0$, then $\Delta$ has a nonzero kernel and, hence, \eqref{eq order 1 char p for p curv} has a nonzero rational solution.   
\end{proof}

\begin{rem}
An easy calculation shows that, if \eqref{eq order 1 char p for p curv} has $p$-curvature $0$, then an explicit nonzero rational solution is given by
\[
\sum_{k=0}^{p-1} (-1)^{k}\frac{x^{k}}{k!} \Delta^{k}(1).
\]
\end{rem}

We conclude this section by giving inductive and closed formulae for the $p$-curvature. 
As it is $\mathbb{F}_{p}(x)$-linear, the $p$-curvature is entirely determined by its value at $1$:
\[
\forall f \in \mathbb{F}_{p}(x), \ \Delta^{p}(f)=\Delta^{p}(1) f.
\]
For this reason, we often say that the $p$-curvature of \eqref{eq order 1 char p for p curv} is $\Delta^{p}(1)$.
It turns out that the $p$-curvature $\Delta^{p}(1)$ can be calculated inductively.
Indeed, for all $k \geq 0$, we denote by $b_k(x) \in \mathbb{F}_{p}(x)$ the constant term of the differential operator $\big(\partial_x+b(x)\big)^{k}$, so that
\begin{equation}
\label{eq:dxbxk}
\big(\partial_x+b(x)\big)^{k} = 
\partial_x^{k}+ \star \partial_x^{k-1}+\cdots  +\star \partial_x + b_{k}(x),
\end{equation}
where $\star$ are some unspecified elements of $\mathbb{F}_{p}(x)$. 
Equating the terms of degree $0$ (with respect to $\partial_x$) in the equality $(\partial_x+b(x))^{k+1}=(\partial_x+b(x)) \cdot (\partial_x+b(x))^{k}$, we get the following inductive formula for computing the $b_{k}(x)$'s:
\begin{equation}\label{eq: inductive formula p-curv ord 1}
\forall k \geq 0, \ b_{k+1}(x)=b_{k}'(x)+b(x)b_{k}(x). 
\end{equation}
This gives the expected inductive formula for the $p$-curvature of \eqref{eq order 1 char p for p curv}, since this is equal to
\[\Delta^{p}(1)= b_{p}(x).\]

\begin{rem} \label{rem:pcurv order 1}
When $k = p$, it is actually possible to determine the coefficients $\star$ in
equation~\eqref{eq:dxbxk}. Indeed, we observe that $(\partial_x+b(x))^p$ is a
central element in $\Fp(x)\langle \partial_x \rangle$. This shows that the
right-hand side of equation~\eqref{eq:dxbxk} must be central as well, implying
eventually that all terms in $\partial_x^i$ with $0 \leq i < p$ have to vanish, 
and that $b_p(x)$ belongs to $\F_p(x^p)$.

In conclusion, we have the relation
\[\big(\partial_x+b(x)\big)^p = \partial_x^p + b_p(x).\]
From this, we deduce that $b_p(x)$ is also the opposite of the remainder in the
right Euclidean division of $\partial_x^p$ by $\op L = \partial_x + b(x)$.
In particular, the $p$-curvature vanishes if and only if $\op L$ divides
$\partial_x^p$ in $\Fp(x)\langle \partial_x \rangle$.

\end{rem}

Last, one can deduce from \eqref{eq: inductive formula p-curv ord 1} the
following remarkable closed formula (that does not extend to higher order
equations).

\begin{theo}
\label{theo:pcurvorder1}
We have $b_{p}(x)=b^{(p-1)}(x)+b(x)^{p}$.
\end{theo}

\begin{proof}[Proof (after Jacobson~\cite{Jacobson37})]
For a positive integer~$k$, let $I_k$ be the set of all
tuples $\underline \alpha = (\alpha_1, \ldots, \alpha_k)$
of nonnegative integers such that $\sum_{i=1}^k i \alpha_i = k$.
A calculation shows that $b_k(x)$ is explicitly given by 
\[b_k(x) = \sum_{\underline \alpha \in I_k} 
\lambda_{\underline \alpha} \cdot
b(x)^{\alpha_1} \cdot b^{(1)}(x)^{\alpha_2} \cdots 
b^{(k-1)}(x)^{\alpha_k},\]
where $\lambda_{\underline \alpha}$ is a coefficient in $\Z$
determined by the following rule
\[\lambda_{\underline \alpha} =
\sum_{i=1}^k (\alpha_{i-1} + 1) \cdot \lambda_{\tau_i(\underline \alpha)}
\qquad \text{(for } \underline \alpha \in I_k \text{)},\]
where $\tau_i$ denotes the function from $I_k$ to $I_{k-1}$ defined by
\[\tau_i(\underline \alpha) = (\alpha_1,\, \ldots,\, \alpha_{i-2},\,
\alpha_{i-1}{-}1,\, \alpha_i{+}1,\, \alpha_{i+1},\, \ldots,\, \alpha_{k-1})\]
and where we agree that $\lambda_{\underline \beta} = 0$ if 
$\underline \beta$ has one negative coordinate. From this relation,
one can check by induction on $k$ that 
$\lambda_{\underline \alpha}$ (with $\underline \alpha = (\alpha_1,
\ldots, \alpha_k) \in I_k$) is given by the closed formula:
\[\lambda_{\underline \alpha} = \frac{k!}{(\alpha_1)! \cdots (\alpha_k)! 
\cdot (2!)^{\alpha_2} \cdot (3!)^{\alpha_3} \cdots (k!)^{\alpha_k}}.\]
In particular, when $k = p$, we find that the $\lambda_{\underline
\alpha}$'s vanish modulo $p$ for all $\underline \alpha \in I_p$ 
(thanks to the numerator $p!$)
except when $\underline \alpha = (p, 0, \ldots, 0)$ or $\underline
\alpha = (0, \ldots, 0, 1)$ (because, in those cases, the numerator
cancels with a factor $p!$ in the denominator). Besides, in both
cases, one finds $\lambda_{\underline \alpha} = 1$. This concludes the proof.
\end{proof}

\begin{remark}
\label{rem:pcurvpfd}
Remarkably, the explicit formula of Theorem~\ref{theo:pcurvorder1}
provides a second proof of Proposition~\ref{prop:rationalorderone char p}.
Indeed, consider a rational function $b(x) \in \Fp(x)$ and write its
partial fraction decomposition
\[b(x) = P(x) + \sum_{i=1}^m \sum_{j=1}^{r_i} \frac{\beta_{i,j}}{(x-b_i)^j}\]
where $P(x)$ is a polynomial, the $b_i$'s are pairwise distinct
elements of $\overline{\Fp}$ and $\beta_{i,j} \in \overline{\Fp}$ with 
$\beta_{i,r_i} \neq 0$. Each $b_i$ is a pole of $b(x)$ of multiplicity~$r_i$ 
and residue $\beta_{i,1}$. Moreover, $b(x)$ has an extra pole at infinity when the
degree of $P(x)$ is positive.
A direct computation now gives:
\[b_p(x) = P^{(p-1)}(x) + P(x)^p 
- \sum_{i=1}^m \sum_{\substack{1 \leq j \leq r_i \\ j \equiv 1 \text{ mod } p}}
\frac{\beta_{i,j}}{(x-b_i)^{j+p-1}}
+ \sum_{i=1}^m \sum_{j=1}^{r_i} \frac{\beta_{i,j}^p}{(x-b_i)^{pj}} \]
and we see that the latter is zero if and only if $P(x)$ vanishes and, 
for all $i \in \{1, \ldots, m\}$, we have $r_i = 1$ and $\beta_{i,1}^p = 
\beta_{i,1}$, \emph{i.e.} $\beta_{i,1} \in \Fp$.
After Proposition~\ref{prop:cartier order 1}, we then recover by different 
means the result of Proposition~\ref{prop:rationalorderone char p}.
\end{remark}

\begin{ex}
\label{ex:exp-arctan2}
Applying the above recipe with the differential operator
$$\op L_p = \partial_x - \frac 1{x^2 + 1}$$
of Example~\ref{ex exp-arctan}, we find that its $p$-curvature is
explicitly given by
$$b_p(x) = - \frac{c_p}{(x^2 + 1)^p}$$
where $c_p = 1$ if $p = 2$, $c_p = 0$ for $p \equiv 1 \pmod 4$ and $c_p = 2$ 
for $p \equiv 3 \pmod 4$. In particular, we retrieve by different 
means the dichotomy that we had already observed in Example~\ref{ex 
exp-arctan}.

The situation can be more complex if we slightly change the coefficient $b(x)$.
For instance, consider the first-order differential operator:
$$\op L_p = \partial_x - \frac{1}{x^3 - x - 1}.$$
Its $p$-curvature is zero if and only if the polynomial $x^3 - x - 1$ splits in $\F_p[x]$ and $p\neq 23$.
By~\cite[\S5.3]{Serre03}, this happens only for the primes $p \in \{ 59, 101, 167, 173, 211, 223, 271, 307, 317, \ldots\}$ that have the property that $p\neq 23$ and they can be written as $p = m^2+mn+6n^2$ with $m,n \in \Z$, or equivalently, if and only if the coefficient of $x^{p-1}$ in $\prod_{n=1}^\infty
(1-x^n)(1-x^{23n})$ is equal to 2.
(See also~\cite[Prop.~3.3]{Put96}.)
\end{ex}

\begin{remark}
There is a link between the $p$-curvature of first-order linear differential operators and a fairly famous algorithm for factoring polynomials in $\F_p[x]$, designed by Niederreiter in~\cite{Niederreiter93}. 
This connection seems to have been unnoticed until now.
To factor a separable polynomial $f = \prod_i g_i$ of $\F_p[x]$, with irreducible $g_i$, Niederreiter considers the space of rational functions $y = h/f$ solutions of the equation $y^{(p-1)} + y^p = 0$, and shows that as a vector space over $\F_p$ it is generated by the logarithmic derivatives $g_i'/g_i$. As a result, factoring boils down to a linear algebra problem over $\F_p$. This algorithm  created a lot of excitement 
as a promising alternative to the much more classical one due to Berlekamp~\cite{Berlekamp67}.
\end{remark}	

\medskip

Putting together Theorem~\ref{theo: from char 0 to p}, 
Theorem~\ref{theo: from char p to 0}, 
Proposition~\ref{prop:cartier order 1}
and Remark~\ref{rem:pcurv order 1},
we obtain the following result.

\begin{theo}
\label{theo:groth-order1}
Let $\op L = \partial_x + a(x)$ as in equation~\eqref{op order 1 char 0}
and, for almost all prime numbers $p$, denote by $\op L_p$ its reduction 
modulo $p$ as in equation~\eqref{op order 1 reduced}.
The following properties are equivalent:
\begin{enumerate}[label=(\arabic{enumi}),topsep=\parsep,itemsep=\parsep,parsep=0pt]
 \item $\op L$ has a nonzero algebraic solution;
 \item for almost all primes $p$, $\op L_p$ has a nonzero rational solution;
 \item for almost all primes $p$, the $p$-curvature of $\op L_p$ vanishes;
 \item for almost all primes $p$, the operator $\op L_p$ divides 
       $\partial_x^p$ in $\Fp(x)\langle \partial_x \rangle$.
\end{enumerate}
\end{theo}

Grothendieck's $p$-curvature conjecture is a far reaching conjectural generalization of these equivalences for higher order equations.

\subsection{The general case}\label{sec:stat Groth conj}
\label{ssec:groth-general}

Let us now consider a linear differential operator of arbitrary order: 
\begin{equation}
\label{eq:opL}
\op L = 
\partial_x^n + a_{n-1}(x) {\cdot}
\partial_x^{n-1} + \cdots + a_1(x) {\cdot}\partial_x + a_0(x)
\end{equation}
with $a_{i}(x) \in \Q(x)$. As in the order-$1$ case, one can consider 
the reduction $\op L_{p}$ of $\op L $ modulo~$p$ for almost all primes
$p$. This is a differential operator of order $n$ with coefficients in 
$\mathbb{F}_{p}(x)$. Grothendieck's conjecture relates the algebraicity 
of the solutions of $\op L$ to the rationality of the solutions of $\op 
L_{p}$ for almost all primes $p$.

First of all, we notice that the straightforward generalization of 
Theorem~\ref{theo:groth-order1} cannot be true for higher order
differential operators; indeed, we have seen 
in \S \ref{sec:examples} many examples of differential equations that 
do not admit algebraic solutions and whose reductions modulo~$p$ have 
nonzero rational solutions for almost all~$p$; this is the case, for 
instance, of most of hypergeometric functions and diagonals.
The main new insight behind Grothendieck's conjecture is the brilliant
idea to replace the existence of a unique nonzero solution by the 
existence of a \emph{full basis} of solutions.

We recall that the set of solutions of $\op L$ in $\overline{\Q(x)}$ is a $\Qbar$-vector space of dimension at most $n$. When this dimension is maximal, that is, equal to $n$, we say that $\op L$ has a full basis of algebraic solutions. 
Similarly, it is tempting to look at the set of solutions of $\op L_{p}$ in $\mathbb{F}_{p}(x)$ as an $\mathbb{F}_{p}$-vector space. However, the example given by the differential equation $y^{(p)}=0$ shows that this vector space may be infinite dimensional (any element of $\mathbb{F}_{p}(x)$ is a solution of $y^{(p)}=0$). The point is that $\Qbar$ is the relevant base field in characteristic $0$ because it is the field of differential constants of $\overline{\Q(x)}$. In characteristic $p$, the field of differential constants of $\mathbb{F}_{p}(x)$ is not $\mathbb{F}_{p}$ but $\mathbb{F}_{p}(x^{p})$; thus, a differential constant may depend on $x$ in characteristic~$p$ (!)
Now, one can prove that the set of solutions of $\op L_{p}$ in $\mathbb{F}_{p}(x)$ is an $\mathbb{F}_{p}(x^{p})$-vector space of dimension at most $n$. When this dimension is maximal, that is, equal to $n$, we say that $\op L_{p}$ has a full basis of rational solutions. 

We are now ready to state Grothendieck's conjecture.

\begin{conj}[Grothendieck's conjecture]\label{groth conj}
For a differential operator $\op L \in \Q(x)\langle \partial_x \rangle$
as in equation~\eqref{eq:opL},
the following properties are equivalent:
\begin{enumerate}[label=(\arabic{enumi}),topsep=\parsep,itemsep=\parsep,parsep=0pt]
 \item $\op L$ has a full basis of algebraic solutions; 
 \item for almost all primes $p$, $\op L_{p}$ has a full basis 
       of rational solutions.
\end{enumerate}
\end{conj}

Consider the linear differential operator
\begin{equation}
\label{eq:opLcharp}
\op L =
\partial_x^n + b_{n-1}(x) {\cdot}
\partial_x^{n-1} + \cdots + b_1(x) {\cdot}\partial_x + b_0(x)
\end{equation}
with $b_{i}(x) \in \mathbb{F}_{p}(x)$. There is no straightforward generalization of Proposition \ref{prop:rationalorderone char p} for determining whether \eqref{eq:opLcharp} has a full basis of rational solutions but the criterion given by Proposition \ref{prop:cartier order 1} via the $p$-curvature does extend to higher order equations. Let us briefly explain this. 

Let $Y'+B(x)Y=0$ be the differential system associated to \eqref{eq:opLcharp}, where 
\begin{equation}
\label{eq:companionopL}
B =
\left(
\begin{array}{cccccc}
0 & -1 & 0 & \cdots & 0 & 0 \\
0 & 0 & -1 & \cdots & 0 & 0 \\
\vdots & \vdots &\vdots &\vdots &\vdots &\vdots \\
0 & 0 & 0 & \cdots & -1 & 0 \\
0 & 0 & 0 & \cdots & 0 & -1 \\
b_{0} & b_{1} & b_{2} & \cdots & b_{n-2} & b_{n-1}
\end{array}\right) \in M_{n}(\mathbb{F}_{p}(x)).
\end{equation}
Mimicking what has been done in Section \ref{sec: p-curv order 1} in the order-$1$ case,  we consider the $\mathbb{F}_{p}(x^p)$-linear map 
\begin{eqnarray*}
 \Delta : \mathbb{F}_{p}(x)^{n} &\rightarrow& \mathbb{F}_{p}(x)^{n} \\ 
 F &\mapsto& F'+B(x) F. 
\end{eqnarray*}

\begin{defi}\label{def:pcurv-general}
The map 
\[
\Delta^{p}: \mathbb{F}_{p}(x)^{n} \rightarrow \mathbb{F}_{p}(x)^{n}
\] 
is called the $p$-curvature of \eqref{eq:opLcharp}.  
\end{defi}

As in the first-order case, one can easily prove that the $p$-curvature is not
only $\mathbb{F}_{p}(x^p)$-linear, but also $\mathbb{F}_{p}(x)$-linear.
Moreover, the inductive formula \eqref{eq: inductive formula p-curv ord 1} for
computing the $p$-curvature of equations of order $1$ can be extended as
follows: the matrix $B_{p}(x)$ of the $p$-curvature with respect to the
canonical basis is given by the recurrence
\begin{equation}\label{rec for p curv}
B_{k+1}(x)=B_{k}'(x)+B(x)B_{k}(x)
\end{equation}
starting with $B_{0}(x)=B(x)$.

The following fundamental result is a generalization of 
Proposition~\ref{prop:cartier order 1} to higher order differential
equations.
We recall the fact, already mentioned at the very beginning of 
Section~\ref{sec:stat Groth conj}, that the set of solutions of a given 
$\op L \in \Fp(x)\langle \partial_x \rangle$ in $\mathbb{F}_{p}(x)$ of 
order $n$ is an $\mathbb{F}_{p}(x^{p})$-vector space of dimension at 
most $n$ and that, when this dimension is maximal, that is, equal to 
$n$, we say that $\op L$ has a full basis of rational solutions.

\begin{theo}[Cartier's lemma]\label{prop:cartier lemma}
Let $\op L \in \Fp(x)\langle \partial_x \rangle$ be a differential
operator as in equation~\eqref{eq:opLcharp}.
The following properties are equivalent:
\begin{enumerate}[label=(\arabic*),topsep=\parsep,itemsep=\parsep,parsep=0pt]
 \item \label{item 1 proof Cartier} $\op L$ has a full basis of rational solutions; 
 \item \label{item 2 proof Cartier} the $p$-curvature of $\op L$ (that is $\Delta^p$) vanishes;
 \item \label{item 3 proof Cartier} $\op L$ divides $\partial_x^p$ in $\Fp(x)\langle\partial_x\rangle$.
\end{enumerate}
\end{theo}

\begin{proof}
Let us first note that the following properties, relative to the $\mathbb{F}_{p}(x^p)$-vector space $S\coloneqq \ker(\Delta)$, are equivalent: 
\begin{enumerate}[label=\roman*)]
\item \label{item 1 claim in proof Cartier} the differential equation \eqref{eq:opLcharp} has a full basis of rational solutions; 
\item \label{item 2 claim in proof Cartier} the $\mathbb{F}_{p}(x^p)$-vector space $S$ has dimension $n$;
\item \label{item 3 claim in proof Cartier} the $\mathbb{F}_{p}(x)$-vector space $\mathbb{F}_{p}(x)^{n}$ is spanned by $S$.
\end{enumerate}

The equivalence between \ref{item 1 claim in proof Cartier} and 
\ref{item 2 claim in proof Cartier} follows immediately from the easily 
verifiable fact that the map
$$
f(x) \mapsto (f(x), f'(x), \ldots, f^{(n-1)}(x))
$$ 
induces an $\mathbb{F}_{p}(x^p)$-linear isomorphism from the 
$\mathbb{F}_{p}(x^p)$-vector space of solutions of $\op L$ in 
$\mathbb{F}_{p}(x)$ to the $\mathbb{F}_{p}(x^p)$-vector space $S$. The 
equivalence between \ref{item 2 claim in proof Cartier} and 
\ref{item 3 claim in proof Cartier} follows from the 
``Wronskian lemma'' \cite[Lemma 1.12]{PutSinger03}. 
Indeed, the wronskian lemma 
ensures that any family of elements of $S$ is linearly dependent over 
$\mathbb{F}_{p}(x^p)$ if and only if it is linearly dependent over 
$\mathbb{F}_{p}(x)$. Therefore, the dimension of the 
$\mathbb{F}_{p}(x^p)$-vector space $S$ and the dimension 
of the $\mathbb{F}_{p}(x)$-vector space spanned by $S$ are equal. Considering 
the case where one or other of these dimensions is $n$, we obtain the 
equivalence between \ref{item 2 claim in proof Cartier} and 
\ref{item 3 claim in proof Cartier}.

We are now ready to prove the theorem.

Let us first prove \ref{item 1 proof Cartier}$\Longrightarrow$\ref{item 2 proof Cartier}.
If $\op L$ has a full basis of rational solutions, then the implication \ref{item 1 claim in proof Cartier}$\Longrightarrow$\ref{item 3 claim in proof Cartier} ensures that the $\mathbb{F}_{p}(x)$-vector space $\mathbb{F}_{p}(x)^{n}$ is spanned by $S$. Since $\Delta^{p}$ is $\mathbb{F}_{p}(x)$-linear and vanishes on $S$, we have $\Delta^{p}=0$.

Let us now prove \ref{item 2 proof Cartier}$\Longrightarrow$\ref{item 1 proof Cartier}.
We assume that $\Delta^{p}=0$. We claim that the $\mathbb{F}_{p}(x)$-vector space $\mathbb{F}_{p}(x)^{n}$ is spanned by $S$.
Consider the map 
\begin{eqnarray*}
 P : \mathbb{F}_{p}(x)^{n} &\rightarrow& \mathbb{F}_{p}(x)^{n} \\ 
 F &\mapsto& \sum_{k=0}^{p-1} (-1)^{k}\frac{x^{k}}{k!} \Delta^{k}(F). 
\end{eqnarray*}
A simple calculation shows that  
\[
\Delta(P(F))=-(-x)^{p-1}\Delta^{p}(F)=0 
\]
and, hence, $P$ has values in $S$. 
But, another simple calculation shows that, for all $F \in \mathbb{F}_{p}(x)^{n}$, we have 
\[
F=\sum_{k=0}^{p-1} \frac{x^{k}}{k!} P(\Delta^{k}(F)).
\]
This shows that the  $\mathbb{F}_{p}(x)$-vector space $\mathbb{F}_{p}(x)^{n}$ is spanned by $S$ as claimed.
Now, using the implication \ref{item 3 claim in proof Cartier}$\Longrightarrow$\ref{item 1 claim in proof Cartier}, we get that $\op L$ has a full basis of rational solutions.  

It remains to prove that \ref{item 2 proof Cartier}$\Longleftrightarrow$\ref{item 3 proof Cartier}. In order to do so, we first notice that, given rational functions $f_0(x), \ldots, 
f_{n-1}(x), g_0(x), \ldots, g_{n-1}(x)$, the equality
\[\Delta\big(f_0(x), \ldots, f_{n-1}(x)\big) 
= \big(g_0(x), \ldots, g_{n-1}(x)\big)\]
is equivalent to the following congruence in $\Fp(x)\langle\partial_x
\rangle$:
\[\big(f_0(x) + \cdots + f_{n-1}(x)\partial_x^{n-1}\big) \cdot \partial_x
\equiv g_0(x) + \cdots + g_{n-1}(x)\partial_x^{n-1} \pmod{\op L}.\]
It follows from this observation that, writing $E_i = (0, \ldots, 0, 
1, 0, \ldots, 0)$ with the coordinate $1$ in $i$-th position, the 
coordinates of $\Delta^p(E_i)$ are exactly the coefficients of the 
remainder in the division of $\partial_x^{p+i}$ by $\op L$.
Hence $\Delta^p(E_i)$ vanishes if and only if $\op L$ divides
$\partial_x^{p+i}$. 
The equivalence \ref{item 2 proof Cartier}$\Longleftrightarrow$\ref{item 3 proof Cartier} follows immediately.
\end{proof}

\begin{rem}
The three assertions of \cref{prop:cartier lemma} are also equivalent to 
\begin{enumerate}[label=(\arabic{enumi}),topsep=\parsep,itemsep=\parsep,parsep=0pt]
 \item[(4)] $\op L$ admits $n$ power series solutions in $\F_p[[x]]$, linearly independent  over $\F_p((x^p))$;
 \item[(5)] $\op L$ admits $n$ polynomial solutions in $\F_p[x]$, linearly independent  over $\F_p((x^p))$.
\end{enumerate}
The implication (4) $\Longrightarrow$ (5) is proved in~\cite[Lemma~1]{Honda79}, while (5) $\Longrightarrow$ (1) and (1) $\Longrightarrow$ (4) are trivial.
Moreover, under the equivalent assertions (1)--(5), Proposition~1 in \cite{BoSc09} shows that there exists a full basis of polynomial solutions in $\F_p[x]$, each of them having degree less than~$pd$, where $d$ is the maximal degree of the numerators/denominators of the coefficients $b_i(x)$ of~$\op L$ in  \eqref{eq:opLcharp}.
\end{rem}

\begin{rem}
Assume that $\op L$ has $p$-curvature zero. An easy calculation shows that 
\[
U_0(x)=\sum_{k=0}^{p-1} (-1)^{k}\frac{x^{k}}{k!} B_k(x) \in M_n(\mathbb{F}_p(x)) 
\]
is a solution of $Y'+B(x)Y=0$. If, moreover, $B(x)$ has no pole at $0$, then $U_0(x)$ has no pole at $0$ as well and we have $U_0(0)=I_n$, so $U_0(x)$ is a fundamental matrix of rational  solutions of  $Y'+B(x)Y=0$. If $B(x)$ has a pole at $0$, then $U_0(x)$ is not necessarily invertible. Note that, more generally, if $a\in\mathbb{F}_{p}$ is not a pole of $B(x)$, then 
\[
U_a(x)=\sum_{k=0}^{p-1} (-1)^{k}\frac{(x-a)^{k}}{k!} B_k(x-a) 
\]
is a fundamental matrix of rational solutions of  $Y'+B(x)Y=0$.
\end{rem}

Putting together all that precedes, we obtain a simple algorithm to determine whether 
\eqref{eq:opLcharp} has a full basis of rational solutions: compute inductively $B_{p}(x)$ and, then, check whether $B_{p}(x)$ vanishes.
Note however that no extension of the simple formula of Theorem~\ref{theo:pcurvorder1} is known for higher order differential equations.
Roughly speaking, this is due to the fact that, contrarily to $\mathbb{F}_{p}(x)$, the
ring of $n \times n$ matrices over $\mathbb{F}_{p}(x)$ is noncommutative as soon
as $n \geq 2$. 
Computing the $p$-curvature is then much more complicated in this case but
rather efficient algorithms for this task are nevertheless available; we will
discuss them in Section~\ref{ssec:algorithms}.

Using Theorem \ref{prop:cartier lemma} (Cartier's lemma), we get the following reformulation of Grothendieck's conjecture. 

\begin{conj}[Grothendieck's conjecture in terms of $p$-curvature]\label{groth conj in terms p-curv}
For a differential operator~$\op L$ as in equation~\eqref{eq:opL},
the following properties are equivalent:
\begin{enumerate}[label=(\arabic{enumi}),topsep=\parsep,itemsep=\parsep,parsep=0pt]
 \item $\op L$ has a full basis of algebraic solutions; 
 \item for almost all primes $p$, the $p$-curvature of $\op L_{p}$ vanishes;
 \item for almost all primes $p$, $\op L_{p}$ divides $\partial_x^p$
       in the ring of differential operators $\Fp(x)\langle\partial_x\rangle$.
\end{enumerate}
\end{conj}

\subsection{Progresses toward Grothendieck's conjecture} \label{ssec:progressGK}

\subsubsection{A known case: the generalized hypergeometric equations}

It is in general very difficult to determine whether a given differential
equation has a full basis of algebraic solutions. In their celebrated work~\cite{BeHe89}, Beukers
and Heckman managed to do this for an important class of differential
equations, omnipresent in the mathematical and physical literature, namely the 
\emph{generalized hypergeometric equations}. 
In fact, Beukers and Heckman extended the Landau-Errera criterion mentioned in~\S\ref{sec:hypergeom}.
Let ${\bf a} = ( a_1,\ldots, a_{s+1} )$ and ${\bf b} = ( b_1,\ldots, b_{s}, b_{s+1}=1 )$
be two $(s+1)$-tuples of rational parameters, assumed disjoint modulo~$\mathbb{Z}$.
This assumption is equivalent to the irreducibility in $\Q(x)\langle\partial\rangle$ of the ``generalized hypergeometric operator'' defined by
\[\hyp{{\bf a},{\bf b}} \coloneqq (x\partial_x+b_1-1)\cdots (x\partial_x+b_{s}-1) x\partial_x - x (x\partial_x+a_1)\cdots (x\partial_x+a_{s+1}).\] 
It is easy to check that $\hyp{{\bf a},{\bf b}}$  admits in its solution space
the generalized hypergeometric function ${}_{s+1} F_s ([a_1, \ldots, a_{s+1}], [b_1, \ldots, b_s]; x)$ defined in~\eqref{deq:sFs}.
The Beukers-Heckman result then reads as follows.

\begin{theo}[``interlacing criterion'', Beukers-Heckman, \cite{BeHe89}]
\label{theo:beukers and heckmann}
Given two $(s+1)$-tuples of rational numbers ${\bf a} = ( a_1,\ldots, a_{s+1} )$ and ${\bf b} = ( b_1,\ldots, b_{s}, b_{s+1}=1 )$, assumed to be disjoint modulo~$\mathbb{Z}$,
let $D$ be the common denominator of their elements.
Then, the following assertions are equivalent:
\begin{enumerate}
\item the hypergeometric function ${}_{s+1} F_s ([a_1, \ldots, a_{s+1}], [b_1, \ldots, b_s]; x)$ is algebraic;	
\item the operator $\hyp{{\bf a},{\bf b}}$ admits a full basis of algebraic solutions;
\item for all $1\leq \ell < D$ with $\gcd(\ell, D) = 1$ the $(s+1)$-tuples
$(e^{2 \pi i \ell a_j})_{1 \leq j \leq s+1}$ and $(e^{2 \pi i \ell b_j})_{1 \leq j \leq s+1}$
interlace on the unit circle.
\end{enumerate}
\end{theo}  

The interlacing condition mentioned in the latter result is defined as follows.  We say that two $(s+1)$-tuples $(u_j)_{1 \leq j \leq s+1}$ and $(v_j)_{1 \leq j \leq s+1}$ of elements of the unit circle interlace if, up to renumbering the $u_j$ and the $v_j$, we have $u_1=e^{2 \pi i \alpha_1},\ldots,u_{s+1}=e^{2 \pi i \alpha_{s+1}}$ and $v_1=e^{2 \pi i \beta_1},\ldots,v_{s+1}=e^{2 \pi i \beta_{s+1}}$ with either 
$$
0\leq \alpha_1 < \beta_1 < \alpha_2 < \beta_2 \leq \cdots < \alpha_{s+1} < \beta_{s+1} <1
$$
or
$$
0\leq \beta_1 < \alpha_1 < \beta_2 < \alpha_2 \leq \cdots < \beta_{s+1} < \alpha_{s+1} <1. 
$$

\begin{ex}
The Beukers-Heckman criterion immediately implies that the
operator
$\hyp{{\bf a},{\bf b}}$ admits a full basis of algebraic solutions for the choice
\[{\bf a} = \left\{ \frac{1}{30}, \frac{7}{30}, 
\frac{11}{30}, \frac{13}{30}, \frac{17}{30}, 
\frac{19}{30}, \frac{23}{30}, \frac{29}{30} \right\},
\quad
{\bf b} = \left\{
\frac15, \frac13, \frac25, \frac12, \frac35, \frac23, \frac45, 1 \right\}.\]
This proves in particular a beautiful observation due to Fernando Rodriguez-Villegas, namely that the generating function $\sum_{n \geq 0} u_n x^n $ of the  
sequence 
\[ u_n \coloneqq \frac{(30n)!n!}{(15n)!(10n)!(6n)!} \]
(used by Chebyshev in his work on estimates for the prime counting function) is algebraic.
\end{ex}

Note that without the irreducibility assumption on $\hyp{{\bf a},{\bf b}}$, the situation is much more subtle.
For instance, 
${}_2F_1([1/2,1/3],[3/2];x)$ is transcendental, while
${}_2F_1([3/2,1/3],[1/2];x)$ is algebraic.
In a recent work by Fürnsinn and Yurkevich~\cite{FuYu23},
a generalization of \cref{theo:beukers and heckmann}
is given, which allows to decide algebraicity/transcendence of arbitrary generalized hypergeometric functions (with potentially irrational parameters).

In addition to~\cref{theo:beukers and heckmann}, Beukers and Heckman also drew up in \cite{BeHe89} the list of generalized
hypergeometric equations having a full basis of algebraic solutions,
thus extending Schwarz's classification of algebraic ${}_2F_1$'s~\cite{Schwarz1873}.

On the other hand, a calculation due to Katz in \cite[\S5.5]{Katz90} (also in \cite[\S6]{Katz72} for the specific case of ${}_2F_1$'s) shows that this list coincides with the list of generalized hypergeometric equations whose reductions modulo $p$ have a full basis of rational solutions for almost all primes $p$, in accordance with Grothendieck's conjecture.

\subsubsection{State of the art on Grothendieck's conjecture}

Besides for order-$1$ equations and for generalized hypergeometric equations, Grothendieck's conjecture has been proved in several particular cases. 

On the one hand, for Picard-Fuchs differential
equations (satisfied by periods of a family of smooth algebraic varieties),
and more generally for certain direct factors, Grothendieck's conjecture was established by Katz~\cite{Katz72}. As an application, Katz gave in \cite[Theorem~5.5.3]{Katz90} a new proof of the aforementioned results of Beukers and Heckman~\cite{BeHe89} 
about the  generalized hypergeometric equations.  Katz~\cite[\S1]{Katz72}, and later
Andr\'e~\cite[\S{III}]{Andre04}, related the $p$-curvatures to the reduction modulo $p$
of the 
\emph{Kodaira-Spencer map}. (See also
Foucault~\cite{Foucault92} and Foucault and Toffin~\cite{FoTo07} for explicit
computations for families of curves of genus~2 and~3.) 
As explained in~\cite[p.~108]{Andre04}, this approach has a
potential of delivering effective versions of Grothendieck's conjecture,
similar to effective versions of Chebotarev's density theorem~\cite{LaOd77,Serre81}: the hope is to obtain, for instance for any Picard-Fuchs operator~$\op L$, an integer $N(\op L)$ such that  
the fact that 
$\op L$ has a full basis of algebraic solutions can be read on the $p$-curvatures of $\op L$ for the primes $p<N(\op L)$.

On the other hand, an arithmetic approach to Grothendieck's conjecture was introduced by the Chudnovsky
brothers~\cite{ChCh85} who proved Grothendieck's conjecture for any rank one
linear homogeneous differential equation over an algebraic curve~\cite[Theorem 8.1]{ChCh85} (the case of
first order equations over $\P^1$ had been proved by Honda in~\cite[\S1]{Honda79}, see Theorem~\ref{theo:Honda}).
They also proved Grothendieck's conjecture for the class of Lamé equations~\cite[Theorem 7.2]{ChCh85}, of the form
\[p(x)y''(x) + \frac12 p'(x) y'(x) - \big(n(n+1)x+B\big){\cdot}y(x)=0\]
where $n\in\N$, $B\in\Q$ and $p(x)\in\Q[x]$ has degree 3.
The arithmetic approach was extended by Andr\'e to the case when the differential Galois
group has a solvable neutral component~\cite{Andre04} (see also~\cite{Andre87}, 
\cite[Chap.~VIII]{Andre89}, \cite[Thm.~2.9]{Bost01} and~\cite[Thm.~3.5]{Chambert-Loir02}).

Katz~\cite{Katz82} proposed a conjectural description of the differential Galois group in terms of $p$-curvatures and he proved that his conjecture is equivalent to the initial conjecture by Grothendieck.
	
Using the language of schemes and sheaves, Grothendieck's conjecture 
can be formulated more generally for differential equations over any
algebraic smooth curve defined over a number field.
In \cite[Remark~7.1.4]{Andre04}, André noticed 
that, using Belyi maps, one can reduce the general case to that of
the curve $X = \P^1\backslash\{0, 1, \infty\}$. In our setting, this
means that one can safely assume that the differential operator $\op L$
has only singularities at $0$, $1$ and $\infty$.
Under this additional assumption, Tang~\cite{Tang18} proves that if 
\emph{all}\footnote{When $\op L$ does not reduce properly at a prime
$p$, the $p$-curvature of $\op L_p$ is \emph{a priori} not defined; 
however Tang manages to give an alternative definition of the vanishing 
of the $p$-curvature, see~\cite[Definition~2.1.7]{Tang18}.} the 
$p$-curvatures 
of $\op L$ vanish, then $\op L$ has a full basis of \emph{rational}
solutions. Although this latter result differs from Grothendieck's 
in the hypotheses (which are stronger) and the conclusion (which is 
also stronger), it is closely related.

We also point out the work of Bost in \cite{Bost01} giving an  algebraicity criterion for leaves of algebraic foliations defined over number fields. For additional details, we refer to \cite{Chambert-Loir02}.
We mention the work of van der Put in \cite{Put01} concerned with inhomogeneous equations of order $1$.  
Other special cases of the conjecture have been proven recently, see~\cite{FaKi09,Shankar18,PaShWh21}.
Last but not least, an analogue of Grothendieck's conjecture for $q$-difference equations was conjectured by Bézivin~\cite[\S5]{Bezivin91} and 
proved by Di Vizio in~\cite{DiVizio02}.

\subsection{A formal parallel with Kronecker's theorem}

It is instructive to observe that Grothendieck's conjecture appears to be, in some sense, a
differential version of Kronecker's theorem we have already encountered earlier (see
Theorem~\ref{theo:kronecker}). Indeed, Kronecker's theorem can be reformulated as follows.

\begin{theo}
\label{theo:kronecker2}
For a separable polynomial $L \in \Q[x]$,
the following conditions are equivalent:
\begin{enumerate}[label=(\arabic{enumi}),topsep=\parsep,itemsep=\parsep,parsep=0pt]
\item all the roots of $L$ are in $\Q$;
\item for almost all primes~$p$, all the roots of $L \text{ mod } p$
are in $\Fp$;
\item for almost all primes~$p$, we have $X^p \equiv X \pmod{L,p}$.
\end{enumerate}
\end{theo}

It is striking that the three conditions of Theorem~\ref{theo:kronecker2} are formal analogues
of the conditions of Conjecture~\ref{groth conj}, at least if we admit that algebraic solutions
in the differential case correspond to rational solutions in the algebraic case. Besides, the
fact that the condition $X^p \equiv X \pmod{L,p}$ translates to $\partial_x^p \equiv 0
\pmod{\op L, p}$, \emph{i.e.} that the right-hand side shifts from $X$ to $0$, is explained by
the fact that the classical Frobenius map behaves ``multiplicatively'' (it belongs naturally to some
Galois group) while the $p$-curvature behaves ``additively'' (it belongs naturally to some Lie
algebra).

In the classical setting, Kronecker's theorem is obtained as a corollary of Chebotarev's
density theorem, which is itself proved by means of Artin's $L$-functions. Unfortunately,
similar tools do not seem to be available so far in the differential context; developing them
might then sound as an exciting project.

As mentioned above, Honda proved that the Grothendieck conjecture for first order
differential equations is equivalent to Kronecker's theorem. In
\cite[\S4]{ChCh85}, the Chudnovsky brothers gave an elementary (although
``extravagant'') proof of these equivalent statements; their approach is based
on Hermite's explicit Hermite-Padé approximants to binomial functions. More
precisely, they proved that if $y'(x) = \frac{x}{\alpha} y(x)$ has zero
$p$-curvature for almost all primes~$p$, then for all primes ideals
$\mathfrak{P}$ of $\Q(\alpha)$ all the binomial coefficients
$\binom{\alpha}{n}$ are $\mathfrak{P}$-integral for all $n$. From there, it is
shown that Hermite-Padé approximants to $1, x^\alpha, \ldots, x^{(m-1)\alpha}$
at $x=1$ with weights $(N,\ldots,N)$ are trivial for large $m$ and~$N$. This in
turn implies that $1, x^\alpha, \ldots, x^{(m-1)\alpha}$ are linearly dependent
over $\Q(x)$, that is $x^\alpha$ is an algebraic function, which is equivalent
to $\alpha\in\Q$.

\section{About the computation of the $p$-curvature}
\label{ssec:algorithms}

After Cartier's lemma (Theorem~\ref{prop:cartier lemma}) and 
Grothendieck's conjecture (Conjecture~\ref{groth conj}), it is clear that the $p$-curvature
is an invariant of primary importance of linear differential equations
in characteristic~$p$.
In this section, we outline some algorithms for computing it (or
other quantities associated to it) efficiently.

\subsection{Operators of order $1$}

In the case of differential equations of order~$1$ of the form
\eqref{eq order 1 char p for p curv}:
\[y' + b(x) y = 0\]
we have seen in Theorem~\ref{theo:pcurvorder1} that the $p$-curvature is 
explicitly given by the formula 
\[b_p(x) = b^{(p-1)}(x) + b(x)^p.\]
Furthermore,
computing explicitly the latter is quite an easy task which directly reduces
to writing down the partial fraction decomposition of $b(x)$.
Indeed, we have seen that if $b(x)$ decomposes as
\[b(x) = P(x) + \sum_{i=1}^m \sum_{j=1}^{r_i} \frac{\alpha_{i,j}}{(x-a_i)^j}\]
then
\[b_p(x) = P^{(p-1)}(x) + P(x)^p 
- \sum_{i=1}^m \sum_{\substack{1 \leq j \leq r_i \\ j \equiv 1 \text{ mod } p}}
\frac{\alpha_{i,j}}{(x-a_i)^{j+p-1}}
+ \sum_{i=1}^m \sum_{j=1}^{r_i} \frac{\alpha_{i,j}^p}{(x-a_i)^{pj}}.\]
Importantly for algorithmic purposes, we observe that the size of
$b_p(x)$ is roughly the same as the size of the input $b(x)$, although
the degree of (the numerator and the denominator of) the former is $p$
times larger the degree of the latter.
This apparent contradiction is explained by the fact that $b_p(x)$ is
actually a function of $x^p$; it is then a sparse rational function,
in the sense that a large proportion of its coefficients vanish.

\begin{rem}
\label{rem:algo pcurv order 1}
Another option for computing explicitly the $p$-curvature of 
differential operators of order~$1$ is presented 
in~\cite[Thm.~2]{BoSc09}; it avoids the computation of partial
fraction decomposition and the factorization of the denominator of $b(x)$.
Let us briefly sketch it with the differential operator
$$\partial_x - \frac{1}{x^2 + 1}$$
of Example~\ref{ex exp-arctan}.
We write $b(x) = -1/(x^2 + 1)$ and assume $p > 2$ for simplicity.
In order to compute $b_p(x)$, we expand $b(x)$ in power series:
\[b(x) = - \sum_{n=0}^\infty (-1)^n x^{2n}.\]
The $(p{-}1)$-st derivative of $x^{2n}$ is $0$ when $2n \not\equiv
-1 \pmod p$, and it is $-x^{2n-p+2}$ otherwise thanks to Wilson's
theorem. Writing $2n = p - 1 + pk$ and noticing that $k$ has to be
even, $k = 2\ell$, we end up with
\[b^{(p-1)}(x) = - \sum_{\ell = 0}^\infty (-1)^{\ell - \frac{p-1}2} x^{2\ell p}\]
On the other hand, it is clear that
\( b(x)^p = - \sum_{n = 0}^\infty (-1)^n x^{2np}.\)
Adding both sums, we find
\[b_p(x) = - \sum_{n = 0}^\infty (-1)^n \cdot \left(1 - (-1)^{\frac{p-1}2}\right) \cdot x^{2np}.\]
When $p \equiv 1 \pmod 4$, the exponent $\frac {p-1} 2$ is even and the term in the parenthesis vanishes. Therefore $b_p(x) = 0$ in this case.
On the contrary, when $p \equiv 3 \pmod 4$, we have
\[b_p(x) = - 2 \cdot \sum_{n = 0}^\infty (-1)^n \cdot x^{2np} 
= - \frac 2{1+x^{2p}}
= - \frac 2 {(1+x^2)^p}.
\]
and we recover the result of Example~\ref{ex:exp-arctan2}.

The same idea applies actually to any differential operator $\op L
= \partial_x - b(x)$. Indeed, as already noticed its $p$-curvature is
a rational function in $x^p$. Besides, it is of the form
$f(x^p)/\text{denom}(b(x))^p$, where $f(x)$ is 
a polynomial of degree at most $d=\deg(b(x))$. Hence, it is enough to 
determine the power series expansion of $(b(x)^{(p-1)})^{1/p}$ at 
precision $x^d$, starting from the power series expansion of $b(x)$ at 
the same precision~$d$.
If $b(x) = \sum_{n \geq 0} u_n x^n$, then by Wilson's theorem we have
$(b(x)^{(p-1)})^{1/p} = -\sum_{n \geq 1} u_{np-1} x^{n-1}$, and hence it is
enough to be able to compute the terms $u_{p-1}, \ldots, u_{d p-1}$.
Since $b(x)$ is a rational function, the sequence $(u_n)_{n\geq 0}$ satisfies a
linear recursion of order at most~$d$, with coefficients in $\F_p$ (given by
the coefficients of $\text{denom}(b)$). As the $N$-th term of such a linear
recurrence with \emph{constant} coefficients can be computed using $O(d \log
(d)\log(N))$ operations in $\F_p$ using the technique of \emph{binary powering} 
combined with fast polynomial multiplication in $\F_p[x]$,
we conclude that the $p$-curvature $b_p(x)$ can be computed by an algorithm that
uses $O(d^2 \log(d) \log(p))$ operations in~$\F_p$.

Note that the reasoning above shows that the $p$-curvature $b_p(x)$ of 
$\op L = \partial_x - b(x)$ is zero if and only if the following 
infinite systems of congruences holds
\[ u_n \equiv u_{(n+1)p-1}\pmod p \quad \text{for all} \; n \geq 0.\]
\end{rem}

\subsection{Reading the $p$-curvature on the solutions}
\label{ssec:pcurvandsol}

For differential equations of higher orders, the sparsity of the 
$p$-curvature no longer holds in general. However, we have the following 
result.

\begin{prop}
\label{prop:sizepcurv}
Let 
\[\op L =
\partial_x^n + b_{n-1}(x) {\cdot}
\partial_x^{n-1} + \cdots + b_1(x) {\cdot}\partial_x + b_0(x)\]
with $b_{i}(x) \in \mathbb{F}_{p}(x)$ and let $B(x)$ be the associated
companion matrix (see equation~\eqref{eq:companionopL}).
Let $f(x) \in \Fp[x]$ be a common denominator of the $b_i(x)$'s
and
\[d = \max\big(\deg f(x),\, \deg(f(x)b_0(x)),\, \ldots,\, 
  \deg(f(x) b_{n-1}(x))\big).\]
Let $B_p(x)$ be the matrix of the $p$-curvature of $\op L$
defined by the recurrence~\eqref{rec for p curv}. Then, the following 
holds.
\begin{itemize}
\item[(i)] The matrix $B_p(x)$ has the form $\frac 1{f(x)^p} C_p(x)$ where
the entries of $C_p(x)$ are all polynomials of degree at most $dp$.
\item[(ii)] The matrix $B_p(x)$ is similar to a matrix with coefficients
in $\Fp(x^p)$.
\end{itemize}
\end{prop}

The first assertion of the proposition follows easily from the 
induction formula~\eqref{rec for p curv} (see~\cite[Lemma~1]{BoSc09}).
The assertion~(ii) is more subtle and requires the construction of a
differential extension of $\mathbb F_p(x)$ over which $\op L$ has a full 
basis of solutions. We refer to~\cite[Proposition~2.1.2]{Dwork90} for a
detailed proof.

Besides, we notice that (ii)~implies that the trace, the determinant 
of $B_p(x)$ and, more generally, all the coefficients of its 
characteristic polynomial lie in $\Fp(x^p)$. This latter statement 
can be also seen as a consequence of the following easy lemma: the
determinant of a linear mapping $\Fp(x)^n \to \Fp(x)^n$ commuting
with $\nabla : F \mapsto F' + B(x) F$ has zero derivative. This
alternative proof has the advantage to avoid going to an extension.
We derive from what precedes that the sizes of the coefficients 
of the characteristic polynomial of $B_p(x)$ are comparable to the sizes 
of the $b_i(x)$'s, although the size of the $p$-curvature itself is in 
general $p$ times larger.

In order to design fast algorithms for computing the $p$-curvature, 
it is useful to go beyond Cartier's lemma 
(see~\cref{prop:cartier lemma}) and relate the $p$-curvature
to the shape of solutions. Let
\[\op L = \partial_x^n + b_{n-1}(x) {\cdot} 
\partial_x^{n-1} + \cdots + b_1(x) {\cdot}\partial_x + 
b_0(x)\] 
be a differential operator as before. In \S 
\ref{ssec:groth-general}, we have seen that, when the $p$-curvature of $\op L$ 
vanishes and the $b_i(x)$'s have no pole at $0$, a fundamental system of
solutions of $\op L$ is explicitly given by 
\[\sum_{k=0}^{p-1} (-1)^{k} B_k(x) \frac{x^{k}}{k!}\] 
where the $B_k(x)$'s are the matrices defined by the 
recurrence~\eqref{rec for p curv}. In full generality, \emph{i.e.} 
without assuming the vanishing of $B_p$, the idea is to consider the 
formal expansion
\[\sum_{k=0}^{\infty} (-1)^{k} B_k(x) \frac{x^{k}}{k!}.\] 
Of course, this does not make sense in $\Fp(x)$ because of the division 
by $k!$, but we shall see that it does make sense in a suitable ring. A 
natural idea to achieve this goal is to introduce divided powers, 
\emph{i.e.} to consider the 
ring of \emph{Hurwitz series}, denoted by 
$\Fp[[x]]^\DP$, whose elements are formal series of the form 
\[a_0 + a_1 \gamma_1(x) + a_2 \gamma_2(x) + \cdots + a_k \gamma_k(x) + 
\cdots .\]
In the above expression, the $\gamma_k(x)$'s are formal names 
without further additional meaning. Of course, they should be thought of 
as $\frac{x^k}{k!}$ but we cannot write this division because the 
denominator may vanish. The multiplication on $\Fp[[x]]^\DP$ is governed 
by the rule
\[\gamma_m(x) \cdot \gamma_n(x) = \binom{m+n}m \cdot \gamma_{n+m}(x)\]
for any nonnegative integers $m$ and $n$. Besides $\Fp[[x]]^\DP$ is 
equipped with a natural derivation that takes $\sum_k a_k \gamma_k(x)$ 
to $\sum_k a_{k+1} \gamma_k(x)$. We have to be careful however that 
$\Fp[[x]]^\DP$ is not a domain (\emph{e.g.} $\gamma_1(x)^p = 0$) 
and, because of that, 
we cannot consider its ring of fractions. But still, if the matrix $B(x)$ 
has polynomial coefficients, all the $B_k(x)$'s have the same property and 
we can consider their images $B_k^\DP(x)$ is the ring 
$M_n\big(\Fp[[x]]^\DP\big)$. We then can form
\begin{equation}
\label{eq:Sdp}
S^\DP(x) = \sum_{k=0}^\infty (-1)^k B_k^\DP(x) \cdot \gamma_k(x) 
\end{equation} 
obtaining this way a fundamental matrix of solutions of $\op L$ over 
$\Fp[[x]]^\DP$. This construction works actually more 
generally as soon as the entries of $B(x)$ have no pole at zero: in this 
case, we can expand them as series in~$x$ in order to view them in 
$\Fp[[x]]^\DP$. The precise relation between $S^\DP(x)$ and the 
$p$-curvature is given by the next lemma.

\begin{lem}[Bostan--Caruso--Schost~\cite{BoCaSc15}]
\label{lem:dpSDP}
We have the matrix relation:
\[\frac{d^p S^\DP(x)}{dx^p} = - B_p^\DP(x) \cdot S^\DP(x). \]
\end{lem}

\begin{proof}
Set $M = \Fp[x]$ and let
$\Delta : M^n \to M^n, Y \mapsto \frac{dY}{dx} + B(x) Y.$
By definition of the $p$-curvature, $\Delta^p$ is the
multiplication by $B_p(x)$.

Now consider the endomorphism $\Delta^\DP$ of $M^\DP = 
\Fp[[x]]^\DP \otimes_{\Fp[x]} M$ defined by
\[\Delta^\DP = \frac d{dx} \otimes \id_M + 1 \otimes \Delta_M.\]
One checks that it satisfies the Leibniz rule: for $f \in \Fp[[x]]^\DP$
and $m \in M$, we have
\[\Delta^\DP(f \otimes m) = \frac{df}{dx} \otimes m + f \otimes \Delta(m).\]
Hence, raising it to the $p$-th power, we obtain
\[(\Delta^\DP)^p(Y^\DP)
= \frac {d^p Y^\DP}{dx^p} + B_p^\DP(x) \cdot Y^\DP\]
for all vectors $Y^\DP \in M^\DP$.
The equality of the lemma follows given that the columns of $S^\DP(x)$
map to $0$ under $\Delta^\DP$.
\end{proof}

\begin{remark}
Hurwitz series provide a framework in which the Picard–Lindelöf (or, Cauchy-Lipschitz) theorem holds
for linear differential equations in characteristic~$p$. The recent article~\cite{FuHa23} provides
another construction allowing this theorem, which has the advantage of yielding an integral domain but which, in return, does not seem so directly linked to the $p$-curvature.
\end{remark}

\subsection{Application to algorithmics}

Given that the matrix $S^\DP(x)$ is invertible, it follows from
Lemma~\ref{lem:dpSDP} that one can deduce the value of
$B_p^\DP(x)$ from that of $S^\DP(x)$ which, in turn,
can be computed using the techniques of~\cite{BCOSSS07}.
However, at this point, we have not solved entirely the question of
the computation of the $p$-curvature for two reasons. Firstly, the
previous reasoning assumes implicitely that the coefficients $b_i(x)$'s
have no pole at $0$. Secondly, and more importantly, the knowledge of
$B_p^\DP(x)$ is not enough to recover $B_p$. Precisely, letting 
$\Fp(x)_0$ denote the subring of $\Fp(x)$ consisting of functions
with no pole at~$0$, the natural map $\delta_0 : \Fp(x)_0 \to 
\Fp[[x]]^\DP$ is not injective; its kernel is the ideal generated
by $x^p$.

To tackle these issues, the idea is to shift around any other base point
$a \in \Fp$. Doing so, we get a new differential ring homomorphism
$\delta_a : \Fp(x)_a \to \Fp[[x{-}a]]^\DP$ and, reusing the same
techniques, we end up with a fast algorithm that computes the
$p$-curvature $B_p$ modulo $(x{-}a)^p$.
Since we have moreover at our disposal \emph{a priori} bounds on the
size of the $p$-curvature (see Proposition~\ref{prop:sizepcurv}), one
can pick enough elements $a$ in $\Fp$ (or in a finite extension of $\Fp$, if
needed), compute the $p$-curvature modulo
$(x{-}a)^p$ for all those points~$a$ and reconstruct the complete
matrix $B_p(x)$ using the Chinese Remainder Theorem. Implementing this
strategy, we end up with the following theorem.

\begin{theo}[Bostan--Caruso--Schost~\cite{BoCaSc15}]
\label{theo:pcurvalgo}
There exists an algorithm that takes as input a differential
operator
\[\op L =
\partial_x^n + b_{n-1}(x) {\cdot}
\partial_x^{n-1} + \cdots + b_1(x) {\cdot}\partial_x + b_0(x)\]
over $\Fp(x)$ and outputs its $p$-curvature for a cost of
$\softO(dn^\omega p)$ operations in $\Fp$ with
\[d = \max\big(\deg f(x),\, \deg(f(x)b_0(x)),\, \ldots,\, 
  \deg(f(x) b_{n-1}(x))\big)\]
where $f(x)$ is a common denominator of the $b_i(x)$'s.
\end{theo}

Before commenting on the above result, we need to explain some
notation. Firstly, the notation $\softO(-)$ means
that we are hiding logarithmic factors.
Secondly, the exponent $\omega$ refers to what we usually call
a \emph{feasible} exponent for the matrix multiplication.
It simply means that we suppose that we are given an algorithm
that computes the product of two square matrices of size $n$
using at most $O(n^\omega)$ operations in the base field.
The naive method for multiplying matrices (the one we have all
learnt in our first course of linear algebra) indicates that 
we can take $\omega = 3$. However, it turns out that better
algorithms exist.
For example, Strassen's algorithm~\cite{Strassen69} results in
$\omega = \log_2 7 \approx 2.8$. Nowadays, the best known value
for $\omega$ is about $2.37188$ and the corresponding algorithm is
due to Duan, Wu and Zhou~\cite{DuWuZh23}.
It is a widely open conjecture if one can take $\omega = 2 +
\varepsilon$ for all $\varepsilon > 0$.

The announced complexity $\softO(d n^\omega p)$ should be compared
to the size of the output, \emph{i.e.} the number of scalars in~$\F_p$
needed to write down completely the $p$-curvature.
By Proposition~\ref{prop:sizepcurv}, $B_p$ if an $n \times n$ matrix
whose entries are rational functions with numerators and denominators
of degree at most $dp$; in practice, this bound is in general sharp.
Therefore, the size of the output is about $d n^2 p$.
As a consequence, the algorithm behind Theorem~\ref{theo:pcurvalgo}
would be quasi-optimal (\emph{i.e.} optimal up to constant and
logarithmic factors) if $\omega$ were equal to $2$. Even if this
limit cannot be attained, this comparison underlines the good
performances of the algorithm. In practice, using it makes it possible 
to compute $p$-curvatures of operators of order and degree $20$ 
in a few seconds when $p < 100$ and in about half an hour when $p = 12007$.

\subsection{Similarity class and characteristic polynomial}

We have seen previously (after Proposition~\ref{prop:sizepcurv})
that, although the size of the $p$-curvature grows linearly with
respect to~$p$, its characteristic polynomial has roughly the 
same size as the input operator $\op L$ even when $p$
gets large.
For this reason, one might hope to be able to compute the
characteristic polynomial faster than the $p$-curvature itself.

The main observation for achieving this is a refinement of
Lemma~\ref{lem:dpSDP} which asserts that $B_p^\DP(x) = B_p(x) 
\text{ mod } x^p$ is not only equal to
\[-(S^\DP(x))^{-1} \cdot \frac{d^p S^\DP(x)}{dx^p}\]
but it is further \emph{similar} to the value at $x = 0$ (\emph{i.e.}
the reduction modulo $x$) of the latter product.
On the other hand, evaluating this reduction can be done with
standard algorithmic techniques (based on a 
``baby step / giant step'' approach) in time proportional to
$\sqrt p$. Based on this, we obtain the next theorem.

\begin{theo}[Bostan--Caruso--Schost~\cite{BoCaSc16}]
\label{theo:pcurvsimalgo}
There exists an algorithm that takes as input a 
differential operator
\[\op L =
\partial_x^n + b_{n-1}(x) {\cdot}
\partial_x^{n-1} + \cdots + b_1(x) {\cdot}\partial_x + b_0(x)\]
over $\Fp(x)$
and outputs the invariant factors of its $p$-curvature for 
a cost of 
\[\softO\big(d^{\omega+\frac 3 2} n^{\omega+1} \sqrt p\big)\]
operations in $\Fp$ where $d$ is defined as in Theorem~\ref{theo:pcurvalgo}.
\end{theo}

We notice that the knowledge of the invariant factors is
finer than that of the characteristic polynomial since the
latter is the product of the formers. Furthermore, knowing
the invariant factors, one can decide whether the $p$-curvature
vanishes or not, whereas the characteristic polynomial only
gives information about its nilpotency.

On the complexity side, we notice that the cost of the
algorithm of Theorem~\ref{theo:pcurvsimalgo} is worse with
respect to the parameters~$d$ and $n$ but better with respect
to~$p$. It is then interesting for small operators but large
characteristic.

Finally, we mention that, if we are only interested by the 
characteristic polynomial of the $p$-curvature, faster algorithms 
(based on different techniques) exist.

\begin{theo}[Bostan--Caruso--Schost~\cite{BoCaSc14}]
There exists an algorithm that takes as input a 
differential operator
\[\op L =
\partial_x^n + b_{n-1}(x) {\cdot}
\partial_x^{n-1} + \cdots + b_1(x) {\cdot}\partial_x + b_0(x)\]
over $\Fp(x)$
and outputs the characteristic polynomial of its $p$-curvature 
for a cost of 
\[\softO\big((d{+}n)^{\omega} \min(d,n) \sqrt p + (d{+}n)^{\omega+1} \min(d,n)\big)\]
operations in $\Fp$ where $d$ is defined as in Theorem~\ref{theo:pcurvalgo}.
\end{theo}

In practice, this algorithm performs quite well and allows for
computing the characteristic polynomial in less than one hour for
the parameters $d = n = 20$ and $p = 120\:011$.

Recently, Pagès proved an ``average version'' of this theorem which,
roughly speaking, states that, starting with a differential operator
over $\Q(x)$, one can compute all its $p$-curvatures (up to some
given bound) in average time proportional to $\log p$.

\begin{theo}[Pagès~\cite{Pages21}]
There exists an algorithm that takes as input a 
differential operator
\[\op L =
b_n(x) \cdot \partial_x^n + b_{n-1}(x) {\cdot}
\partial_x^{n-1} + \cdots + b_1(x) {\cdot}\partial_x + b_0(x)\]
with $b_i(x) \in \Z[x]$
and outputs the characteristic polynomial of all the $p$-curvatures
of $\op L \text{ mod } p$ for $p \leq N$ for a cost of 
\[\softO\big( \big( (d{+}n)^{\omega} (d{+}m) + (d{+}n)^3\big) \cdot N d \big)\]
operations on bits, where $d$ is the maximal degree of the $b_i(x)$'s and
$m$ is the maximal bitsize of an integer appearing as the coefficient of
one of the $b_i(x)$'s.
\end{theo}

\section{Algebraicity and integrality}
\label{ssec:integrality}

The theoretical developments we carried out in \S\ref{ssec:pcurvandsol} 
have also interesting consequences in characteristic~$0$, as they allow
to relate algebraicity of solutions with the growth of denominators of
the coefficients in their Taylor expansions.
The aim of this section is to appetize the reader with some charming 
results and perspectives in this direction.

Throughout this section, we let
\[\op L =
\partial_x^n + b_{n-1}(x) {\cdot}
\partial_x^{n-1} + \cdots + b_1(x) {\cdot}\partial_x + b_0(x)\]
be a differential operator over $\Q(x)$ and assume that $b_i(x)$ has
no pole at $0$ for all $i$. We set:
\begin{equation}
\label{eq:fundamentalsol2}
S(x) = \sum_{k=0}^\infty (-1)^k B_k(x) \cdot \frac{x^k}{k!}
\end{equation}
where the $B_k(x)$'s are defined by
equation~\eqref{rec for p curv}. The matrix $S(x)$ has entries in
$\Q[[x]]$ and we write $S(x) = \sum_{i=0}^\infty S_i x^i$ where the
$S_i$ are matrices over $\Q$. We emphasize that $S_i$ is \emph{not}
equal to $\frac{(-1)^i}{i!} B_i(x)$ because the latter has in general
coefficients in $\Q(x)$.

\subsection{Growth of denominators and $p$-curvatures}

We would like to compare $S(x)$ with the matrix $S^\DP(x)$ introduced
in equation~\eqref{eq:Sdp} and, for this, to reduce everything
($\op L$, $S(x)$, \emph{etc.}) modulo a prime number $p$.
However, this operation requires some care
because the $B_k(x)$'s may exhibit denominators.
In order to do it properly, we introduce new rings. For any subring
$R \subset \Q$, we set:
\begin{align*}
R(x) & =
\left\{\, \frac{P}{Q}
\quad \text{with} \quad P, Q \in R[x]
\text{ and } Q \text{ monic}
\,\right\}, \\
R(x)_0 & =
\left\{\, \frac{P}{Q}
\quad \text{with} \quad P, Q \in R[x],\,
Q \text{ monic} \text{ and } Q(0) \neq 0
\,\right\}, \\
R[[x]]^\DP & =
\left\{\, \sum_{i=0}^\infty a_i\:\frac{x^i}{i!}
\quad \text{with} \quad a_i \in R \text{ for all } i
\,\right\} .
\end{align*}
One has the following chain of inclusions
$R[x] \subset R(x)_0 \subset R(x)$ together with an injective
morphism of rings $R(x)_0 \hookrightarrow R[[x]]$. Besides,
all these maps commute with the derivation $\frac d{dx}$ and, if
$p$ is a prime number which is noninvertible in $R$, they are
compatible with the reduction modulo $p$;
in particular, we have the following commutative diagram:
\[\xymatrix @C=4em {
R(x)_0 \ar[r] \ar[d]^-{\mod p} &
R[[x]] \ar[r] \ar[d]^-{\mod p} &
R[[x]]^\DP \ar[d]^-{\mod p} \\
\Fp(x)_0 \ar[r] & \F_p[[x]] \ar[r] & \F_p[[x]]^\DP
}\]
where all the arrows are homomorphisms of rings and commute 
with the derivation.
We now assume that the entries of $B(x)$ have no pole at~$0$ and
choose $R$ in such a way that they all belong to $R(x)_0$
(one can always take $R = \Z[\frac 1 D]$ for $D$ large enough). All
the $B_m(x)$'s then take coefficients in $R(x)_0$ as well
and the matrix $S(x)$ is defined over $R[[x]]^\DP$. Besides, the
image of $S(x)$ modulo $p$ is the matrix $S^\DP(x)$ modulo $p$
associated to the differential system $Y' + (B(x) \mod p) \cdot
Y = 0$ in characteristic~$p$.
Lemma~\ref{lem:dpSDP} then leaves us with the congruence:
\begin{equation}
\label{eq:SpAp}
S_p \equiv \frac{B_p(x)}{p!} \pmod x,
\quad \text{i.e.} \quad
S_p = \frac{B_p(0)}{p!}.
\end{equation}
Hence
the $p$-curvatures (which, we recall, are the matrices $B_p(x) 
\mod p$) are directly related to the coefficients appearing in a
fundamental system of solutions.
In particular, the vanishing of $B_p(0)$ modulo~$p$ is equivalent to
the fact that the denominator of $S_p$ is coprime with~$p$. Many
variations on this theme are possible; a beautiful example is given by
the following theorem.

\begin{theo}[see Proposition 5.3.3 in \cite{Andre04}]
\label{theo:Andre}
Let
\[\op L =
\partial_x^n + b_{n-1}(x) {\cdot}
\partial_x^{n-1} + \cdots + b_1(x) {\cdot}\partial_x + b_0(x)\]
be a differential operator over $\Q(x)$ and $D$ be a positive integer.
We assume that $\op L$ admits $n$ solutions $Y_1, \ldots, Y_n$
which have coordinates in $\Z[\frac 1 D][[x]]$ and are linearly
independent over $\Q$.
Then almost all the $p$-curvatures of $\op L$ vanish.
\end{theo}

\begin{remark}
Under Grothendieck's conjecture, Theorem~\ref{theo:Andre} can be elegantly
rephrased as follows: if a differential system admits a basis
of solutions in $\Z[\frac 1 D][[x]]$ for some positive integer $D$
(\emph{i.e.}, a basis of \emph{globally bounded} solutions),
then it also admits a basis of algebraic solutions.
This is known as \emph{Bézivin's conjecture}; it was formulated by Bézivin
in~\cite[p.~299]{Bezivin91} and proved by him for $q$-differential
equations~\cite[Thm.~7-1]{Bezivin91}; see also Conjecture 6.3 in \cite{Christol83}. 
It is widely open whether
Bézivin's conjecture is more difficult than Grothendieck's conjecture; at any
rate, it appears that for the time being the only cases for which Bézivin's
conjecture is proven are those for which Grothendieck's conjecture is known to be true.
\end{remark}

\begin{ex}
\label{ex exp-arctan3}
We illustrate Theorem~\ref{theo:Andre} with the differential equation
\[y' = \frac 1 {x^2 + 1} y\]
already considered in Example~\ref{ex exp-arctan}.
Over the rationals, the solutions are all proportional to the 
fundamental solution
\[y_0(x) = \exp(\arctan(x)) = \sum_{n=0}^\infty c_n x^n\]
where the $c_n$'s are rational numbers. We would like to find 
bounds on denominators of the $c_n$'s.
More precisely, we fix a prime number $p \neq 2$ (for simplicity)
and ask whether the denominators of the $c_n$'s are all coprime 
with~$p$. 
For this, we use the following important result due to Dwork.  
\begin{theo}[{Dwork's criterion, \cite[p.~409]{Robert00}}]
Given a prime $p$ and $f(x) \in x\Q[[x]]$, we have $\exp(f(x)) \in 1+x\Z_{(p)}[[x]]$ if and only if $f(x^p)-pf(x) \in px\Z_{(p)}[[x]]$, where $\Z_{(p)}$ is the subring of $\Q$ consisting of fractions
$\frac a b$ with $b$ coprime with $p$. 
\end{theo}
According to Dwork's criterion, the denominators of the $c_n$'s are all coprime 
with~$p$ if and only if
\[\arctan(x^p) - p \cdot \arctan(x) \in p \Z_{(p)}[[x]].\]
We have:
\begin{equation}
\label{eq arctan}
\arctan(x^p) - p \cdot \arctan(x) =
\sum_{n=0}^\infty \frac{(-1)^n}{2n+1} x^{(2n+1)p} - 
p \cdot \sum_{n=0}^\infty \frac{(-1)^n}{2n+1} x^{(2n+1)}.
\end{equation}
Clearly, when $2n+1$ is coprime with $p$, the coefficient
$p \cdot \frac{(-1)^n}{2n+1}$ is divisible by $p$. Therefore,
we can only retain in the second sum of equation~\eqref{eq arctan} 
the terms for which $2n \equiv -1 \pmod p$, \emph{i.e.} 
$2n = p - 1 + 2\ell p$ with $\ell \in \N$. We thus get:
\begin{align*}
\arctan(x^p) - p \cdot \arctan(x) 
& \equiv \sum_{n=0}^\infty \frac{(-1)^n}{2n+1} x^{(2n+1)p} - 
  \sum_{\ell=0}^\infty \frac{(-1)^{\ell - \frac{p-1}2}}{2\ell+1} x^{(2\ell+1)p} \\
& = \sum_{n=0}^\infty \frac{(-1)^n}{2n+1} \cdot \left(1 - (-1)^{\frac{p-1} 2}\right) \cdot x^{(2n+1)p}
  \pmod{p \Z_{(p)}[[x]]},
\end{align*}
hence $\arctan(x^p) - p \cdot \arctan(x)$ is divisible
by $p$ when $p \equiv 1 \mod 4$ and is not otherwise.
In conclusion,
the denominators of the $c_n$'s are all coprime 
with~$p$ (that is, $\exp(\arctan(x))$ can be reduced modulo $p$)
if and only if $p \equiv 1 \mod 4$.
\end{ex}

\begin{rem} Note that the sequence $(T_n)_{n \geq 0}$ defined by $T_n = n!
\cdot c_n$, satisfies the linear recurrence $T_{n+2} = T_{n+1} - n (n + 1) T_n$
with $T_0=T_1=1$, hence its terms are all integer numbers. Kelinsky proved in
\cite[Thm.~3]{Kelisky59} that for any prime $p\neq 2$, the term $T_p$ is
congruent to $0$ modulo~$p$ if $p \equiv 1 \mod 4$ (and to $2$ if $p \equiv 3
\mod 4$). The computation in \cref{ex exp-arctan3} 
provides a new proof of this statement: 
if $p \equiv 1 \mod 4$, then $T_n$ is congruent to $0$ 
modulo $p$ for all $n\geq p$, in other terms the generating function 
$\sum_{n\geq 0} T_n x^n$ is a \emph{polynomial} modulo~$p$.
\end{rem}

\medskip In the same orbit, we mention two other theorems which are not directly
related to our discussion but highlights other intricate relations
between algebraicity and integrality.

\begin{theo}[Conjectured by Ogus~\cite{Ogus82}, proved by André~\cite{Andre89}]
Let $f(x)\in\mathbb{Z}[[x]]$ such that $f'(x)$ is 
algebraic over $\Q(x)$.
Then $f(x)$ is  algebraic over $\Q(x)$.
\end{theo}

\begin{theo}[Conjectured by Katz~\cite{Katz72}, proved by the Chudnovsky--Chudnovsky~\cite{ChCh85}]
Let $f(x)\in\mathbb{Z}[[x]]$ such that $f'(x)/f(x)$ is 
algebraic over $\Q(x)$.
Then $f(x)$ is algebraic over $\Q(x)$.
\end{theo}

\subsection{An analytic perspective on the $p$-curvature}

All that precedes indicates that the vanishing properties of the
$p$-curvatures tend to control the growth of the denominators of the
coefficients of the fundamental system of solutions~$S(x)$. Typically,
after equation~\eqref{eq:SpAp}, we have seen that $B_p(0) \equiv 0 \pmod
p$ is equivalent to the fact that $p$ does not divide the smallest
common denominator of the entries of $S_p$.

It is convenient to reformulate this class of properties in terms
of $p$-adic valuation and $p$-adic numbers.
We recall that the $p$-adic valuation of an integer~$n$, denoted
by~$v_p(n)$, is the greatest integer $v$ such that $p^v$
divides~$n$. We then define the $p$-adic valuation of a rational number
$x = \frac a b$ by setting $v_p(x) = v_p(a) - v_p(b)$.
Having a denominator coprime with~$p$ is then equivalent to having
nonnegative $p$-adic valuation.
If $M$ is a matrix over $\Q$, we define $v_p(M)$ as the minimum of
the $p$-valuations of its entries.

Recall that, for all nonnegative integer $i$, we have:
\[v_p(i!) = \sum_{n=1}^\infty \left\lfloor \frac i {p^n}\right\rfloor \leq
\frac i{p-1}\]
(where $\lfloor \,\cdot\,\rfloor$ is the floor function). Hence, if $S(x)$ is
defined over $\Z[\frac 1 D][[x]]$ and $p$ is a prime number which does
not divide $N$, we deduce from the very first definition of $S(x)$
(see equation~\eqref{eq:fundamentalsol2}) that
$v_p(S_i) \geq -v_p(i!) \geq \frac{-i}{p{-}1}$ for all $i$.
On the other hand, the property we have recalled above indicates
that $B_p(0) \equiv 0 \pmod p$ if and only if $v_p(S_p) \geq 0$.
It turns out actually that the vanishing of the $p$-curvature
implies a general lower bound on the $p$-adic valuation of the
$S_i$'s.

\begin{prop}
\label{prop:valp}
Let
\[\op L =
\partial_x^n + b_{n-1}(x) {\cdot}
\partial_x^{n-1} + \cdots + b_1(x) {\cdot}\partial_x + b_0(x)\]
be a differential operator over $\Q(x)$
and let $S$ be the matrix defined by equation~\eqref{eq:fundamentalsol2}.
If the reduction of $\op L$ modulo p is well-defined and if the 
$p$-curvature of $\op L \text{ mod } p$ vanishes, then for all
$i \geq 0$:
\begin{equation}
\label{eq:valpSi}
v_p(S_i) \geq -v_p\big( \lfloor i/p \rfloor !\big)
\geq \frac{-i}{p(p{-}1)}.
\end{equation}
\end{prop}

Proposition~\ref{prop:valp} can be further rephrased in more
analytic terms using $p$-adic numbers. We recall briefly that
the field of $p$-adic numbers is the completion of $\Q$ for the
$p$-adic norm $\Vert \cdot \Vert_p$ defined by
$\Vert x \Vert_p = p^{-v_p(x)}$.
Set $\omega = p^{-1/(p-1)}$. Without any assumption on the
$p$-curvature, we have seen that $v_p(S_i) \geq \frac{-i}
{p-1}$, that is $\Vert S_i \Vert_p \leq \omega^{-i}$. This upper
bound indicates that the ($p$-adic) radius of convergence of the
series $S(x) = \sum_{i=0}^\infty S_i x^i$ is at least $\omega$.
On the contrary, when the $p$-curvature vanishes,
Proposition~\ref{prop:valp} tells us that $\Vert S_i \Vert_p \leq
\omega^{-i/p}$ for all $i$. Hence, the radius of convergence of $S$
is now at least $\omega^{1/p} > \omega$.
The $p$-curvature then measures some analytic properties of the
solutions of our differential system in the $p$-adic world.
A classical result in the theory of $p$-adic differential
equations~\cite[Theorem 10.4.2]{Kedlaya10},
refining the 
\emph{Frobenius antecedent theorem} of Christol and Dwork~\cite[Thm.~5.4]{ChDw94}
(see also~\cite[Theorem 6.15]{Kedlaya05}), asserts that when the
radius of convergence of a fundamental system of solutions is
strictly greater than $\omega$, the corresponding differential system
$Y'(x) + B(x) Y(x) = 0$ is equivalent, up to a base change, to a system 
of the form $Y'(x^p) + C(x^p) Y(x^p) = 0$ where the entries of $C(x)$
are $p$-adic analytic functions converging on the closed unit disk.
The differential system:
\[(\Sigma_1) : Y' + C(x)Y = 0\]
is called a \emph{Frobenius antecedent} of $(\Sigma)$.
The aforementioned convergence conditions allow for reducing
$(\Sigma_1)$ modulo $p$, thus obtaining a new differential
system on $\F_p(x)$. The latter has a well-defined $p$-curvature
and if this second $p$-curvature persists to vanish, one
eventually deduces that the radius of convergence of $S$
is at least $\omega^{1/p^2}$. When this occurs, one can continue this 
process
and find a second Frobenius antecedent $(\Sigma_2)$ of $(\Sigma)$.
If its $p$-curvature vanishes, the radius of convergence of $S$ will
be at least $\omega^{1/p^3}$ and so on and so forth.

In the perspective of the Grothendieck conjecture, we would like
to let $p$ vary and understand how the aforementioned convergence
properties interact.
A promising object, which looks capable to reflect these interactions 
is the \emph{Berkovich line over $\Z$}, which was anticipated by
Berkovich himself in~\cite{Berkovich90} and then developed by
Poineau~\cite{Poineau10}.
By definition, it is the space $\calM(\Z[x])$ consisting of
all \emph{bounded multiplicative semi-norms} $\Vert \cdot \Vert :
\Z[x] \to \R$. A semi-norm is a norm except that we authorize
nonzero elements to have norm zero. It is said multiplicative
if $\Vert f g \Vert = \Vert f \Vert \cdot \Vert g \Vert$ for
all $f, g \in \Z[x]$ and bounded when
\[\Vert a_0 + a_1 x + \cdots + a_n x^n \Vert \leq
\max\big(|a_0|, |a_1|, \ldots, |a_n|\big),\]
where $|a_i|$ denotes the usual absolute value of $a_i$.
Of course, the Berkovich line $\calM(\Z[x])$ is not only a set
but is endowed with a rich additional geometrical structure: a
topology, a structural sheaf, \emph{etc.}
Besides, after Poineau's work, we have at our disposal a whole
panel of powerful tools (inspired by modern algebraic geometry)
to work with it.

Describing entirely the space $\calM(\Z[x])$ is not obvious,
but it is not difficult to exhibit elements in it.
Take $K = \R$ or $\Q_p$ (the field of $p$-adic numbers) for
some prime number~$p$ and write $\Vert \cdot \Vert_K$ for the
standard absolute value of $K$. Pick in addition an element
$a \in K$ of norm at most $1$ and a nonnegative real number
$r$. Polynomials in $\Z[x]$ then define (real or $p$-adic)
analytic functions on the closed ball $B(a,r)$ of center $a$
and radius $r$. For $f \in \Z[x]$, we can then consider the
sup norm on this domain:
\[\Vert f \Vert_{a,r} = \sup_{x \in B(a,r)} \Vert f(x)\Vert_K.\]
One checks that it is an element of $\calM(\Z[x])$.
Moreover, at least in the $p$-adic case, the completion of $\Z[x]$ 
with respect
to this norm is the ring of $p$-adic analytic functions on
$B(a,r)$; we shall denote it by $\calA_{a,r}$ in what follows.
Another nice observation is that the notion of ``ball of center
$a$ and radius $r$'' has a well-defined meaning in the Berkovich
geometry. Indeed, let $\calM(\calA_{a,r})$ be the Berkovich space
associated to the ring $\calA_{a,r}$, \emph{i.e.} the set of
bounded multiplicative semi-norms on $\calA_{a,r}$. Restricting
a semi-norm from $\calA_{a,r}$ to $\Z[x]$ leaves us with an
injective map:
\[\calM(\calA_{a,r}) \hookrightarrow \calM(\Z[x])\]
whose image, denoted by $U_{a,r}$, is an open subset (for the
Berkovich topology) in $\calM(\Z[x])$. In addition, we observe
that any analytic function $f$ on $B(a,r)$, that is any element
$f \in \calA_{a,r}$, induces a function on $U_{a,r}$:
to each semi-norm $\Vert \cdot \Vert \in \calM(\calA_{a,r})$,
we associate $\Vert f \Vert$. For this reason, it is natural
to think at $U_{a,r}$ as the Berkovich incarnation of the ball
of center $a$ and radius~$r$.

Coming back to our topic, let us consider a differential system
$(\Sigma) : Y' + A(x) Y = 0$ over $\Z[x]$. By what we have seen
previously, for almost all prime numbers~$p$ and all $a \in \Z_p$,
the system $(\Sigma)$ admits a full basis of solutions in
$\calA_{a,\omega}$ where we recall that we have set $\omega =
p^{-1/(p-1)}$.
In the Berkovich language, these functions give rise to new
functions defined on $U_{a,\omega}$. Putting them together,
we conclude that $(\Sigma)$ always admits a basis of solutions
on a certain open subspace $U_0 \subset \calM(\Z[x])$.
Now, the assumption that almost all the $p$-curvatures vanish
shows that those solutions extend automatically to a larger
subspace $U_1 \subset \calM(\Z[x])$. On the other hand, in the
Berkovich language, proving the Grothendieck conjecture 
amounts to showing
that there exist a full basis of solution on an étale covering of
$\calM(\Z[x])$. Of course, these remarks do not give any proof of
the Grothendieck conjecture because $U_1$ itself is certainly
\emph{not} an étale covering of $\calM(\Z[x])$.
However, we think that this point of view has the potential
to lead to new interesting developments towards the Grothendieck
conjecture in the future.

\bigskip 
\thanks{{\bf Acknowledgements}. We are very grateful to Herwig Hauser for the initial idea of this survey, for his constant support along the various phases of the project, and for the marvelous workshops he organized in Lisbon these last years, from which the three authors benefited a lot (supported by the Austrian Science Fund FWF, project P-34765).
Our warm thanks go to Florian Fürnsinn, Mark van Hoeij and Sergey Yurkevich for their careful reading and helpful comments.
This work has been partially supported by the French grants CLap-CLap (ANR-18-CE40-0026) and DeRerumNatura (ANR-19-CE40-0018), and by the French--Austrian project EAGLES (ANR-22-CE91-0007 \& FWF I6130-N).}


\begin{thebibliography}{10}

\bibitem{AdRi18}
B.~Adamczewski and T.~Rivoal.
\newblock Exceptional values of {$E$}-functions at algebraic points.
\newblock {\em Bull. Lond. Math. Soc.}, 50(4):697--708, 2018.

\bibitem{AlZe90}
G.~Almkvist and D.~Zeilberger.
\newblock The method of differentiating under the integral sign.
\newblock {\em J. Symbolic Comput.}, 10(6):571--591, 1990.

\bibitem{Andre87}
Y.~Andr\'{e}.
\newblock Quatre descriptions des groupes de {G}alois diff\'{e}rentiels.
\newblock In {\em S\'{e}minaire d'alg\`ebre {P}aul {D}ubreil et {M}arie-{P}aule
  {M}alliavin ({P}aris, 1986)}, volume 1296 of {\em Lecture Notes in Math.},
  pages 28--41. Springer, Berlin, 1987.

\bibitem{Andre89}
Y.~Andr\'{e}.
\newblock {\em {$G$}-functions and geometry}.
\newblock Aspects of Mathematics, E13. Friedr. Vieweg \& Sohn, Braunschweig,
  1989.

\bibitem{Andre04}
Y.~Andr\'{e}.
\newblock Sur la conjecture des {$p$}-courbures de {G}rothendieck-{K}atz et un
  probl\`eme de {D}work.
\newblock In {\em Geometric aspects of {D}work theory. {V}ol. {I}, {II}}, pages
  55--112. Walter de Gruyter, Berlin, 2004.

\bibitem{Berkovich90}
V.~G. Berkovich.
\newblock {\em Spectral theory and analytic geometry over non-{A}rchimedean
  fields}, volume~33 of {\em Mathematical Surveys and Monographs}.
\newblock American Mathematical Society, Providence, RI, 1990.

\bibitem{Berlekamp67}
E.~R. Berlekamp.
\newblock Factoring polynomials over finite fields.
\newblock {\em Bell System Tech. J.}, 46:1853--1859, 1967.

\bibitem{BeBoRa21}
O.~Bernardi, M.~Bousquet-M\'{e}lou, and K.~Raschel.
\newblock Counting quadrant walks via {T}utte's invariant method.
\newblock {\em Comb. Theory}, 1:Paper No. 3, 77, 2021.

\bibitem{BeDuYa18}
M.~Bertola, B.~Dubrovin, and D.~Yang.
\newblock Simple {L}ie algebras and topological {ODE}s.
\newblock {\em Int. Math. Res. Not. IMRN}, (5):1368--1410, 2018.

\bibitem{BeHe89}
F.~Beukers and G.~Heckman.
\newblock Monodromy for the hypergeometric function {$_nF_{n-1}$}.
\newblock {\em Invent. Math.}, 95(2):325--354, 1989.

\bibitem{Bezivin91}
J.-P. B\'{e}zivin.
\newblock Les suites {$q$}-r\'{e}currentes lin\'{e}aires.
\newblock {\em Compositio Math.}, 80(3):285--307, 1991.

\bibitem{Bost01}
J.-B. Bost.
\newblock Algebraic leaves of algebraic foliations over number fields.
\newblock {\em Publ. Math. Inst. Hautes \'{E}tudes Sci.}, (93):161--221, 2001.

\bibitem{Bostan21}
A.~Bostan.
\newblock Computer algebra in the service of enumerative combinatorics.
\newblock In {\em I{SSAC} '21---{P}roceedings of the 2021 {I}nternational
  {S}ymposium on {S}ymbolic and {A}lgebraic {C}omputation}, pages 1--8. ACM,
  New York, [2021] \copyright 2021.

\bibitem{BoCaSc14}
A.~Bostan, X.~Caruso, and E.~Schost.
\newblock A fast algorithm for computing the characteristic polynomial of the
  $p$-curvature.
\newblock In {\em I{SSAC}'14---{P}roceedings of the 2014 {ACM} {I}nternational
  {S}ymposium on {S}ymbolic and {A}lgebraic {C}omputation}, pages 59--66. ACM,
  New York, 2014.

\bibitem{BoCaSc15}
A.~Bostan, X.~Caruso, and E.~Schost.
\newblock A fast algorithm for computing the {$p$}-curvature.
\newblock In {\em I{SSAC}'15---{P}roceedings of the 2015 {ACM} {I}nternational
  {S}ymposium on {S}ymbolic and {A}lgebraic {C}omputation}, pages 69--76. ACM,
  New York, 2015.

\bibitem{BoCaSc16}
A.~Bostan, X.~Caruso, and E.~Schost.
\newblock Computation of the similarity class of the $p$-curvature.
\newblock In {\em I{SSAC}'16---{P}roceedings of the 2016 {ACM} {I}nternational
  {S}ymposium on {S}ymbolic and {A}lgebraic {C}omputation}, pages 111--118.
  ACM, New York, 2016.

\bibitem{BCOSSS07}
A.~Bostan, F.~Chyzak, F.~Ollivier, B.~Salvy, E.~Schost, and A.~Sedoglavic.
\newblock Fast computation of power series solutions of systems of differential
  equations.
\newblock In {\em 18th ACM-SIAM Symposium on Discrete Algorithms}, pages
  1012--1021. ACM, New Orleans, 2007.

\bibitem{BoChHoKaPe17}
A.~Bostan, F.~Chyzak, M.~van Hoeij, M.~Kauers, and L.~Pech.
\newblock Hypergeometric expressions for generating functions of walks with
  small steps in the quarter plane.
\newblock {\em European J. Combin.}, 61:242--275, 2017.

\bibitem{BoKa10}
A.~Bostan and M.~Kauers.
\newblock The complete generating function for {G}essel walks is algebraic.
\newblock {\em Proc. Amer. Math. Soc.}, 138(9):3063--3078, 2010.
\newblock With an appendix by Mark van Hoeij.

\bibitem{BoKuRa17}
A.~Bostan, I.~Kurkova, and K.~Raschel.
\newblock A human proof of {G}essel's lattice path conjecture.
\newblock {\em Trans. Amer. Math. Soc.}, 369(2):1365--1393, 2017.

\bibitem{BoRiSa22}
A.~Bostan, T.~Rivoal, and B.~Salvy.
\newblock Minimization of differential equations and algebraic values of
  {$E$}-functions.
\newblock {\em Math. Comp.}, 93(347):1427--1472, 2024.

\bibitem{BoSaSi}
A.~Bostan, B.~Salvy, and M.~Singer.
\newblock On deciding transcendence of power series.
\newblock In preparation, 2023.

\bibitem{BoSc09}
A.~Bostan and E.~Schost.
\newblock Fast algorithms for differential equations in positive
  characteristic.
\newblock In {\em I{SSAC} 2009---{P}roceedings of the 2009 {I}nternational
  {S}ymposium on {S}ymbolic and {A}lgebraic {C}omputation}, pages 47--54. ACM,
  New York, 2009.

\bibitem{BoYu22}
A.~Bostan and S.~Yurkevich.
\newblock On a class of hypergeometric diagonals.
\newblock {\em Proc. Amer. Math. Soc.}, 150(3):1071--1087, 2022.

\bibitem{Bousquet06}
M.~Bousquet-M\'{e}lou.
\newblock Rational and algebraic series in combinatorial enumeration.
\newblock In {\em International {C}ongress of {M}athematicians. {V}ol. {III}},
  pages 789--826. Eur. Math. Soc., Z\"{u}rich, 2006.

\bibitem{Bousquet16}
M.~Bousquet-M\'{e}lou.
\newblock An elementary solution of {G}essel's walks in the quadrant.
\newblock {\em Adv. Math.}, 303:1171--1189, 2016.

\bibitem{BMM10}
M.~Bousquet-M\'{e}lou and M.~Mishna.
\newblock Walks with small steps in the quarter plane.
\newblock In {\em Algorithmic probability and combinatorics}, volume 520 of
  {\em Contemp. Math.}, pages 1--39. Amer. Math. Soc., Providence, RI, 2010.

\bibitem{Budd20}
T.~Budd.
\newblock Winding of simple walks on the square lattice.
\newblock {\em J. Combin. Theory Ser. A}, 172:105191, 59, 2020.

\bibitem{Chambert-Loir02}
A.~Chambert-Loir.
\newblock Th\'{e}or\`emes d'alg\'{e}bricit\'{e} en g\'{e}om\'{e}trie
  diophantienne (d'apr\`es {J}.-{B}. {B}ost, {Y}. {A}ndr\'{e}, {D}. \& {G}.
  {C}hudnovsky).
\newblock Number 282, Exp. No. 886, viii, pages 175--209. 2002.
\newblock S\'{e}minaire Bourbaki, Vol. 2000/2001.

\bibitem{Christol83}
G.~Christol.
\newblock Solutions alg\'{e}briques des \'{e}quations diff\'{e}rentielles
  {$p$}-adiques.
\newblock In {\em Seminar on number theory, {P}aris 1981--82 ({P}aris,
  1981/1982)}, volume~38 of {\em Progr. Math.}, pages 51--58. Birkh\"{a}user
  Boston, Boston, MA, 1983.

\bibitem{Christol86}
G.~Christol.
\newblock Fonctions hyperg\'eom\'etriques born\'ees.
\newblock {\em Groupe de travail d'analyse ultram\'etrique}, 14, 1986-1987.
\newblock Talk:8.

\bibitem{Christol88}
G.~Christol.
\newblock Diagonales de fractions rationnelles.
\newblock In {\em S\'{e}minaire de {T}h\'{e}orie des {N}ombres, {P}aris
  1986--87}, volume~75 of {\em Progr. Math.}, pages 65--90. Birkh\"{a}user
  Boston, Boston, MA, 1988.

\bibitem{Christol90}
G.~Christol.
\newblock Globally bounded solutions of differential equations.
\newblock In {\em Analytic number theory ({T}okyo, 1988)}, volume 1434 of {\em
  Lecture Notes in Math.}, pages 45--64. Springer, Berlin, 1990.

\bibitem{ChDw94}
G.~Christol and B.~Dwork.
\newblock Modules diff\'{e}rentiels sur des couronnes.
\newblock {\em Ann. Inst. Fourier (Grenoble)}, 44(3):663--701, 1994.

\bibitem{ChCh85}
D.~V. Chudnovsky and G.~V. Chudnovsky.
\newblock Applications of {P}ad\'{e} approximations to the {G}rothendieck
  conjecture on linear differential equations.
\newblock In {\em Number theory ({N}ew {Y}ork, 1983--84)}, volume 1135 of {\em
  Lecture Notes in Math.}, pages 52--100. Springer, Berlin, 1985.

\bibitem{CSTU02}
O.~Cormier, M.~F. Singer, B.~M. Trager, and F.~Ulmer.
\newblock Linear differential operators for polynomial equations.
\newblock {\em J. Symbolic Comput.}, 34(5):355--398, 2002.

\bibitem{DiVizio02}
L.~Di~Vizio.
\newblock Arithmetic theory of {$q$}-difference equations: the {$q$}-analogue
  of {G}rothendieck-{K}atz's conjecture on {$p$}-curvatures.
\newblock {\em Invent. Math.}, 150(3), 2002.

\bibitem{DHRS18}
T.~Dreyfus, C.~Hardouin, J.~Roques, and M.~F. Singer.
\newblock On the nature of the generating series of walks in the quarter plane.
\newblock {\em Invent. Math.}, 213(1):139--203, 2018.

\bibitem{DuWuZh23}
R.~Duan, H.~Wu, and R.~Zhou.
\newblock Faster matrix multiplication via asymmetric hashing.
\newblock In {\em 2023 IEEE 64th Annual Symposium on Foundations of Computer
  Science (FOCS)}, pages 2129--2138, Los Alamitos, CA, USA, nov 2023. IEEE
  Computer Society.

\bibitem{DuYaZa23}
B.~Dubrovin, D.~Yang, and D.~Zagier.
\newblock Geometry and arithmetic of integrable hierarchies of {K}d{V} type.
  {I}. {I}ntegrality.
\newblock {\em Adv. Math.}, 433:Paper No. 109311, 73, 2023.

\bibitem{Dwork90}
B.~Dwork.
\newblock Differential operators with nilpotent {$p$}-curvature.
\newblock {\em Amer. J. Math.}, 112(5):749--786, 1990.

\bibitem{Errera1913}
A.~Errera.
\newblock Zahlentheoretische {L}{\"o}sung einer functionentheoretischen
  {F}rage.
\newblock {\em Rend. Circ. Mat. Palermo}, 35:107--144, 1913.

\bibitem{Euler1733}
L.~Euler.
\newblock Specimen de constructione aequationum differentialium sine
  indeterminatarum separatione.
\newblock {\em Commentarii academiae scientiarum Petropolitanae}, 6:168--174,
  1733.

\bibitem{FaKi09}
B.~Farb and M.~Kisin.
\newblock Rigidity, locally symmetric varieties, and the {G}rothendieck-{K}atz
  conjecture.
\newblock {\em Int. Math. Res. Not. IMRN}, (22):4159--4167, 2009.

\bibitem{FIM99}
G.~Fayolle, R.~Iasnogorodski, and V.~Malyshev.
\newblock {\em Random walks in the quarter-plane}, volume~40 of {\em
  Applications of Mathematics (New York)}.
\newblock Springer-Verlag, Berlin, 1999.
\newblock Algebraic methods, boundary value problems and applications.

\bibitem{Flajolet87}
P.~Flajolet.
\newblock Analytic models and ambiguity of context-free languages.
\newblock {\em Theoret. Comput. Sci.}, 49(2-3):283--309, 1987.
\newblock Twelfth international colloquium on automata, languages and
  programming (Nafplion, 1985).

\bibitem{FlSe09}
P.~Flajolet and R.~Sedgewick.
\newblock {\em Analytic combinatorics}.
\newblock Cambridge University Press, Cambridge, 2009.

\bibitem{Foucault92}
F.~Foucault.
\newblock \'{E}quations de {P}icard-{F}uchs et invariants des courbes de genre
  {$2$}.
\newblock {\em C. R. Acad. Sci. Paris S\'{e}r. I Math.}, 314(8):617--619, 1992.

\bibitem{FoTo07}
F.~Foucault and P.~Toffin.
\newblock Courbes hyperelliptiques de genre trois et application de
  {K}odaira-{S}pencer.
\newblock {\em C. R. Math. Acad. Sci. Paris}, 345(12):685--687, 2007.

\bibitem{Furstenberg67}
H.~Furstenberg.
\newblock Algebraic functions over finite fields.
\newblock {\em J. Algebra}, 7:271--277, 1967.

\bibitem{FuHa23}
F.~Fürnsinn and H.~Hauser.
\newblock Fuchs' theorem on linear differential equations in arbitrary
  characteristic, 2023.
\newblock 40 pages, arXiv: \href{https://arxiv.org/abs/2307.01712}{2307.01712}.

\bibitem{FuYu23}
F.~Fürnsinn and S.~Yurkevich.
\newblock Algebraicity of hypergeometric functions with arbitrary parameters.
\newblock {\em Bull. Lond. Math. Soc.}, 2024.
\newblock 23 pages, in print, \url{https://doi.org/10.1112/blms.13103}.

\bibitem{Honda79}
T.~Honda.
\newblock Algebraic differential equations.
\newblock In {\em Symposia {M}athematica, {V}ol. {XXIV} ({S}ympos., {INDAM},
  {R}ome, 1979)}, pages 169--204. Academic Press, London-New York, 1981.

\bibitem{Jacobson37}
N.~Jacobson.
\newblock Abstract derivation and {L}ie algebras.
\newblock {\em Trans. Amer. Math. Soc.}, 42:206--224, 1937.

\bibitem{Katz72}
N.~M. Katz.
\newblock Algebraic solutions of differential equations ({$p$}-curvature and
  the {H}odge filtration).
\newblock {\em Invent. Math.}, 18:1--118, 1972.

\bibitem{Katz82}
N.~M. Katz.
\newblock A conjecture in the arithmetic theory of differential equations.
\newblock {\em Bull. Soc. Math. France}, 110(2):203--239, 1982.

\bibitem{Katz90}
N.~M. Katz.
\newblock {\em Exponential sums and differential equations}, volume 124 of {\em
  Annals of Mathematics Studies}.
\newblock Princeton University Press, Princeton, NJ, 1990.

\bibitem{KaKoZe09}
M.~Kauers, C.~Koutschan, and D.~Zeilberger.
\newblock Proof of {I}ra {G}essel's lattice path conjecture.
\newblock {\em Proc. Natl. Acad. Sci. USA}, 106(28):11502--11505, 2009.

\bibitem{Kedlaya05}
K.~S. Kedlaya.
\newblock Local monodromy of {$p$}-adic differential equations: an overview.
\newblock {\em Int. J. Number Theory}, 1(1):109--154, 2005.

\bibitem{Kedlaya10}
K.~S. Kedlaya.
\newblock {\em {$p$}-adic differential equations}, volume 125 of {\em Cambridge
  Studies in Advanced Mathematics}.
\newblock Cambridge University Press, Cambridge, 2010.

\bibitem{Kelisky59}
R.~Kelisky.
\newblock The numbers generated by {${\rm exp}({\rm arctan}$} {$x)$}.
\newblock {\em Duke Math. J.}, 26:569--581, 1959.

\bibitem{Kovacic86}
J.~J. Kovacic.
\newblock An algorithm for solving second order linear homogeneous differential
  equations.
\newblock {\em J. Symbolic Comput.}, 2(1):3--43, 1986.

\bibitem{LaOd77}
J.~C. Lagarias and A.~M. Odlyzko.
\newblock Effective versions of the {C}hebotarev density theorem.
\newblock In {\em Algebraic number fields: {$L$}-functions and {G}alois
  properties ({P}roc. {S}ympos., {U}niv. {D}urham, {D}urham, 1975)}, pages
  409--464, 1977.

\bibitem{Landau1904}
E.~Landau.
\newblock Eine {A}nwendung des {E}isensteinschen {S}atzes auf die {T}heorie der
  {G}aussschen {D}ifferentialgleichung.
\newblock {\em J. Reine Angew. Math.}, 127:92--102, 1904.

\bibitem{Landau1911}
E.~Landau.
\newblock Über einen zahlentheoretischen {S}atz und seine {A}nwendung auf die
  hypergeometrische {R}eihe.
\newblock {\em Sitzungsber. Heidelb. Akad. Wiss. Math.-Natur. Kl.}, 18:3--38,
  1911.

\bibitem{Lipshitz88}
L.~Lipshitz.
\newblock The diagonal of a {$D$}-finite power series is {$D$}-finite.
\newblock {\em J. Algebra}, 113(2):373--378, 1988.

\bibitem{Melczer21}
S.~Melczer.
\newblock {\em Algorithmic and symbolic combinatorics---an invitation to
  analytic combinatorics in several variables}.
\newblock Texts and Monographs in Symbolic Computation. Springer, Cham, [2021]
  \copyright 2021.
\newblock With a foreword by Robin Pemantle and Mark Wilson.

\bibitem{Niederreiter93}
H.~Niederreiter.
\newblock A new efficient factorization algorithm for polynomials over small
  finite fields.
\newblock {\em Appl. Algebra Engrg. Comm. Comput.}, 4(2):81--87, 1993.

\bibitem{Ogus82}
A.~Ogus.
\newblock Hodge cycles and crystalline cohomology.
\newblock In {\em Hodge cycles, motives, and {S}himura varieties, LNM 900},
  pages 357--414. Springer-Verlag, 1982.

\bibitem{Pages21}
R.~Pagès.
\newblock Computing characteristic polynomials of $p$-curvatures in average
  polynomial time.
\newblock In {\em I{SSAC}'21---{P}roceedings of the 2021 {ACM} {I}nternational
  {S}ymposium on {S}ymbolic and {A}lgebraic {C}omputation}. ACM, New York,
  2021.

\bibitem{PaShWh21}
A.~Patel, A.~N. Shankar, and J.~P. Whang.
\newblock The rank two {$p$}-curvature conjecture on generic curves.
\newblock {\em Adv. Math.}, 386:Paper No. 107800, 33, 2021.

\bibitem{Poineau10}
J.~Poineau.
\newblock La droite de {B}erkovich sur $\mathbb{Z}$.
\newblock In {\em Astérisque}, volume 333. Soc. Math. France, 2010.

\bibitem{Polya22}
G.~P\'olya.
\newblock Sur les s\'eries enti\`eres, dont la somme est une fonction
  alg\'ebrique.
\newblock {\em Enseign. Math.}, 22:38--47, 1921/1922.

\bibitem{Poole}
E.~G.~C. Poole.
\newblock {\em Introduction to the theory of linear differential equations}.
\newblock Dover Publications, Inc., New York, 1960.

\bibitem{Robert00}
A.~M. Robert.
\newblock {\em A course in {$p$}-adic analysis}, volume 198 of {\em Graduate
  Texts in Mathematics}.
\newblock Springer-Verlag, New York, 2000.

\bibitem{Schwarz1873}
H.~A. Schwarz.
\newblock Ueber diejenigen {F}\"{a}lle, in welchen die {G}aussische
  hypergeometrische {R}eihe eine algebraische {F}unction ihres vierten
  {E}lementes darstellt.
\newblock {\em J. Reine Angew. Math.}, 75:292--335, 1873.

\bibitem{Serre81}
J.-P. Serre.
\newblock Quelques applications du th\'{e}or\`eme de densit\'{e} de
  {C}hebotarev.
\newblock {\em Inst. Hautes \'{E}tudes Sci. Publ. Math.}, (54):323--401, 1981.

\bibitem{Serre03}
J.-P. Serre.
\newblock On a theorem of {J}ordan.
\newblock {\em Bull. Amer. Math. Soc. (N.S.)}, 40(4):429--440, 2003.

\bibitem{Shankar18}
A.~N. Shankar.
\newblock The {$p$}-curvature conjecture and monodromy around simple closed
  loops.
\newblock {\em Duke Math. J.}, 167(10):1951--1980, 2018.

\bibitem{Singer79}
M.~F. Singer.
\newblock Algebraic solutions of {$n$}th order linear differential equations.
\newblock In {\em Proceedings of the {Q}ueen's {N}umber {T}heory {C}onference,
  1979 ({K}ingston, {O}nt., 1979)}, volume~54 of {\em Queen's Papers in Pure
  and Appl. Math.}, pages 379--420. Queen's Univ., Kingston, ON, 1980.

\bibitem{Strassen69}
V.~Strassen.
\newblock Gaussian elimination is not optimal.
\newblock {\em Numerische mathematik}, 13:354--356, 1969.

\bibitem{Stridsberg1910}
E.~Stridsberg.
\newblock Sur le th\'eor\`eme d'{E}isenstein et l'\'equation diff\'erentielle
  de {G}auss.
\newblock {\em {Ark. Mat. Astron. Fys.}}, 6(35):1--17, 1911.

\bibitem{Tang18}
Y.~Tang.
\newblock Algebraic solutions of differential equations over {$\Bbb
  P^1-\{0,1,\infty\}$}.
\newblock {\em Int. J. Number Theory}, 14(5):1427--1457, 2018.

\bibitem{Put96}
M.~van~der Put.
\newblock Reduction modulo {$p$} of differential equations.
\newblock {\em Indag. Math. (N.S.)}, 7(3):367--387, 1996.

\bibitem{Put01}
M.~van~der Put.
\newblock Grothendieck's conjecture for the {R}isch equation {$y'=ay+b$}.
\newblock {\em Indag. Math. (N.S.)}, 12(1):113--124, 2001.

\bibitem{PutSinger03}
M.~van~der Put and M.~F. Singer.
\newblock {\em Galois theory of linear differential equations}, volume 328 of
  {\em Grundlehren der Mathematischen Wissenschaften [Fundamental Principles of
  Mathematical Sciences]}.
\newblock Springer-Verlag, Berlin, 2003.

\bibitem{Vargas21}
D.~Vargas-Montoya.
\newblock Alg\'{e}bricit\'{e} modulo {$p$}, s\'{e}ries
  hyperg\'{e}om\'{e}triques et structures de {F}robenius fortes.
\newblock {\em Bull. Soc. Math. France}, 149(3):439--477, 2021.

\bibitem{Yurkevich22}
S.~Yurkevich.
\newblock The art of algorithmic guessing in \textsf{gfun}.
\newblock {\em Maple Trans.}, 2(1):14421:1--14421:19, 2022.

\bibitem{Zagier18}
D.~Zagier.
\newblock The arithmetic and topology of differential equations.
\newblock In {\em European {C}ongress of {M}athematics}, pages 717--776. Eur.
  Math. Soc., Z\"{u}rich, 2018.

\end{thebibliography}
\end{document}